\documentclass[aos,preprint]{imsart} 

\RequirePackage[OT1]{fontenc}
\RequirePackage{amsthm,amssymb,amsfonts,amsmath}
\RequirePackage[numbers]{natbib}
\RequirePackage[colorlinks,citecolor=blue,urlcolor=blue, pdffitwindow=false, pdfstartview={FitH}]{hyperref}
\usepackage{epsfig}
\usepackage{stmaryrd}
\usepackage{graphics}
\usepackage{subfigure}
\usepackage[T1]{fontenc}
\usepackage{graphicx}
\usepackage{epstopdf}
\usepackage{hyperref}
\usepackage{pdfsync}
\usepackage{color}
\usepackage{dsfont}
\newcommand{\N}{\ensuremath{\mathbb N} }
\newcommand{\R}{\ensuremath{\mathbb R} }
\newcommand{\C}{\ensuremath{\mathbb C} }
\newcommand{\Z}{\ensuremath{\mathbb Z} }

\newcommand{\PR}{\ensuremath{{\mathbb P}} }
\newcommand{\PP}{\ensuremath{{\mathbb P}} }
\newcommand{\EE}{\ensuremath{{\mathbb E}} }

\newcommand{\cA}{{\cal A}}
\newcommand{\cB}{{\cal B}}
\newcommand{\cC}{{\cal C}}

\newcommand{\cE}{{\cal E}}
\newcommand{\cF}{{\cal F}}
\newcommand{\cG}{{\cal G}}
\newcommand{\cH}{{\cal H}}

\newcommand{\cL}{{\cal L}}
\newcommand{\cM}{{\cal M}}
\newcommand{\cN}{{\cal N}}
\newcommand{\cO}{{\cal O}}
\newcommand{\cP}{{\cal P}}

\newcommand{\cV}{{\cal V}}

\newtheorem{theo}{Theorem}
\newtheorem{proposition}{Proposition}
\newtheorem{lemma}{Lemma}

\newtheorem{remark}{Remark}
\newtheorem{coro}{Corollary}

\numberwithin{equation}{section}
\numberwithin{ass}{section}
\numberwithin{theo}{section} \numberwithin{proposition}{section}
\numberwithin{lemma}{section}
\numberwithin{remark}{section}

\newcommand{\re}{\Re\mathrm{e}} 
\def\M{\mathfrak{M}([0,1])}
\def\Mnu{\mathfrak{M}_{\nu}([0,1])(A)}
\def \i{\mathfrak{i}}
\def \xikn{\xi_{n}} 

\newcommand{\1}{\ensuremath{\textbf{1}}}

\newcommand{\underbraceabs}[2]{\left|\vphantom{#1}\right. \underbrace{#1}_{#2} \left.\vphantom{#1}\right| }

\arxiv{arXiv:0000.0000}

\startlocaldefs
\numberwithin{equation}{section}
\theoremstyle{plain}

\endlocaldefs

\begin{document}

\begin{frontmatter}
\title{Bayesian posterior consistency and contraction rates in the Shape Invariant Model}
\runtitle{Bayesian posterior consistency and rates in the SIM}

\begin{aug}
\author{\fnms{Dominique} \snm{Bontemps}\thanksref{t2}\ead[label=e1]{dominique.bontemps@math.univ-toulouse.fr}} \and
\author{\fnms{S\'ebastien} \snm{Gadat}\thanksref{t2}\ead[label=e2]{sebastien.gadat@math.univ-toulouse.fr}}

\thankstext{t2}{The authors acknowledge the support of the French Agence Nationale de la Recherche (ANR) under references ANR-JCJC-SIMI1 DEMOS and ANR Bandhits.}
\runauthor{D. Bontemps and S. Gadat}

\affiliation{Institut Math\'ematiques de Toulouse, Universit\'e Paul Sabatier}

\address{Institut Math\'ematiques de Toulouse, 
 Universit\'e Paul Sabatier\\118 route de Narbonne
F-31062 Toulouse Cedex 9 FRANCE\\
\printead{e1}\\\printead{e2}\\}

\end{aug}

\begin{abstract}

In this paper, we consider the so-called Shape Invariant Model which stands for the estimation of a function $f^0$ submitted to a random translation of law $g^0$ in a white noise model. We are interested in such a model when the law of the deformations is {\em unknown}. 
We aim to recover the law of the process $\PP_{f^0,g^0}$ as well as $f^0$ and $g^0$. 

In this perspective, we adopt a Bayesian point of view and find prior on $f$ and $g$ such that the posterior distribution  concentrates around $\PP_{f^0,g^0}$  at a polynomial rate when $n$ goes to $+\infty$. We obtain a  logarithmic posterior contraction rate for the shape $f^0$ and the distribution $g^0$. We also derive 
 logarithmic lower bounds for the estimation of $f^0$ and $g^0$ in a frequentist paradigm.
\end{abstract}



\end{frontmatter}

\section{Introduction}

We are interested in this work in the so-called Shape Invariant Model (SIM). Such model aims to describe a statistical process which involves a deformation of a functional shape according to some randomized geometric variability. 
Such geometric deformation of a common unknown shape may be well-suited in various and numerous fields, like image processing (see for instance \cite{amitpiccgre} or \cite{Park}). It corresponds to a particular case of  the general Grenander's theory of shapes (see \cite{gremil} for a detailed introduction on this topic). 
This kind of model is also useful in medicine:  the recent work of \cite{Bigotecg} deals with the differentiation between normal and arrhythmic cycles in electrocardiogram. 
It appears in genetics if one deals with some delayed activation curves of genes when drugs are administrated to patients, or in Chip-Seq estimation when translations in protein fixation yield randomly shifted counting processes (see for instance \cite{Mortazaviscience} and \cite{BGKM12}). It also occurs in econometric for the analysis of Engel curves \cite{Blundell}, in landmark registration \cite{bigland}\ldots

Such a model has received a large interest in the statistical community as pointed by the large amount of references on this subject. Some works  consider a semi-parametric approach for the estimation ({\it self-modeling regression framework} used by \cite{kg} and \cite{BGV}). 
In \cite{Cast12}, the author applies some Bayesian techniques to obtain also statistical results on SIM in a semi-parametric setting when the level of noise on observations asymptotically vanishes. 
Older approaches use parametric settings (see \cite{glasbey} and the discussion therein for an overview) and study the so-called Fr\'echet mean of pattern. 
Standard $M$-estimation or Bayesian methods are exploited in \cite{BGL09} or \cite{AAT} and same authors develop in \cite{AKT} a nice stochastic algorithm to run estimation in such a model. 
Some recent works follow some testing strategies to obtain curve registration \cite{CD11}, \cite{C12}. At last, note that \cite{BG10} obtains some minimax adaptive results for non-parametric estimations in the Shape Invariant Model when one knows the law of the randomized translations. 

All these works are interested in the statistical process of deformation of the "mean common shape" and generally aim to recover this unknown functional object according to noisy i.i.d. observations. Moreover, the Shape Invariant Model is considered as a standard benchmark for statistical methods which aim to compute estimations in some more general deformable models. 
Of course, the SIM could be extended to some more general situations of geometrical deformations described through an action of a finite dimensional Lie Group (see \cite{BCG} for a precise non parametric description). We have decided to restrict our work here to the simplest case of the one dimensional Lie group of translation $\mathbb{S}^1$ to warp the functional objects.

This work has been inspired by several discussions with Alain Trouv\'e about the work \cite{AKT} for the study of the Shape Invariant Model. We aim to extend their parametric Bayesian framework to the non-parametric setting and then study the behaviour of some posterior distributions. 
Hence,  the motivation of the paper is mainly theoretical: we want to describe the asymptotic evolution of the posterior probability distributions when data are coming from the SIM. Of course, we need to build  suitable prior which yield nice contraction rate for this posterior distribution. 
We have decided to consider the general case where both the functional shape and the probability distribution of the  deformations are unknown. Indeed, it corresponds to the more realistic case. From the best of our knowledge, no sharp statistical results have been derived yet in this non-parametric situation.

Our work will describe the evolution of the posterior distribution when the number of observations grows to $+ \infty$ with a fixed noise level $\sigma$. It is an important difference with the study of the asymptotically vanishing noise situation ($\sigma \rightarrow 0$). It is itself a special feature of the Shape Invariant Model: there is no obvious Le Cam equivalence of experiments (see \cite{LY}) for the SIM between the experiments when $n \rightarrow + \infty$ and when $\sigma \mapsto 0$.  It is illustrated by the very different minimax results obtained in \cite{BG10} ($n \rightarrow + \infty$) and in \cite{BG12} ($\sigma \mapsto 0$).
We will use in the sequel quite standard Bayesian non parametric methods to obtain the frequentist consistency and some contraction rates of the Bayesian procedures. Such tools rely on some important contributions of \cite{BSW} and \cite{GGvdW00} for the posterior behaviour in general situations, as well as Bayesian properties on mixture models  stated in \cite{GvdW01} and \cite{GW} or prior distribution on smooth densities through Gaussian processes given by \cite{vdWvZ}.

The paper is organised as follows. Section \ref{sec:model_not_result} presents a sharp description of the Shape Invariant Model (shortened as SIM in the sequel), as well as standard elements on Bayesian and Fourier analysis. It also provides some notations for mixture models. It ends with the statements of our main results.
Section \ref{sec:proof_main} presents the proof of the posterior contraction around the true law on functional curves, which is our first main result. Section \ref{sec:identifiability} provides some general identifiability results and up to these identifiability conditions, shows the posterior contraction on the functional objects themselves. At last, this section also establishes a lower bound result of reconstruction in a frequentist paradigm. We end the paper with numerous challenging issues. 

We gather in the appendix sections some technical points: the metric description of the Shape Invariant Model embedded in a special randomized curves space  and the calibration of suitable priors for the SIM.

\section{Model, notations and main results} \label{sec:model_not_result}

\subsection{Statistical settings}

\paragraph{Shape Invariant Model}
We recall here the random Shape Invariant Model. We assume $f^0$ to be  a  function  which belongs to a subset $\cF$ of smooth functions. 
We also consider a probability measure $g^0$ which is an element of the set $\mathfrak{M}([0,1])$. This last set stands for the set of probability measures on $[0,1]$. We observe $n$ realizations of  noisy and randomly shifted complex valued curves $Y_1, \ldots, Y_n$ coming from the following white noise model
\begin{equation}\label{eq:model}
\forall x \in [0,1] \quad \forall j=1 \ldots n \qquad dY_j(x): = f^0(x-\tau_j) dx + \sigma dW_j(x).
\end{equation}
Here, $f^0$ is the \textit{mean} pattern of the curves $Y_1, \ldots, Y_n$ although the random shifts $(\tau_j)_{j = 1 \ldots n}$ are sampled independently according to the probability measure $g^0$. 
Moreover, $(W_j)_{j = 1 \ldots n}$ are independent complex standard Brownian motions on $[0,1]$ and model the presence of noise in the observations, the noise level is kept fixed in our study and is set to $1$ for sake of simplicity.

In the sequel, $f^{-\tau}$ will denote the pattern $f$ shifted by $\tau$, that is to say the function $x\mapsto f(x-\tau)$. 
Complex valued curves are considered here for the simplicity of notations. However all our results can be adapted to the simpler case where all curves $Y_j$'s are real valued. 
A complex standard Brownian motion $W_t$ on $[0,1]$ is such that $W_1$ is a standard complex Gaussian random variable, whose distribution is denoted by $\cN_{\C}(0,1)$; a standard complex Gaussian random variable have independent real and imaginary parts with a real centered Gaussian distribution of variance $1/2$.

This work will address the question of the behaviour of some posterior distributions on $\cF\otimes\mathfrak{M}([0,1])$ given some functional  $n$-sample $(Y_1, \ldots, Y_n)$.
Since our work will be mainly asymptotic with $n \rightarrow + \infty$, we  intensively use some standard notation  such as "$\lesssim$" which refers to an inequality up to a multiplicative absolute constant. In the meantime, $a\sim b$  stands for $a/b \longrightarrow 1$.

\paragraph{Bayesian framework}

Since most of statistical works on the SIM are frequentists, we have decided to briefly recall here the Bayesian formalism following the presentation of \cite{GGvdW00}. Familiar readers can thus omit this paragraph.

 Functional objects $f^0$ and $g^0$ we are looking for, belong to $\cF \otimes \M$ and for any couple $(f,g) \in \cF \otimes \M$,  equation \eqref{eq:model} describes the law of one continuous curve. Its law is denoted $\PP_{f,g}$ and possesses a density  $p_{f,g}$ with respect to the Wiener measure on the sample space. 
 Since $f^0$ and $g^0$ are unknown, $\PP_{f^0,g^0}$ is also unavailable but belongs to a set $\cP$ of probability measure over the sample space. This set $\cP$ is the set of all possible measures described by \eqref{eq:model} when $(f,g)$ varies into $\cF \otimes \M$.

 Given some prior distribution $\Pi_n$ on $\cP$ (generally defined through a prior on $\cF\otimes \M$), Bayesian procedures are generally built using the posterior distribution defined by
 $$
 \Pi_n\left(B|Y_1,\ldots, Y_n\right)  = \frac{\int_{B} \prod_{j=1}^n p(Y_j) d\Pi_n(p)}{\int_{\cP} \prod_{j=1}^n p(Y_j) d\Pi_n(p)},
 $$
 which is a random measure on $\cP$ that depends on the observations $Y_1, \ldots, Y_n $. For instance, Bayesian estimators can be obtained using the mode, the mean or the median of the posterior distribution. 
This is exactly the approach adopted by \cite{AKT} which is mainly dedicated to compute such a posterior mean in a parametric setting with a stochastic EM algorithm.
 
 The posterior distribution  is then said \emph{consistent} if it concentrates to arbitrarily small neighbourhoods  of $\PP_{f^0,g^0}$ in $\cP$ with a probability tending to $1$ when $n$ grows to $+\infty$. 
One \emph{frequentist} property of such a posterior distribution describes the contraction rate of such neighbourhoods meanwhile still capturing most of posterior mass.
According to equation~\eqref{eq:model}, we thus tackle such a Bayesian consistency and compute such convergence rates in the frequentist paradigm. Of course, these properties will highly depend on the metric structure of the sets $\cP$ and $\cF$.

\paragraph{Functional setting and Fourier analysis}
Without loss of generality, the function $f^0$ is assumed to be periodic with period $1$ and to belong to a subset $\cF$ of $L^2_\C([0,1])$, the space of squared integrable functions on $[0,1]$ endowed with the euclidean norm $\|h\|:=\int_0^1 |h(s)|^2 ds$. 
Moreover, each element $h\in L^2_\C([0,1])$ may naturally be extended to a periodic function on $\R$ of period $1$. Since we will intensively use some Fourier analysis in the sequel, let us first recall some notations: $\i$ will stand for the complex number such that $\mathfrak i^2=-1$. The Fourier coefficients of $h$ are denoted
\begin{equation}\label{eq:fourier}
\theta_{\ell}(h) := \int_{0}^1 e^{- \i 2 \pi  \ell t} h(t) dt.
\end{equation}
All along the paper, we will often use the parametrisation of any element of $h \in L_\C^2([0,1])$ through its Fourier expansion and  will simply use the notation $(\theta_{\ell})_{\ell \in \Z}$ instead of $(\theta_{\ell}(h))_{\ell \in \Z}$.

Our work is dedicated to the analysis of SIM when  $\cF$  models smooth functions of $[0,1]$. Hence, natural subspaces of $L^2_\C([0,1])$ are Sobolev spaces $\cH_s$ with a smoothness parameter $s$:
$$
\cH_s
:=\left\{f\in L^2_\C([0,1]) \quad \vert \quad \sum_{\ell \in \Z} (1+|\ell|^{2s}) |\theta_\ell(f)|^2 
< + \infty \right\}.
$$
A useful set of functions for the identifiability part will also be the following restriction of $\cH_s$:
$$
\cF_s
:=\left\{f\in L^2_\C([0,1]) \quad \vert \quad \theta_1(f)>0 \quad  \text{and} \quad \sum_{\ell \in \Z} (1+|\ell|^{2s}) |\theta_\ell(f)|^2 
< + \infty  \right\}.
$$
In the sequel,  we aim to find prior on $\cP$ that reaches good frequentist properties, and if possible adaptive with the smoothness parameter $s$ since this parameter is generally unknown.
We will consider only some regular cases when $s \geq 1$, the quantity $\sum_{\ell} \ell^2 |\theta_\ell|^2$ is thus bounded and we denote the Sobolev norm
$$
\|\theta\|_{\cH_1} := \sqrt{ \sum_{\ell\in \Z} \ell^2 |\theta_\ell|^2}.
$$
It will also be useful to consider in some cases Fourier "thresholded" elements of $\cH_s$. Hence, we set for any integer $\ell$ (which is the frequency threshold)
$$
\cH^\ell := \left\{f \in L^2_\C([0,1])
\quad \vert \quad \forall |k|>\ell \quad \theta_k(f) = 0 \right\}.
$$

\paragraph{Mixture  model}
According to equation \eqref{eq:model}, we can write in the Fourier domain that
\begin{equation}\label{eq:model_fourier}
\forall \ell \in \Z \quad \forall j \in \{1 \ldots n\}  \qquad 
\theta_{\ell}(Y_j) = \theta^0_{\ell} e^{- \i 2 \pi j \tau_j} + \xi_{\ell,j},
\end{equation}
where $\theta^0: =(\theta^0_{\ell})_{\ell \in \Z}$ denotes the true unknown Fourier coefficients of $f^0$. Owing to the white noise model, the variables $(\xi_{\ell,j})_{\ell,j}$ are independent standard (complex) Gaussian random variables: $\xi_{\ell,j} \sim_{i.i.d.} \cN_{\C}(0,1), \forall \ell,j$. 

For sake of simplicity,  $\gamma$ will refer to $\gamma(z) := \pi^{-1} e^{-|z|^2}, \forall z \in \C$, the density of the standard complex Gaussian centered distribution $\cN_{\C}(0,1)$, and $\gamma_{\mu}(.) := \gamma(.-\mu)$ is the density of the standard complex Gaussian with mean $\mu$.
We keep also the same notation for $p$ dimensional complex Gaussian densities 
$\gamma(z) := \pi^{-p} e^{-\|z\|^2}, \forall z \in \C^p$, where $\|z\|$ is the euclidean $p$ dimensional norm of the complex vector $z$.

For any frequence $\ell$, equation \eqref{eq:model} implies that $\theta_\ell(Y)$ follows a mixture of complex Gaussian standard variables with mean $\theta^0_{\ell} e^{- \i 2 \pi \ell \varphi}, \varphi \in [0,1]$:
$$
\theta_\ell(Y) \sim \int_{0}^1 \gamma_{\theta^0_{\ell} e^{- \i 2 \pi \ell \varphi}}(\cdot) dg(\varphi).
$$
In the sequel, for any phase $\varphi \in [0,1]$ sampled according to any distribution $g$, and for any $\theta \in \ell^2(\Z)$, $\theta \bullet \varphi$ will denote the element of $\ell^2(\Z)$ given by
$$
\forall \ell \in \Z \qquad (\theta \bullet \varphi)_{\ell}: = \theta_{\ell} e^{- \i 2 \pi \ell \varphi}.
$$
When $\theta$ is a complex vector, 
for instance $\theta = (\theta_{-\ell}, \ldots, \theta_\ell)$, we keep the same notation  $\theta \bullet \varphi := (\theta_{-\ell} e^{\i 2 \pi \ell \varphi}, \ldots, \theta_0, \theta_1 e^{-\i 2 \pi \varphi}, \ldots, \theta_\ell e^{-\i 2 \pi \ell \varphi})$ to refer to the $2\ell+1$ dimensional vector.
It  corresponds to a rotation of each coefficient $\theta_{\ell}$  around the origin with an angle $2 \pi \ell \varphi$. 
According to this notation, the law of the infinite series (of Fourier coefficients of $Y$) can thus be rewritten as
$$
\theta(Y) \sim \int_{0}^1 \gamma_{\theta^0 \bullet \varphi}(.) dg(\varphi).
$$
One should remark the important fact that from one frequency to another, the rotations used to build $\theta(Y)$ are not independent, which traduces the fact that the coefficients $\left(\theta_{\ell}(Y)\right)_\ell$ are highly correlated.

\subsection{Notations on Mixture models}

Our study will intensively use some classical tools of mixture models,  see for instance the papers of \cite{GvdW01} or \cite{GW}. We thus choose to keep some notations already used in such works. 

For any vector $\theta \in \ell^2_\C(\Z)$ corresponds a function $f \in L^2([0,1])$ according to equation \eqref{eq:fourier} and for any measure $g \in \M$, $\PP_{\theta,g}$ will refer to the law of the vector of $\ell^2(\Z)$ described by the location mixture of Gaussian variables:
$$
\PP_{\theta,g}: = \int_{0}^1 \gamma_{\theta \bullet \varphi}(.) dg(\varphi).
$$
This mixture model is of infinite dimension since $\theta$ belongs to $\ell^2(\Z)$. Following an obvious notation shortcut, $\PP_{f,g}$ will be its equivalent for the functional law on curves derived from $\PP_{\theta,g}$. When $\theta$ is of finite length $k$, $p_{\theta,g}$ will be the density with respect to the Lebesgue measure on $\C^k$ of the law $\PP_{\theta,g}$:
$$
\forall z \in \C^k \qquad 
p_{\theta,g}(z): = \int_{0}^1 \gamma (z - \theta \bullet \varphi) dg(\varphi).
$$

We also use standard objects such as the Hellinger distance $d_H$ between probability measures and the Total Variation distance $d_{TV}$, as well as covering numbers of metric spaces such as $D(\epsilon,\cP,d)$. These objects are precisely described in Appendix \ref{section:appendix_proba}.

\paragraph{Bayesian frequentist consistency rate}

In our setting,  $d$ is chosen according to one of the metric introduced above 
($d_H$ or $d_{TV}$) 
on the set  
$$\cP:= \left\{ \PP_{f,g}  \vert   (f,g) \in \cH_s
\otimes \M \right\}.$$
We can now remind Theorem 2.1 of \cite{GGvdW00} which will be useful for our purpose. 
\begin{theo}[Posterior consistency and convergence rate, \cite{GGvdW00}]\label{theo:posterior}
Assume that  a sequence $(\epsilon_n)_n$ with $\epsilon_n \rightarrow 0$ and $n \epsilon_n^2  \rightarrow + \infty$, a constant $C>0$, and a sequence of sets $\cP_n \subset \cP$ satisfy

\begin{equation}\label{eq:bound_covering}
\log D(\epsilon_n,\cP_n,d) \leq n \epsilon_n^2
\end{equation}
\begin{equation}\label{eq:bound_sieve}
\Pi_n \left( \cP \setminus \cP_n \right) \leq e^{- n \epsilon_n^2 (C+4)}
\end{equation}
\begin{equation}\label{eq:bound_neighbourhood}
\Pi_n \left( \PP_{f,g} \in \cP \vert d_{KL} (\PP_{f^0,g^0}, \PP_{f,g})
 \leq \epsilon_n^2 , V (\PP_{f^0,g^0}, \PP_{f,g}) \leq \epsilon_n^2 \right) \geq e^{-  n \epsilon_n^2 C}.
\end{equation}
Then there exists a sufficiently large $M$ such that
 $$\Pi_n\left( \PP_{f,g} : d(\PP_{f^0,g^0}, \PP_{f,g})  \geq M \epsilon_n | Y_1, \ldots Y_n \right) \longrightarrow 0$$ 
 in $\PP_{f^0,g^0}$ probability as $n\longrightarrow + \infty$.
\end{theo}
The posterior concentration rate obtained in the above result is $\epsilon_n$. The growing set $\cP_n$ is referred to as a Sieve over $\cP$. 
Generally,  this rate $\epsilon_n$ can be compared to the classical frequentist benchmark: for instance \cite{GGvdW00} obtained for the Log Spline model a contraction rate $\epsilon_n= n^{-s/(2s+1)}$ when the unknown underlying density belongs to an Hˆlder class $\cC^s([0,1])$, 
and this rate is known to be the optimal one (in the sense that it is the minimax one) in the frequentist paradigm over Hˆlder densities of regularity $s$ (see \cite{ibragimov_book}). 
Similarly, the recent work of \cite{RR} considers the situation of density estimation for infinite dimensional exponential families and reaches also contraction rates close or equal to the known optimal frequentist one.

\subsection{Bayesian prior and posterior concentration in the randomly shifted curves model}\label{section:prior}

We detail here the Bayesian prior $\Pi_n$ on $\cP$ used to obtain a polynomial concentration rate. Note that such prior will be in our work independent on the unknown smoothness  parameter $s$. 
As pointed in the paragraph above, it is sufficient to define some prior on the space $\cH_s \otimes \M$ since equation \eqref{eq:model} will then transport this prior to a law $\Pi_n$ on $\cP$. The two parameters $f$ and $g$ are picked independently at random following the next prior distributions.

\paragraph{Prior on $f$} The prior on $f$ is slightly adapted from \cite{RR}. It is defined on $\cH_s$ through 
$$
\pi := \sum_{\ell \geq 1} \lambda(\ell) \pi_\ell.
$$
Given any integer $\ell$, the idea is to decide to randomly switch on with probability $\lambda(\ell)$ all the Fourier frequencies from $-\ell$ to $+ \ell$. Then,  $\pi_{\ell}$ is  a distribution defined on $\ell^2(\Z)$  such that $\pi_{\ell}:=\otimes_{k \in \Z} \pi_{\ell}^k$ and
$$
\forall k \in \Z \qquad \pi^k_{\ell}= 1_{|k| > \ell} \delta_{0}+ 1_{|k| \leq \ell} \cN_\C(0,\xikn^2). 
$$
The randomisation of selected frequencies is done using $\lambda$, a probability distribution on $\N^{\star}$ which satisfies for $\rho \in (1,2)$:
 $$\exists (c_1,c_2) \in \R_+ \quad \forall \ell \in \N^{\star}  \qquad 
e^{-c_1 \ell^2\log^\rho \ell} \lesssim \lambda(\ell) \lesssim e^{-c_2 \ell^2 \log^\rho \ell}.
$$

The prior $\pi$ depends on the variance of the Gaussian laws  $\xikn$ used to sample the Fourier coefficients. 
In the sequel, we use a variance that depends on $n$ according to 
\begin{equation}\label{eq:variance_prior}
\xi_n^2 := n^{-\mu_s} (\log n)^{-\zeta},
\end{equation}
where $\mu_s$ and $\zeta$ are parameters that may depend on $s$ (non adaptive prior) or not (adaptive prior).

\paragraph{Prior on general probability distribution $g$} As our model  does not seem so far from a mixture Gaussian model, a natural prior on $g$ is built according to a Dirichlet process following the ideas of \cite{GvdW01}. Given any finite base measure $\alpha$ that has a positive continuous density on $[0,1]$ w.r.t. the Lebesgue measure, the Dirichlet process $D_{\alpha}$ generates a random probability measure $g$ on $[0,1]$. For any finite partition $(A_1, \ldots, A_k)$ of $[0,1]$, the probability vector $(g(A_1), \ldots, g(A_k))$ on the $k$-dimensional simplex has a Dirichlet distribution $Dir(\alpha(A_1), \ldots, \alpha(A_k))$. Such process may be built according to the Stick-Breaking construction (see for instance \cite{F73}). In the sequel, we refer to $q_{D,\alpha}$ for this prior based on the Dirichlet process $D_{\alpha}$.

\paragraph{Prior on smooth probability distributions $g$}
 In the sequel, we will also consider the special case of \textit{smooth} densities to push our results further. This set is characterised by a regularity parameter $\nu$ and a radius $A$:
$$
\Mnu := \left\{ g \in \M \quad \vert \quad \sum_{k \in \Z} k^{2 \nu} |\theta_k(g)|^2  <A^2\right\},
$$
where  $\M$ is the set of probability on $[0,1]$. At last, we will also need the set
$$
\M^{\star} := \left\{ g \in \M  \quad \vert \quad \forall k \in \mathbb{Z} \quad  \theta_k(g) \neq 0\right\}.
$$
The prior on $\Mnu$ will be rather different from the one defined on $\M$ above. Such smoothness is not compatible with Dirichlet priors and even kernel convolution with Dirichlet process seems problematic in our situation. 
Thus, we have chosen to use some prior based on gaussian process. \textit{In this case, we assume for smooth mixture models that we know the smoothness parameter $\nu$ of $g^0$, as well as the radius $A$ of the Sobolev balls where $g^0$ is living.}

Given $\nu \geq 1/2$ and $A>0$, we define the integer $k_{\nu} := \lfloor \nu -1/2 \rfloor $ to be the largest integer smaller  than $\nu-1/2$. We follow the strategy of section 4 in \cite{vdWvZ} and the important point is that we have to take into account the $1$-periodicity of the density $g$, as well as its regularity.
In this view, we denote $B$ a Brownian bridge between $0$ and $1$. The Brownian bridge can be obtained from a Brownian motion trajectory $W$ using
$B_t=W_t - t W_1$. Then, 
For any continuous function $f$ on $[0,1]$, we define the linear map

$$
J(f): t \longmapsto \int_{0}^t f(s) ds - t \int_{0}^1 f(u) du,
$$
and all its composition are $J_k = J_{k-1} \circ J$.
Moreover, in order to adapt our prior to the several derivatives of $g$ at points $0$ and $1$, we use the family of maps $(\psi_j)_{j=1 \ldots k_{\nu}}$ defined as
$$
\forall t \in [0,1] \qquad 
\psi_k(t) := \sin(2\pi k t) + \cos (2\pi k t).
$$

Our prior is now built as follows, we first sample a real Brownian bridge $(B_\tau)_{\tau \in [0,1]}$ and $Z_1, \ldots Z_{k_\nu}$ independent real standard normal random variables. This enables to generate the Gaussian process
\begin{equation}\label{eq:prior_process}
\forall \tau \in [0,1] \qquad w_\tau := J_{k_{\nu}}(B)(\tau) + \sum_{i=1}^{k_{\nu}}  Z_i \psi_i(\tau).
\end{equation}
 Given  $(w_{\tau},\tau \in [0;1])$  generated by \eqref{eq:prior_process}, we build $p_w$ through
\begin{equation}\label{eq:log_gaussian}
\forall \tau \in [0;1] \qquad p_{w}(\tau) := \frac{ e^{w_\tau}}{\int_{0}^1 e^{w_\tau} d\tau}.
\end{equation}
Hence, a prior on Gaussian process yields a prior on densities on $[0;1]$ and $p_w$ inherits of the smoothness $k_{\nu}$ of the Gaussian process $\tau \mapsto w_{\tau}$. According to our construction, we now consider the restriction of the prior defined above to the Sobolev balls of radius $2A$. This finally defines a prior distribution $q_{\nu,A}$ on $\mathfrak{M}_{\nu}([0,1])(2A)$.

\subsection{Main results}
Our two priors on the elements $g$ enable to define two priors on the model. The first one is  $\pi \otimes q_{D,\alpha}$ for general distributions $g$ although the second one is $\pi \otimes q_{\nu,A}$ for smooth mixture distributions. For sake of convenience, we will always use the notation $\Pi_n$ to refer to one of these two priors and just precise which prior is used in the statements of the theorems. We now give the two  results on the randomly SIM.

\subsubsection[Posterior contraction with the smooth prior]{Posterior contraction for $\Pi_n := \pi \otimes q_{D,\alpha}$}

\begin{theo}\label{theo:posterior_shift}
Define the prior distribution $\Pi_n := \pi \otimes q_{D,\alpha}$ and assume that 
$f^0 \in \cH_s$ with $s \geq 1$, then the values
 $\mu_s=2/(2s+2)$ and $\zeta=0$ in the definition of $\xikn$  yield a non adaptive prior  such that
$$
\Pi_n \left\{ \PP_{f,g} \text{ s.t. } d_H(\PP_{f,g} ,\PP_{f^0,g^0}) \leq M \epsilon_n \vert Y_1, \ldots Y_n \right\} = 1 + \cO_{\PP_{f^0,g^0}}(1)
$$
when $n \rightarrow + \infty$, for a sufficiently large constant $M$. Moreover, the contraction rate $\epsilon_n$ is given by
$$
\epsilon_n = n^{- s / (2 s+2)} \log n.
$$
The values $\mu=1/4$  and $\zeta=3/2$ yield the contraction rate
$$
\Pi_n \left\{ \PP_{f,g} \text{ s.t. } d_H(\PP_{f,g} ,\PP_{f^0,g^0}) \leq M \epsilon_n \vert Y_1, \ldots Y_n \right\} = 1 + \cO_{\PP_{f^0,g^0}}(1)
$$
for a sufficiently large constant $M$, when $n \rightarrow + \infty$ with
$$
\epsilon_n =\left\lbrace
\begin{array}{ll}
n^{-s/(2s+2)} \log n &\quad \text{if} \quad s \in [1,3] \\
n^{-3/8} \log n &\quad \text{if} \quad s \geq 3.
\end{array}\right.
$$
\end{theo}

Let us briefly comment this result. It first describes the posterior concentration around some neighbourhood of the true  law $\PP_{f^0,g^0}$ within a polynomial rate. Our prior is adaptive with the regularity $s$ as soon as $s\in[1,3]$ setting $\xi_n^2 = n^{-1/4} (\log n)^{-3/2}$. For this range of $s$, the convergence rate is $n^{-s/(2s+2)}$ up to a logarithmic term. To the best of our knowledge, the minimax frequentist rate is unknown for the problem on recovering $\PP_{f^0,g^0}$ when both $f^0$ and $g^0$ are unknown. An interpretation of such polynomial rate is rather difficult to provide. It may be interpreted as $-s/(2s+d)$ where $d$ is the number of dimension to estimate in the model ($f^0$ and $g^0$).
When $s$ becomes larger than $3$, the rate of Theorem \eqref{theo:posterior_shift} is "blocked" to $3/8$ (which corresponds to $s/(2s+2)$ when $s=3$) and does not match with $s/(2s+2)$. This difficulty is mainly due to the important condition $w_{\epsilon}^2 \lesssim l_{\epsilon}$ in Theorem \ref{theo:recouvrement}.

At last, the non adaptive prior based on $\xi_n^2 = n^{-2/(2s+2)}$ recovers the good rate $-s/(2s+2)$ for all $s$ larger than $1$.

\subsubsection[Posterior contraction with the Dirichlet prior]{Posterior contraction for $\Pi_n := \pi \otimes q_{\nu,A}$}

The second result concerns the special case of smooth densities $g$ an provide a result for the non adaptive prior based on the knowledge of $\nu$ and $A$.

\begin{theo}\label{theo:posterior_shift2}
For  the prior $\Pi_n := \pi \otimes q_{\nu,A}$ using $\xikn^2= n^{-1/4} (\log n)^{-3/2}$ in the definition of $\pi$, if 
$f^0 \in \cH_s$ with $s \geq 1$ and $g^0 \in \Mnu$, then there exists a sufficiently large $M$ such that
$$
\Pi_n \left\{ \PP_{f,g} \text{ s.t. } d_H(\PP_{f,g} ,\PP_{f^0,g^0}) \leq M \epsilon_n \vert Y_1, \ldots Y_n \right\} = 1 + \cO_{\PP_{f^0,g^0}}(1)
$$
when $n \longrightarrow + \infty$, where for an explicit $\kappa>0$: 
$$
\epsilon_n = n^{- \left[ \frac{\nu}{2\nu+1} \wedge \frac{s}{2s+2} \wedge \frac{3}{8}\right] } \log (n) ^{\kappa}.
$$
\end{theo}

We still obtain a similar result of polynomial order but note that the smoothness parameter $\nu$ of $g^0$ influences now the contraction rate.

\subsubsection{Posterior contraction on functional objects}

The former results establish some results on the law $\PP_{f,g} \in \cP$.
It is also possible to derive a second result on the objects $f \in \cH_s$ themselves, provided that we consider smooth classes for mixture distributions and provides a somewhat quite weak result on the posterior convergence towards the true objects $f^0$ and $g^0$. 

\begin{theo} \label{theo:semiparametric_rates}
For the prior  $\Pi_n := \pi \otimes q_{\nu,A}$ using $\xikn^2= n^{-1/4} (\log n)^{-3/2}$ in the definition of $\pi$,  the two following results hold.

i)
Assume that 
$f^0 \in \cF_s$ with $s \geq 1$ and $g^0 \in \Mnu$ with $\nu>1$, then there exists a sufficiently large $M$ such that
$$
\Pi_n \left\{ g \quad s.t. \, \|g-g^0\| \leq M \mu_n \vert Y_1, \ldots Y_n \right\} = 1 + \cO_{\PP_{f^0,g^0}}(1)
$$
with the contraction rate
$
\mu_n = \left( \log n \right)^{-\nu}.
$

ii)
In the meantime, assume that $g^0 \in \Mnu$ satisfies the inverse problem assumption:
$$
\exists (c) > 0 \quad \exists \beta > \nu+\tfrac{1}{2} \quad \forall k \in \mathbb{Z} \qquad |\theta_k(g^0)| \geq c k^{-\beta}
$$
then we also have 
$$
\Pi_n \left\{ f \quad s.t. \, \|f-f^0\| \leq M \tilde{\mu}_n \vert Y_1, \ldots Y_n \right\} = 1 + \cO_{\PP_{f^0,g^0}}(1)
$$
when $n \longrightarrow + \infty$. Moreover, the contraction rate $\tilde{\mu}_n$ is given by
$$
\tilde{\mu}_n = \left( \log n \right)^{- \frac{4s\nu}{2s+2\beta+1}}.
$$
\end{theo}

\subsubsection{Identifiability and lower bounds}

In paragraph \ref{sec:lowerbound}, we will stress the fact that it is indeed impossible to obtain frequentist convergence rates better than some power of $\log n$ even if our lower bound does not match exactly with the upper bound obtained in the previous result.

\begin{theo} 
i) The Shape Invariant Model is identifiable as soon as $(f^0,g^0) \in \cF_s\times \M^{\star}$: the canonical application 
$$\mathcal{I}\, : \,  (f^0,g^0) \in \cF_s \times \M^{\star} \longmapsto \PP_{f^0,g^0} \qquad \text{is injective}.$$

ii)
Assume that $(f^0,g^0) \in \cF_s \times \Mnu$, then there exists a sufficiently small $c$ such that the minimax rate of estimation over $ \cF_s \times \Mnu$ satisfies
$$
\liminf_{n \rightarrow + \infty} \left( \log n\right)^{2s+2}
\inf_{\hat{f} \in \cF_s}  \sup_{(f,g) \in \cF_s \times \Mnu} \|\hat{f} - f\|^2 \geq c,
$$
and
$$
\liminf_{n \rightarrow + \infty} \left( \log n\right)^{2\nu + 1}
\inf_{\hat{g} \in \cF_s}  \sup_{(f,g) \in \cF_s \times \Mnu} \|\hat{g} - g\|^2 \geq c .
$$
\end{theo}

\noindent
This result is far from being contradictory with the polynomial rate obtained in Theorem \ref{theo:posterior_shift}. One can make  at least three remarks:
\begin{itemize}
\item The first result provides a contraction rate on the probability distribution in $\mathcal{P}$ and not on the functional space $\cF_s$.
\item The link between $(f^0,g^0)$ and $\PP_{f^0,g^0}$ relies on the identifiability of the model, and the lower bound is derived from a net of functions $(f_i,g_i)_{i}$, which are really hard to identify according to the application $\mathcal{I}: (f,g) \mapsto \mathbb{P}_{f,g}$. On this net of functions, the injection is very ``flat'' and the two by two differences of $\mathcal{I}(f^i,g^i)$  are as small as possible and thus the pairs of functions $(f^i,g^i)$ become very hard to distinguish.
\item In fact, \cite{BG10} have shown that in the SIM, when $n \rightarrow + \infty$, it is impossible to recover the unknown true shifts. The abrupt degradation between the polynomial rates on probability laws in $\cP$  and the logarithmic rates on functional objects in $\cF_s \times \Mnu$ also occurs owing to such reason. 
One may argue that such an artefact could be avoided if one chooses a different distance on $\cF_s$, which may be better suited to our framework, such as
$$
d_{Frechet}(f_1,f_2) := \inf_{\tau \in [0,1]} \left\|f_1^{-\tau}-f_2\right\|.
$$ 
We do not have purchased further investigations with this distance on $\cF_s$ but it would certainly be  a nice progress to obtain posterior contraction using such a distance. We expect a polynomial rate, 
but it is clearly an open (and probably hard) task.
\end{itemize}

\section{Proof of Theorem \ref{theo:posterior_shift} and Theorem \ref{theo:posterior_shift2}} \label{sec:proof_main}


We aim to check conditions (\ref{eq:bound_sieve}) and (\ref{eq:bound_neighbourhood}) and then apply  Theorem \ref{theo:posterior}. In this view, we first define in section \ref{sec:entropy} a sieve $\cP_{\ell_{\epsilon},w_{\epsilon}}$, and our goal is to find some optimal calibration of $\epsilon$, $l_\epsilon$ and $w_{\epsilon}$ with respect to $n$.
We thus need to find a lower bound of the prior mass around some Kullback-Leibler neighbourhood of $\PP_{f^0,g^0} \in \cP$. These sets are defined as
$$\cV_{\epsilon_n}(\PP_{f^0,g^0},d_{KL}) =  \left\{ \PP_{f,g} \in \cP \vert d_{KL}(\PP_{f^0,g^0},\PP_{f,g}) \leq \epsilon_n^2 , V(\PP_{f^0,g^0},\PP_{f,g}) \leq\epsilon_n^2\right\}. 
$$
This will be done indeed considering Hellinger neighbourhoods instead of  Kullback-Leibler ones. 
A link between these two kinds of neighbourhood is given in section \ref{sec:link}. In section \ref{sec:Hellinger}, we work with the Hellinger neighbourhoods to exhibit some admissible sizes for $\epsilon_n$, $\ell_n$ and $w_n$. 
At last, we prove Theorem \ref{theo:posterior_shift} in section \ref{sec:core_proof}.

In all this section, we delay most technical proofs to the Appendix.

\subsection{Entropy estimates} \label{sec:entropy}

We first establish some useful results on the complexity of our model $\PP_{f,g}$ when $f\in\cH_s$ and $g \in \M$ in various situations ($f$ known, unknown, parametric  or not).

\subsubsection[Case of known f]{Case of known $f$}


We first give some useful results when $f$ is known and belongs to a finite dimensional vector space (the number of active Fourier coefficients is restricted to $[-\ell,\ell]$ for a given $\ell$). Then $\ell$ will be allowed to grow with $n$ and depend on a  parameter $\epsilon$ introduced below.
 Hence, $f$ is described by the parameter $\theta = (\theta_{-\ell}, \ldots, \theta_0, \ldots, \theta_\ell)$, and we define the set of all possible Gaussian measures
 $$
 \cA_{\theta} := \left\{ \gamma_{\theta \bullet \varphi}, \varphi \in [0,1]\right\}.
 $$
  Following the arguments of \cite{GW}, it is possible to establish the following preliminary result.
\begin{proposition}\label{prop:atheta}
For any sequence $\theta \in \C^{2\ell+1}$, one has
$$
 N_{[]}(\epsilon,\cA_{\theta},d_H) \leq \frac{4 \pi \sqrt{2(2\ell+1)} \|\theta\|_{\cH_1}}{\epsilon} (1+o(1)),
$$
where $o(1)$ goes to zero independently on $\ell$ and $\theta$ as $\epsilon\rightarrow 0$, and
$$
\log N(\epsilon,\cA_{\theta},d_H) \lesssim \log \ell + \log \|\theta\|_{\cH_1} + \log \frac{1}{\epsilon}.
$$
\end{proposition}

Assume now that $g$ possesses a finite number of $k$ points in its support, one can deduce from the proposition above a simple corollary that exploits the complexity of the simplex of dimension $k-1$ (see for instance the proof of Lemma 2 in \cite{GW}).

\begin{proposition}\label{prop:Mktheta}
Assume that $f$ is parametric and known ($\theta \in \C^{2\ell+1}$) and define
\[ \cM^k_{\theta} := \left\{ \sum_{i=1}^k g(\varphi_i) \gamma_{\theta\bullet\varphi_i} : \varphi_i \in [0,1], \, g(\varphi_i) \geq 0,   \forall i \in \llbracket  1,k  
\rrbracket\, \text{and} \,  \sum_{i=1}^k g(\varphi_i) = 1\right\} \]
for a number of components $k$ that may depend on $\epsilon$ (as $\ell$ does). Then
$$
H_{[]}(\epsilon,\cM^k_{\theta},d_H) \lesssim k \left(\log \ell + \log \|\theta\|_{\cH_1} + \log \frac{1}{\epsilon}\right).
$$
\end{proposition}


We then naturally provide a description of the situation when $f$ is known and parametrized by an infinite sequence $\theta \in \ell^2(\Z)$. According to the previous computations, and using a truncation argument at frequency $\ell_\epsilon=\epsilon^{-1/s}$ in the Sobolev space $\cH_s
$, one can show the following result.

\begin{coro} \label{corol:Mktheta}
Assume $f \in \cH_s$ known for $s\geq 1$ ($\theta :=\theta(f)$ such that $\sum_{j\in\Z} |\theta_j|^{2} |j|^{2s} < +\infty$), using the same set $\cA_{\theta}$ as in Proposition \ref{prop:atheta} with $\ell_\epsilon=\epsilon^{-1/s}$, then
$$
H_{[]} (\epsilon,\cA_{\theta},d_H) \lesssim  \frac{s+1}{s}\log \frac{1}{\epsilon} + \log \|\theta\|_{\cH_1}.
$$
Similarly, one also has
$$
H_{[]}(\epsilon,\cM^k_{\theta},d_H) \lesssim k \left(\frac{ s+1}{s}\log \frac{1}{\epsilon} + \log \|\theta\|_{\cH_1}\right).
$$
\end{coro}

The next step is to consider a continuous mixture for $g$, which is the more natural case. For $f$ known, let
\[ \cP_{f} := \left\{ \PP_{f,g} \, \vert \, g \in \M\right\}. \]
Once again, we will only consider functions $f$ with null Fourier coefficients of order higher than $\ell_\epsilon$. For sake of simplicity, we will omit the dependence on $\epsilon$ with the notation $\ell$.

It would be quite tempting to use the results of \cite{GvdW01} to bound the bracketing entropy of $\cP_{f}$, but indeed as pointed by \cite{MM11} applying directly the bounds obtained in Lemma 3.1 and Lemma 3.2 of \cite{GvdW01} to our setting yields a too weak result: the size of the upper bound on $H_{[]}(\epsilon,\cP_{f},d_H)$ will have a too strong dependency on $\ell$. 
By the way, we have to carefully adapt the approach of \cite{GvdW01} to obtain a sufficiently sharp upper bound of the entropy of $\cP_{f}$. Such bound is given in the next result, in which we provide a majorization of the entropy with respect to the Total Variation distance which is easier to handle here. 
Note that all the previous results are still true if we use $d_H$ instead of $d_{TV}$ since \eqref{eq:dvt_dh} also permits to retrieve entropy bounds for $d_H$ from entropy bounds for $d_{TV}$.

\begin{proposition} \label{prop:Pf}
 Let $\epsilon>0$ and $s > 0$, if $\log \tfrac{1}{\epsilon} \lesssim \ell$ and $f \in \cH^\ell$ is such that $\|\theta\|^2 \lesssim 2\ell+1$, then
 \[ \log N (\epsilon,\cP_{f} ,d_{TV}) \lesssim \ell^2 \left(\log \frac{1}{\epsilon} + \log \|\theta\|_{\cH_1}\right).  \]
 If furthermore $w  \lesssim \sqrt{2\ell+1}$ then
 \[ \sup_{f \in \cH^\ell : \|\theta(f)\| \leq w} \log N (\epsilon,\cP_{f} ,d_{TV}) \lesssim   \ell^2 \left(\log \frac{1}{\epsilon} + \log \ell\right).\]
\end{proposition}

The second inequality opens the way for the case of unknown $f$ given below. It is possible since in the first inequality we have carefully expressed the dependency on $f$ and $\ell$.

 The method to build an $\epsilon$-covering of $\cP_f$ follows two natural steps:
 \begin{itemize}
 \item 
  approximate any mixture $g$ by a finite one $\tilde{g}$ such that $$d_{TV}(\PP_{\theta,g},\PP_{\theta,\tilde{g}}) \leq\epsilon/2,$$ with a number of components of the finite mixture $\tilde{g}$ uniformly bounded in $g$ (depending on $f$ and $\epsilon$);
  \item    use Proposition \ref{prop:Mktheta} for the finite mixture to well approximate $\PP_{\theta,\tilde{g}}$. 
 \end{itemize}
 The proof itself is delayed to the Appendix.

\subsubsection[Case of unknown f]{Case of unknown $f$}

We now describe the picture when $f$ is unknown, which is the main objective of this paper. We assume that $f$ belongs to $\cH_s$. In order to bound the bracketing entropy, we define a sieve over 
$\cH_s$ which depends on a frequency cut-off  $\ell$  and a size parameter $w$. We then get
$$
\cP_{\ell,w} := \left\{ \PP_{f,g} \, \vert \, f \in \cH_s^{\ell}, \|\theta(f)\| \leq w, g \in \M\right\}.
$$

\begin{theo}\label{theo:recouvrement}
Let be given $\epsilon>0$ small enough, and 
assume that $\ell_{\epsilon}$ and $w_{\epsilon}$ are such that $ \log \frac{1}{\epsilon} \lesssim l_{\epsilon}$ and $w_{\epsilon} \lesssim \sqrt{\ell_\epsilon}$,  then
$$
\log N (\epsilon,\cP_{\ell_{\epsilon},w_\epsilon} ,d_{TV}) \lesssim  l_\epsilon^2 \left(\log \frac{1}{\epsilon} + \log \ell_\epsilon\right).
$$
\end{theo}
The proof of Theorem \ref{theo:recouvrement} is based on two simple results. The first one is the Girsanov formula obtained by \cite{BG10} in appendix A.2.2 (in the case of known $g$): it can be extended to the situation of unknown $g$ and complex trajectories as in \eqref{eq:model}, which leads to
\begin{equation} \label{eq:Girsanov}
 \frac{d\PP_{f,g}}{d\PP_{f^0,g^0}}(Y) = \frac{\int_{0}^1 exp \left( 2\re\langle f^{- \alpha_1}, d Y \rangle - \|f^{-\alpha_1}\|^2\right) d g(\alpha_1) }{\int_{0}^1 exp \left(2 \re \langle f^{0, - \alpha_2}, d Y \rangle - \|f^{0, -\alpha_2}\|^2\right) d g^0(\alpha_2) },
\end{equation}
for any measurable trajectory $Y$.

The second result is given in the following lemma.
\begin{lemma} \label{lemma:dVT_sur_f}
 Let $f$ and $\tilde{f}$ be any functions in $L^2_\C([0, 1])$, $g$ be any shift distribution in $\M$, then
 \begin{equation*}
  d_{TV}(\PP_{f,g},\PP_{\tilde{f},g}) \leq \frac{\|f-\tilde{f}\|}{\sqrt{2}}.
 \end{equation*}
\end{lemma}

\begin{proof}[Proof of Theorem \ref{theo:recouvrement}]
The idea of the demonstration is to build a $\epsilon$-covering of $\cP_{\ell,w}$ with $\epsilon/2$-coverings for $f$ and $g$. First, let $\PP_{f,g}$ and  $\PP_{\tilde{f},\tilde{g}}$ two elements of $\cP_{l,w}$ and remark that by the triangle inequality
$$
d_{TV}(\PP_{f,g},\PP_{\tilde{f},\tilde{g}}) \leq d_{TV}(\PP_{f,g},\PP_{\tilde{f},g}) + d_{TV}(\PP_{\tilde{f},g},\PP_{\tilde{f},\tilde{g}}).
$$
We will look for a covering method that will use the inequality above and a tensorial argument, it requires to bound both terms. The majorization of the first one comes from Lemma \ref{lemma:dVT_sur_f}. The second term is handled uniformly in $\tilde{f}$ by Proposition \ref{prop:Pf}. 

 Now, we build $\epsilon/2$-coverings of $\PP_{f,g}$ for fixed $g$ from an $\epsilon/\sqrt{2}$-covering of $f$ for the $L^2$-norm:
 \[ \log N \left(\epsilon/\sqrt{2}, \left\{ f \in \cH_s^{\ell_\epsilon}, \|\theta(f)\| \leq w_\epsilon\right\}, \|\cdot\|\right) \lesssim \ell_\epsilon \log \frac{w_\epsilon}{\epsilon} = o\left( \ell_\epsilon^2 \log \frac{1}{\epsilon}\right). \]
\end{proof}
\noindent
According to inequality \eqref{eq:dvt_dh} and since $\log \frac{1}{\epsilon^2} \lesssim \log \frac{1}{\epsilon}$, we can easily deduce the next corollary.
\begin{coro}\label{coro:hellinger}
Let be given $\epsilon>0$ small enough, and assume $\log \frac{1}{\epsilon} \lesssim \ell_{\epsilon}$ and $w_\epsilon \leq \sqrt{2\ell_\epsilon+1}$, then
$$
\log N (\epsilon,\cP_{\ell_\epsilon,w_\epsilon},d_H) \lesssim \ell_{\epsilon}^2 \left(\log \frac{1}{\epsilon} + \log \ell_\epsilon \right).
$$
\end{coro}

\begin{remark}
i) Even if the model studied here is a very special case of Gaussian mixture models, one may think that such kind of results may help the analysis of more general mixture cases within a  growing dimension setting.

ii) In our case, we will use a much higher choice of $l_{\epsilon}$ than $\log \frac{1}{\epsilon}$. This choice will be fixed in section \ref{sec:core_proof}.
\end{remark}


\subsection{Link between Kullback-Leibler and Hellinger neighbourhoods}\label{sec:link}

We first recall a useful result of Wong \& Shen given as Theorem 5 in \cite{WS95}. It enables to handle Hellinger neighbourhoods instead of $\cV_{\epsilon_n}(\PP_{f^0,g^0},d_{KL})$, which is generally easier for mixture models.

\begin{theo}[Wong \& Shen]\label{theo:wong_shen}
Let $\mu$ and $\nu$ be two measures such that $\mu$ is a.c. with respect to $\nu$ with a density $q = d \mu/d\nu$. Assume that $d_H(\mu,\nu)^2 = \int [\sqrt{q} -1]^2 d  \nu\leq \epsilon^2$ and that there exists $\delta \in (0,1]$ such that 
\begin{equation}\label{eq:condition_moment}
M_\delta^2 
:=\int_{q\geq e^{1/\delta}} q^{\delta+1} d\nu < \infty.
\end{equation}
Then, for $\epsilon$ small enough, there exists a universal constant $C$ large enough   such that
$$
d_{KL}(\mu,\nu)  = \int q \log q d \nu \leq C \log(M_\delta)\epsilon^2 \log \frac{1}{\epsilon},
$$
and
$$
V(\mu,\nu) \leq \int q  \log^2 q d \nu \leq C \log(M_\delta)^2 \epsilon^2 \left[\log \frac{1}{\epsilon}\right]^2. 
$$
\end{theo}
Hence, Hellinger neighbourhoods are almost Kullback-Leibler ones (up to some logarithm terms) provided that a sufficiently large moment exists for $q$ ($q \log q$ is killed by $q^{1+\delta}$ for large values of $q$ and a second order expansion of $q \log q - q + 1$ around $1$ yields a term similar to $[\sqrt{q}-1]^2$). 
Next proposition shows that condition (\ref{eq:condition_moment}) is satisfied in our SIM.

\begin{proposition}\label{prop:appli_wong_shen}
 For any $\PP_{f^0,g^0} \in \cP$, and for any $f\in \cH_s$ such that $\|f\|\leq 2 \|f^0\|$, and any $g \in \M$, define $q=\frac{d\PP_{f^0,g^0}}{d\PP_{f,g}}$. There exists $\delta \in (0,1]$ such that the constant defined in equation \eqref{eq:condition_moment} $M_\delta^2$ is uniformly bounded with respect to $f$.
\end{proposition}

 \subsection{Hellinger neighbourhoods}\label{sec:Hellinger}
Proposition \ref{prop:appli_wong_shen} will allow to use  Theorem \ref{theo:wong_shen}, thus we now aim to find a lower bound on Hellinger neighbourhood of $\PP_{f^0,g^0}$. Consider a frequency cut-off $\ell_n$ that will be fixed later. For any $f  \in \cH_s^{\ell_n}$ and $g\in \M$, remind that we denote $\theta:=\theta(f)$ as well as $\theta^0 = \theta(f^0)$. We define $f^0_{\ell_n}$ the $L^2$ projection of $f^0$ on the subspace $\cH_s^{\ell_n}$.
 
 For sake of simplicity, $\EE_0F(Y)$ will refer to the expectation of a function $F$ of the trajectory $Y$ when $Y$ follows $\PP_{f^0,g^0}$.
The triangle inequality applied to the Hellinger distance shows that
\[ d_H(\PP_{f^0,g^0},\PP_{f,g}) \leq \overbrace{d_H(\PP_{f^{0},g^0},\PP_{f^{0}_{\ell_n},g^0})}^{(E_1)} + \overbrace{d_H(\PP_{f^{0}_{\ell_n},g^0},\PP_{f^{0}_{\ell_n},g})}^{(E_2)} + 
\overbrace{d_H(\PP_{f^{0}_{\ell_n},g},\PP_{f,g})}^{(E_3)}. \]
In the sequel, we will provide sufficiently sharp upper bound on $(E_1)$, $(E_2)$, $(E_3)$ so that we will be able to find a suitable lower bound of the prior mass of Hellinger neighbourhoods.

\paragraph{Upper bound of $(E_1)$} We first bound $(E_1)$ using $d_H^2 \leq d_{KL}$ with the Girsanov formula \eqref{eq:Girsanov}
\begin{align*}
(E_1) 
 &\leq \sqrt{d_{KL}(\PP_{f^0,g^0}, \PP_{f^{0}_{\ell_n},g^0})}\\
 &= \left(\EE_0 \left[ - \log \frac{\int_0^1 \exp\left(2\re \langle f^{0,-\alpha}_{ \ell_n}, d Y\rangle - \|f^{0}_{\ell_n}\|^2\right) d g^0(\alpha)}{\int_0^1 \exp\left(2\re \langle f^{0,-\alpha}, d Y\rangle - \|f^0\|^2\right) d g^0(\alpha)} \right]\right) ^{1/2}\\
 & := (\tilde{E_1})\\
 \end{align*}
We now obtain the upper bound of $(E_1)$  according to the next proposition.

\begin{proposition}\label{prop:E1} Assume that $Y \sim \PP_{f^0,g^0}$ and $f^0\in \cH_s$, then
 $$ (E_1) \leq (\tilde{E_1}) \leq \sqrt{2} \|f^0-f^{0}_{\ell_n}\|\leq \sqrt{2} \|f^0\|_{\cH_1} \ell_n^{-s}.$$
\end{proposition}

\paragraph{Upper bound of $(E_3)$}
 
 We will be interested in the Hellinger distance when $f^0_{\ell_n}$ is close to $f$, and the dimension $\ell_n$ grows up to $+\infty$ (the mixture law on $[0,1]$ is the same for the two laws). The important fact will be its exclusive dependence with respect to the $L^2$ distance between $f^0_{\ell_n}$ and $f$. 
 This upper bound is given in the next proposition, whose proof is immediate from Lemma \ref{lemma:dVT_sur_f} and equation~\eqref{eq:dvt_dh}.

\begin{proposition}\label{prop:E3}
Assume that $f \in \cH_s^{\ell_n}$ and $g\in \M$, then
\[ d_H(\PP_{f^0_{\ell_n},g},\PP_{f,g}) \leq 2^{1/4} \sqrt{ \|f-f^0_{\ell_n}\|}. \]
\end{proposition}

\paragraph{Upper bound for $(E_2)$}
This term is clearly the more difficult to handle. We will obtain a convenient result using some elements obtained in Proposition \ref{prop:Pf}. For a given $\epsilon_n>0$, $\ell_n, f^0_{\ell_n} \in \cH_s^{\ell_n}$ and $g^0 \in \M$, we know that one may find a mixture model $\tilde{g}$ such that $d_H(\PP_{f^0_{\ell_n},g^0},\PP_{f^0_{\ell_n},\tilde{g}})< \epsilon_n$ and $\tilde{g}$ has $C \ell_n^2$ points of support in $[0,1]$ as soon as $\epsilon_n$ is small enough and $\log \frac{1}{\epsilon_n} \lesssim \ell_n$ (the condition $\|f^0_{\ell_n}\|^2\leq 2\ell_n+1$ is immediate since $f^0$ does not depend on $n$). The next step is to control the Hellinger distance $d_H(\PP_{f^0_{\ell_n},g},\PP_{f^0_{\ell_n},\tilde{g}})$ for $g\in\M$, and this can be done thanks to an adaptation in dimension $2 \ell_n+1$ of Lemma 5.1 of \cite{GvdW01}. 

\begin{lemma}\label{lemma:mixture_lemma}
Let be given $\tilde{g}$ a discrete mixture law whose support is of cardinal $J$ whose support points $(\varphi_j)_{j =1 \ldots J}$ are such that $\tilde{g}(\varphi_j)=p_j$ and $\eta$-separated, \textit{i.e.} $|\varphi_j - \varphi_i| \geq \eta, \forall i \neq j$, then  $\forall \check{g} \in \M$
$$
d_H^2(\PP_{f^0_{\ell_n},\tilde{g}},\PP_{f^0_{\ell_n},\check{g}})\leq \sqrt{\frac{\pi}{2}} \|f^0_{\ell_n}\|_{\cH_1} \eta + 2 \sum_{j=1}^J \left| \check{g}([\varphi_j-\eta/2,\varphi_j+\eta/2]) - \tilde{g}(\varphi_j)\right|.
$$
\end{lemma}
Note that Lemma \ref{lemma:mixture_lemma} needs a discrete mixture with $\eta$-separated support points. The following result permits to obtain such a mixture.

\begin{proposition} \label{prop:finite_mixture}
 Assume that and $f^0 \in \cH_s$ for $s\geq 1$, $g^0\in\M$, and $\log \frac{1}{\epsilon_n} \lesssim \ell_n$. For any $\eta_n\leq\epsilon_n^2$, there exists a discrete distribution $\tilde{g}$ with in its support at most $J_n\lesssim \ell_n^2$ points denoted $(\psi_j)_{j=1 \ldots J_n}$, such that these points are $\eta_n$-separated, and
 \[ d_H(\PP_{f^{0}_{\ell_n},g^0},\PP_{f^{0}_{\ell_n},\tilde{g}})\leq\left(1+(8\pi)^{1/4}\|f^0_{\ell_n}\|_{\cH_1}^{1/2}\right) \epsilon_n. \]
 Furthermore, for any $g\in\M$,
 \begin{multline*} 
 d_H(\PP_{f^{0}_{\ell_n},g^0},\PP_{f^{0}_{\ell_n},g}) 
 \leq \left(1+(8\pi)^{1/4}\|f^0\|_{\cH_1}^{1/2}\right) \epsilon_n\\ + \sqrt{ \sqrt{\frac{\pi}{2}} \|f^0\|_{\cH_1} \eta_n + 2 \sum_{j=1}^{J_n} \left|g(\psi_j - \eta_n/2, \psi_j+\eta_n/2) - \tilde{g}(\psi_j)\right|}.
\end{multline*}
\end{proposition}

Latter in the paper we obtain a more general upper bound for $(E_2)$, based on the Wasserstein distance. We could use it to retrieve Proposition \ref{prop:finite_mixture}, but it also leads to Hellinger neighbourhoods described in terms of the Total Variation distance from $g$ to $g^0$. This last distance is adapted to smooth densities $g$ but not to the ones considered here, when the prior distribution for $g$ is a Dirichlet process.

\paragraph{Description of a Hellinger neighbourhood}

We can now gather the upper bounds of $(E_1)$, $(E_2)$, and $(E_3)$ to get the following result.

\begin{proposition}\label{prop:Hellinger_neighbourhood}
 Assume that $f^0 \in \cH_s$ for $s\geq 1$ and $g^0\in\M$. Choose the threshold such as  $\epsilon_n^{-1/s} \lesssim \ell_n \lesssim \epsilon_n^{-1/s}$ and $\eta_n := \epsilon_n^2$, and consider the finite mixture $\tilde{g}$ provided by Proposition \ref{prop:finite_mixture}. Define 
 \begin{align*}
  \cG_{\epsilon_n} &:= \left\{ g \in \M :  \sum_{j=1}^{J_n} |g(\psi_j - \eta_n/2, \psi_j+\eta_n/2) - \tilde{g}(\psi_j)| \leq \epsilon_n^2 \right\}, \\
  \cF_{\epsilon_n} &:= \left\{ f \in \cH_s^{\ell_n} : \|f-f^0_{\ell_n}\| \leq \epsilon_n^2 \right\}.
 \end{align*}
 Then, there exists a constant $C_{0}$ depending only on $\|f^0\|_{\cH_1}$ such that 
 for any $g \in \cG_{\epsilon_n}$ and $f \in \cF_{\epsilon_n}$,
 $$ d_H\left(\PP_{f^0,g^0}, \PP_{f,g}\right) \leq C_{0} \epsilon_n. $$
\end{proposition}

\subsection{Checking the conditions of Theorem \ref{theo:posterior}}\label{sec:core_proof}

We first prove the minoration for the lower bound (\ref{eq:bound_neighbourhood}), necessary to apply Theorem \ref{theo:posterior}.

\begin{proposition}\label{prop:prior_mino}
Assume that and $f^0 \in \cH_s$ for $s\geq 1$ and $g^0\in\M$. 
For any sequence $(\epsilon_n)_{n \in \N}$ which converges to $0$ as $n \rightarrow + \infty$, 
and for the  prior defined in paragraph \ref{section:prior}, there exists a constant $c>0$ such that
\[ \Pi_n \left( \PP_{f,g} \in \cP : d_{KL} (\PP_{f,g}, \PP_{f^0,g^0}) \leq \epsilon_n^2 , V (\PP_{f,g}, \PP_{f^0,g^0}) \leq \epsilon_n^2 \right) \geq h_n,\] 
where 
$$
h_n := e^{-(c +o(1))\, \left[ \epsilon_n^{-2/s} \left(\log (1/\epsilon_n)\right)^{\rho+2/s} \vee \xikn^{-2} \right] }.
$$
\end{proposition}

Proposition \ref{prop:prior_mino} relies on Theorem \ref{theo:wong_shen}, which permits to use Hellinger neighbourhoods instead of $\cV_{\epsilon_n}(\PP_{f^0,g^0},d_{KL})$, 
and on Proposition \ref{prop:Hellinger_neighbourhood}, which describes suitable Hellinger neighbourhoods. To control their prior mass, we remind the following useful result appeared as Lemma 6.1 of \cite{GGvdW00}. This enables to find a lower bound of $\ell_1$-ball of radius $r$ under Dirichlet prior.

\begin{lemma}[\cite{GGvdW00}]\label{lemma:mino_simplex}
Let $r>0$ and $(X_1, \ldots, X_N)$ be distributed according to the Dirichlet distribution on the $\ell_1$ simplex of dimension $N-1$ with parameters $(m,\alpha_1, \ldots, \alpha_N)$. Assume that $\sum_{j} \alpha_j =m$ and $A r^{b} \leq \alpha_j \leq 1$ for some constants $A$ and $b$. Let $(x_1, \ldots, x_N)$ be any points on the $N$ simplex, there exists $c$ and $C$ that only depend on $A$ and $b$ such that if $r \leq 1/N$
$$
\Pr \left(\sum_{j=1}^N |X_j - x_j| \leq 2 r  \right) \geq C \exp \left(-c N \log \frac{1}{r} \right)
$$
\end{lemma}

In the proof of Proposition  \ref{prop:prior_mino} (delayed to the Appendix), one can see that we could obtain a suitable lower bound as soon as  $\lambda(\ell_n) \geq e^{- c \ell_n^2 \log \ell_n}$ for a constant $c$. Of course, a distribution $\lambda$ with some heavier tail would also suit here. However, such a heavier tail is not suitable for the control of the term $\Pi_n\left(\cP \setminus \cP_n\right)$ which 
is detailed in the next proposition.

\begin{proposition} \label{prop:prior_majo}
For any sequences  $k_n \mapsto + \infty$ and $\epsilon_n \mapsto 0$ as $n \mapsto + \infty$, 
define $w_n^2 = 4k_n+2$, then there exists a constant $c$ such that
$$
\Pi_n \left(\cP \setminus \cP_{k_n,w_n}\right)  \leq  e^{ 
- c [ k_n^2 \log^\rho(k_n) \wedge k_n \xikn^{-2}]},
$$
and
$$
\log D\left( \epsilon_n, \cP_{k_n,w_n},d_H \right) \lesssim  k_n^2 \left[ \log k_n + \log \frac{1}{\epsilon_n} \right].
$$
\end{proposition}

We are now able to conclude the proof of the posterior consistency for Theorem \ref{theo:posterior_shift}.
\begin{proof}[Proof of Theorem \ref{theo:posterior_shift}]
 Take $\epsilon_n := n^{-\alpha} (\log n)^{\kappa}$ and $k_n := n^{\beta} (\log n)^{\gamma}$. From our definition \eqref{eq:variance_prior}, we have also $\xikn^{-2} = n^{\mu_s} (\log n)^{\zeta}$, and we look for admissible values of $\alpha$, $\beta$, $\kappa$, $\gamma$, $\mu_s$, and $\zeta$ in order to satisfy \eqref{eq:bound_covering}, \eqref{eq:bound_sieve} and \eqref{eq:bound_neighbourhood}. 

 Proposition \eqref{prop:prior_mino} imposes that in order to satisfy \eqref{eq:bound_neighbourhood}, we could check that
$$
\epsilon_n^{-2/s}  (\log \epsilon_n)^{\rho+2/s} \vee n^{\mu_s} (\log n)^{\zeta} \ll n \epsilon_n^2 = n^{1-2\alpha} (\log n)^{2 \kappa}.
$$
This is true as soon as $\epsilon_n$ satisfies 
$$
\alpha \leq \frac{s}{2s+2} \qquad \text{and} \qquad 
\kappa > (\rho s+2)/(2s+2).$$
 Moreover, we obtain the first  condition on $\mu_s$: 
$\mu_s \leq 1 - 2\alpha$, and if $\mu_s = 1 - 2\alpha$ then $\zeta < 2 \kappa$.

Now, Proposition \eqref{prop:prior_majo} shows that \eqref{eq:bound_covering} is fulfilled provided that
\begin{equation}\label{eq:cond_kn}
k_n^2 \left[\log k_n+ \log \frac{1}{\epsilon_n} \right] \lesssim n \epsilon_n^2 =  n^{1-2\alpha} (\log n)^{2 \kappa}.
\end{equation}
This condition is satisfied when 
$
2 \beta \leq 1-2\alpha$ and $2 \gamma+1 \leq 2 \kappa.$
At last, Proposition \eqref{prop:prior_majo}  again ensures that \eqref{eq:bound_sieve} is true as soon as 
$$
k_n^2 \log^\rho k_n \wedge k_n n^{\mu_s} \gtrsim n \epsilon_n^2
$$
and we deduce from \eqref{eq:cond_kn} that
$$2 \beta = 1-2\alpha \qquad \text{and}\qquad  -\rho/2 + \kappa \leq \gamma \leq -1/2+\kappa. 
$$
Moreover, we also see that $\beta+\mu_s \geq 1-2\alpha$, hence $\mu_s \geq 1/2-\alpha$, and if $\mu_s = 1/2-\alpha$ then $\gamma+\zeta \geq 2\kappa$; the former condition on $\mu_s$ yields $\mu_s \geq 1/2 - \alpha \geq \frac{1}{2s+2}$ (which naturally drives us to set $\mu_s = 1/4$ (case $s=1$) for adaptive prior).

We split the proof according to the adaptive or non adaptive case. 
\paragraph{Adaptive prior}
We first set $\mu$ independent of $s$ and equal to $1/4$. For any $s \in [1,3]$, we see that $\alpha(s)=s/(2s+2)$ is the admissible largest value of $\alpha$ and $\alpha(s)=3/8<s/(2s+2)$ as soon as $s > 3$. The corresponding value of $\beta$ is $1/(2s+2)$ when $s \in [1,3]$ and $\beta=1/8$ otherwise. Any choice of $\zeta \in [3/2, 2)$ permits to deal with the conditions on $\zeta$ that appears when $s=1$ or $s\geq 3$. 
The other values of $\gamma$ and $\kappa$ may be determined with respect to $\rho$. For instance, if we choose $\rho\in (1,2)$, we can take $\kappa =1$ and $\gamma=1/2$.
\paragraph{Non adaptive prior}
The non adaptive case is much more simpler since it is sufficient to fix 
$$
\mu_s = 1-2\alpha = 2/(2s+2)
$$
and $\zeta = 0$ to obtain suitable calibrations for $\alpha,\beta,\kappa$ and $\gamma$. This achieves the proof.
\end{proof}

We now slightly discuss on the proof of Theorem \ref{theo:posterior_shift2}. We follow exactly the same strategy but the neighbourhoods are a little bit modified on the $g$ coordinate. In this view, we establish a useful bound which concerns the closeness of two laws $\PP_{f,g}$ and $\PP_{f,\tilde{g}}$, when we keep the same shape $f \in \cH_1$. In a sense, the next proposition replaces  Lemma \ref{lemma:mixture_lemma},  Proposition \ref{prop:finite_mixture} and Proposition \ref{prop:Hellinger_neighbourhood} in the special framework of smooth densities.

Consider the inverse functions of the distribution functions defined by
\[ \forall u \in [0, 1], \quad G^{-1}(u) = \inf \{t \in (0, 1] : g((0, t]) > u \}. \]
The Wasserstein (or Kantorovich) distance $W_1$ is given by
\[ W_1(g, \tilde{g}) := \int_0^1 \left| G^{-1}(t) - \tilde{G}^{-1}(u) \right| d t. \]
\begin{proposition}\label{prop:transport}
 Consider $f \in \cH_1$, and let $g$ and $\tilde{g}$ be any measures on $(0,1]$. Then
 \begin{align*}
  d_{TV}(\PP_{f,g},\PP_{f,\tilde{g}}) &\leq \sqrt{2} \pi \|f\|_{\cH_1} W_1(g, \tilde{g}) \\ &
  \leq \sqrt{2} \pi \|f\|_{\cH_1} d_{TV}(g, \tilde{g}) 
  \leq  \pi \|f\|_{\cH_1} \| g - \tilde{g} \|/\sqrt{2}.
 \end{align*}
\end{proposition}

\begin{proof}[Proof of Theorem \ref{theo:posterior_shift2}]
We mimic the proof of Theorem 2.2.

\paragraph{Complementary of the sieve} 
 First, we consider the sieve over $\mathcal{P}$ defined as the set of all possible laws when $f$ has truncated Fourier coefficients and a restricted $L^2$ norm:
$$
\cP_{k_n,w_n} := \left\{\PP_{f,g} \, \vert \, (f ,g) \in \cF^{k_n} \times \mathfrak{M}_{\nu}([0,1])(2A),
 \|f\| \leq w_n\right\},
$$
where $k_n$ is a sequence such that $k_n \longmapsto + \infty$ as $n \longmapsto + \infty$, and $w_n^2= 4 k_n+2$.

Since our sieve is included in the set of all mixture laws, we can apply  Proposition 3.10 and get
$$
\Pi_n \left(\cP \setminus \cP_{k_n,w_n}\right)  \leq  e^{ 
- c [ k_n^2 \log^\rho(k_n) \wedge k_n \xikn^{-2}]}.
$$

\paragraph{Entropy estimates} Since $\mathfrak{M}_{\nu}([0,1])(2A) \subset \M$, our sieve is included in the sieve considered above, we also deduce that for any sequence $\epsilon \longmapsto 0$:
$$
\log D\left( \epsilon_n, \cP_{k_n,w_n},d_H \right) \lesssim  k_n^2 \left[ \log k_n + \log \frac{1}{\epsilon_n} \right].
$$

\paragraph{Lower bound of the prior of Kullback neigbourhoods}
We use the description of Kullback neigbourhoods based on our preliminary results .  We define $\tilde{\epsilon}_n = c\epsilon_n \left( \log  \frac{1}{\epsilon_n}\right)^{-1}$, an integer
 $\ell_n$ such that $ \tilde{\epsilon}_n^{-1/s} \lesssim  \ell_n \lesssim \tilde{\epsilon}_n^{-1/s}$, and the sets
$$
\cF_{\tilde{\epsilon}_n} := \left\{ f \in \cH_s^{\ell_n} : \|f-f^0_{\ell_n}\| \leq \tilde{\epsilon}_n^2 \right\},
$$
and 
$$
\cG_{\tilde{\epsilon}_n} := \left\{ g \in \mathfrak{M}_{\nu}([0,1])(2A): d_{TV}(g,g^0) \leq \tilde{\epsilon}_n \right\}.
$$
We deduce from Lemma \ref{lemma:dVT_sur_f}, Proposition \ref{prop:transport} and arguments of Proposition 3.9  that as soon as $f \in \cF_{\tilde{\epsilon}_n}$ and $g \in \cG_{\tilde{\epsilon}_n}$, $\PP_{f,g}$ belongs to an $\epsilon_n$ Kullback neighbourhood of $\PP_{f^0,g^0}$.
From Proposition 3.9, we can use the following lower bound of the prior mass on $\cF_{\tilde{\epsilon}_n}$:
$$
\Pi_n \left( \cF_{\tilde{\epsilon}_n} \right) \geq  e^{-(c +o(1))\, \left[ \epsilon_n^{-2/s} \left(\log (1/\epsilon_n)\right)^{\rho+2/s} \vee \xikn^{-2} \right] }.
$$
According to Theorem \ref{theo:lower_bound_proba} given in the appendix, we know that 
$$
\Pi_n \left( \cG_{\tilde{\epsilon}_n} \right) \geq e^{-(c+o(1)) \tilde{\epsilon}_n^{-1/(k_\nu+1/2)}} \geq  e^{-(c+o(1)) \tilde{\epsilon}_n^{-\frac{1}{\nu}}}$$
since $k_{\nu}+1/2 \leq \nu$ (see also \cite{Li_Shao} for a very complete survey on the small ball probability estimation for Gaussian processes).

\paragraph{Contraction Rate}
We now find a suitable choice of $k_n$ and $\epsilon_n$ in order to satisfy Theorem 2.1 of \cite{GGvdW00}, \textit{i.e.}
$$
\Pi_n\left( \cG_{\tilde{\epsilon}_n} \right) \Pi_n\left( \cF_{\tilde{\epsilon}_n} \right) \geq e^{- C n \epsilon_n^2}
$$
$$
\log D\left( \epsilon_n, \cP_{k_n,w_n},d_H \right) \lesssim   n \epsilon_n^2
$$
$$
 \Pi_n \left(\cP \setminus \cP_{k_n,w_n}\right)  \leq e^{-(C+4)n \epsilon_n^2}.
$$
Following the arguments already developed in Theorem 2.2, we can find $\gamma>0$ and $\kappa>0$ such that 
$$
\epsilon_n := n^{- \left[ \frac{\nu}{2\nu+1} \wedge \frac{s}{2s+2} \wedge \frac{3}{8}\right] } \log (n) ^{\kappa}, \qquad k_n = n^{\frac{1}{2} - \left[\frac{\nu}{2\nu+1} \wedge \frac{s}{2s+2} \wedge \frac{3}{8}\right]} \log(n)^{\gamma}.
$$
\end{proof}

\section{Identifiability and semiparametric results} \label{sec:identifiability}

In the Shape Invariant Model, an important issue is the identifiability of the model with respect to the unknown curve $f$ and the unknown mixture law $g$. We first discuss on a quite generic identifiability condition for $\PP_{f,g}$. Then, we deduce from the previous section a contraction rate of the posterior distribution around the true $f^0$ and $g^0$.

\subsection{Identifiability of the model}

In previous works on SIM, the identifiability of the model is generally given according to a restriction on the support of $g$. For instance, \cite{BG10} assume the support of $g$ to be an interval included in $[-1/4,1/4]$ (their shapes are defined on $[-1/2;1/2]$ instead of $[0,1]$ in our paper) and $g$ is assumed to have $0$ mean although $f$ is supposed to have a non vanishing first Fourier coefficient ($\theta_1(f) \neq 0$). The same kind of condition on the support of $g$ is also assumed in \cite{BG12}.

If the condition on the first harmonic on $f$ is imperative to obtain identifiability of $g$, the restriction on its support size seems artificial and we detail in the sequel how one can avoid such a hypothesis.
First, we recall that for any curve $Y$ sampled from the SIM, the first Fourier coefficient is given by $\theta_1(Y) = \theta^0_1 e^{-\i 2 \pi \tau} + \xi$ (here $\theta_1^0 = \theta_1(f^0)$). Hence, up to  a simple change of variable in $\tau$, we can always modify $g$ in $\tilde{g}$ such that 
$\theta_1^0 \in \R_{+}$. 
It is for instance sufficient to fix $\tilde{g}(\varphi) = g(\varphi+\alpha)$ where $\alpha$ is the complex argument of $\theta^0_1$. Hence, to impose such an identifiability condition, we have chosen to restrict $f$ to $\cF_s$. 
This condition is not restrictive up to a change of measure for the random variable $\tau$. The next result provides an identifiability criterion for the model, both for $f$ and $g$.

\begin{theo}\label{theo:ident}
Assume that $f \in \cF_s$ defined above and $g \in \M^{\star}$ defined by
$$
\M^{\star}:= \left\{ g \in \M \, \vert \, \exists (c,C) \in \R^{\star}_{+} : \forall k \in \Z, c  < |k|^{\nu} |\theta_k(g)| < C \right\},
$$
where $\nu> 1$, then the Shape Invariant Model described by (\ref{eq:model_fourier}) is identifiable.
\end{theo}

\begin{proof}
The demonstration of Theorem \ref{theo:ident} is decomposed using three hierarchical steps. First, we prove that if $ \PP_{f,g} = \PP_{\tilde{f},\tilde{g}}$, then one has necessarily $\theta_1(f) = \theta_1(\tilde{f})$. Then we deduce from this point that $g = \tilde{g}$ and at last we obtain the identifiability for all other Fourier coefficients of $f$.

Note that as soon as $\nu>1/2$, $g$ and $\tilde{g}$ admit densities with respect to the Lebesgue measure on $[0, 1]$. In the sequel we use the same notation $g$ to refer to the density of $g$.

\paragraph{Point 1: Identifiability on $\theta_0(f)$ and $\theta_1(f)$}
We denote $\PP^k_{f,g}$ the marginal law of $\PP_{f,g}$ on the $k^{th}$  Fourier coefficient when the curve follows the Shape Invariant Model (\ref{eq:model_fourier}). Of course, we have the following implications
$$
 d_{TV}(\PP_{f,g},\PP_{\tilde{f},\tilde{g}}) = 0 \Longrightarrow \left( \PP_{f,g} = \PP_{\tilde{f},\tilde{g}} \right) \Longrightarrow  \forall k \in \Z : d_{TV}(\PP^k_{f,g},\PP^k_{\tilde{f},\tilde{g}}) = 0.
$$
We immediately obtain that $\theta_0(f) = \theta_0(\tilde{f})$ since $\theta_0(f)$ (resp. $\theta_0(\tilde{f})$) represents the mean of the distribution $\PP^0_{f,g}$ (resp. $\PP^0_{\tilde{f},\tilde{g}}$). But note that the distribution $\PP^0_{f,g}$ does not bring any information on the measure $g$, and is not helpful for its identifiability.
Concerning now the first Fourier coefficient, we use the notation $\theta_1:=\theta_1(f)$, $\tilde{\theta}_1: = \theta_1(\tilde{f})$ and remark that
\begin{multline*}
 d_{TV}\left(\PP^1_{f,g},\PP^1_{\tilde{f},\tilde{g}}\right) \\ = \frac{1}{2\pi}
\int_{\C} \left|\int_{0}^1 e^{-|\theta_1 e^{\i 2 \pi \alpha} - z|^2} g(\alpha) d\alpha - \int_{0}^1 e^{- |\tilde{\theta}_1 e^{\i 2 \pi \alpha} - z|^2 } \tilde{g}(\alpha) d\alpha  \right| dz.
\end{multline*}
Assume now that $\tilde{\theta}_1 \neq \theta_1$, without loss of generality $\tilde{\theta}_1 > \theta_1 >0$ and consider the disk  $D_{\C}\left(0,\frac{\tilde{\theta}_1-\theta_1}{2}\right)$, we then get
$
\forall z \in D_{\C}\left(0,\frac{\tilde{\theta}_1-\theta_1}{2}\right)\,, \forall \alpha \in [0,1]:$
$$  |\theta_1 e^{i 2 \pi \alpha} - z | < \frac{\tilde{\theta}_1+\theta_1}{2} \, \text{and} \, |\tilde{\theta}_1 e^{i 2 \pi \alpha} - z | > \frac{\tilde{\theta}_1+\theta_1}{2} . $$
Hence, for all $z \in D_{\C}\left(0,\frac{\tilde{\theta}_1-\theta_1}{2}\right)$, we get 
$\int_{0}^1 e^{-|\theta_1 e^{\i 2 \pi \alpha} - z|^2} g(\alpha) d\alpha > e^{- \frac{\left|\tilde{\theta}_1+\theta_1\right|^2}{4} } $
and of course 
$\int_{0}^1 e^{-|\tilde{\theta}_1 e^{\i 2 \pi \alpha} - z|^2} \tilde{g}(\alpha) d\alpha < e^{- \frac{\left|\tilde{\theta}_1+\theta_1\right|^2}{4} }.
$
We can thus write the following lower bound of the Total Variation\footnote{It is indeed possible to write an explicit lower bound which will depend on $|\theta_1 - \tilde{\theta}_1|^2$, with a radius smaller than $\frac{\tilde{\theta}_1 - \theta_1}{2}$.}. 
\begin{multline*}
 d_{TV}\left(\PP^1_{f,g},\PP^1_{\tilde{f},\tilde{g}}\right) \geq \frac{1}{2\pi}\int_{D_{\C}\left(0,\frac{\tilde{\theta}_1-\theta_1}{2}\right)} \left|\int_{0}^1 e^{-|\theta_1 e^{\i 2 \pi \alpha} - z|^2} g(\alpha) d\alpha \right. \\ 
 - \left. \int_{0}^1 e^{- |\tilde{\theta}_1 e^{\i 2 \pi \alpha} - z|^2 } \tilde{g}(\alpha) d\alpha\right|dz > 0.
\end{multline*}
In the opposite,  $d_{TV}(\PP^1_{f,g},\PP^1_{\tilde{f},\tilde{g}})  = 0$ implies that $\theta_1=\tilde{\theta}_1$ since $f$ and $\tilde{f}$ belong to $\cF_s(A)$.

\paragraph{Point 2: Identifiability on $g$}
We still assume that  $d_{TV}(\PP^1_{f,g},\PP^1_{\tilde{f},\tilde{g}})  = 0$. We know that  $\theta_1=\tilde{\theta}_1$ and we want to infer that $g= \tilde{g}$. We are going to establish this result using only the first harmonic of the curves.
Using a polar change of variables $z=\rho e^{\i \varphi}$, we can write that 
\begin{multline*}
 d_{TV}\left(\PP^1_{f,g},\PP^1_{\tilde{f},\tilde{g}}\right) \\
 \begin{aligned}
  &= \frac{1}{2\pi} \int_{\C} e^{-[\theta_1^2 + |z|^2]} \left| \int_{0}^1 e^{2 \re (z \theta_1 e^{\i 2 \pi \alpha})} (g(\alpha)-\tilde{g}(\alpha) d\alpha \right| dz \\
 & = \frac{1}{4 \pi^2}\int_{0}^{+ \infty} \rho e^{-[\theta_1^2 + \rho^2]} \int_{0}^{2 \pi} \left| \int_{0}^{2 \pi} e^{2 \rho \theta_1 \cos(u - \varphi)}  (g-\tilde{g})(u/2\pi) du\right| d\varphi d\rho \\
 & = \frac{1}{4 \pi^2}\int_{0}^{+ \infty} \rho e^{- [\theta_1^2 + \rho^2]} \int_{0}^{2 \pi} \left| \int_{0}^{2 \pi} e^{2 \rho \theta_1 \cos(u )}  (g-\tilde{g})\left(\frac{u + \varphi}{2\pi}\right) d\alpha\right| d\varphi d\rho  \\
 & =   \frac{1}{4 \pi^2}\int_{0}^{+ \infty} \rho e^{- [\theta_1^2 + \rho^2]} \int_{0}^{2 \pi} \left| \psi_{2 \rho \theta_1}(\varphi) \right| d \varphi d\rho.
 \end{aligned}
\end{multline*}
 In the expression above, we denote $h = g-\tilde{g}$ and $\psi_{a}(\varphi)$ is defined as  
  $$
  \psi_{a}(\varphi) = \int_{0}^{2 \pi} e^{a  \cos (u)} h\left(\frac{u + \varphi}{2\pi}\right) d u.
  $$
  Of course, $\psi_a$ is upper bounded by $4 \pi e ^a$, and a very rough inequality yields\footnote{Such an inequality is not very sharp and we can instead use an argument based on the Laplace transform of $g$ and $\tilde{g}$. The main advantage of such inequality is to handle $L^2$ norms instead of $L^1$ ones.}
  $|    \psi_{a}(\varphi) |\geq \frac{|\psi_{a}(\varphi) |^2}{4 \pi e^a}$. Hence, 
  \begin{equation}\label{eq:dvt_g_psi}
  d_{TV}(\PP^1_{f,g},\PP^1_{\tilde{f},\tilde{g}}) \geq \frac{1}{8\pi^2} \int_{0}^{+ \infty} \rho e^{- (\theta_1^2 + \rho^2 + 2 \theta_1 \rho)} \|\psi_{2 \rho \theta_1}\|^2 d \rho.
  \end{equation}
Using the fact that $\nu > 1$, $h$ may be expanded in Fourier series since  $h \in \cL^2([0,1])$:
 $$
 h(x) = \sum_{n \in \Z} c_n(h) e^{i 2 \pi n x},
  $$
and we can also obtain the Fourier decomposition of  $\psi_a$:
  $$
  \psi_a(\varphi)= \sum_{n \in \Z} c_n(h) \int_{0}^{2 \pi} e^{a \cos(u)} e^{i n u} du\, e^{i 2 \pi n \varphi}.
  $$
Thus, the $L^2$ norm of $\psi_a$ is given by
  \begin{equation}\label{eq:norm2psi}
  \|\psi_{a}\|^2 = \sum_{n \in \Z} |c_n(h)|^2 \left| \int_{0}^{2 \pi} e^{a \cos(u)} e^{\i nu} du\right|^2,
  \end{equation}
 Now, if we denote the first and second kind of Tchebychev polynomials $(T_n)_{n \in \Z}$ and $(U_n)_{n \in \Z}$ which satisfy $T_n(\cos \theta) = \cos (n \theta)$ and $(\sin \theta) U_n(\cos \theta) = \sin (n \theta)$, we can decompose
 \begin{multline*}
 \int_{0}^{2 \pi} e^{a \cos(u)} e^{\i nu} du \\
 \begin{aligned}
  &= \int_{0}^{2 \pi} e^{a \cos(u)} \left[ T_n(\cos u) + \i (\sin u) U_n(\cos u)\right] d u \\
  & = \int_{0}^{2 \pi} \sum_{k \geq 0} \frac{a^k (\cos u)^k}{k!} \left[ T_n(\cos u) + \i (\sin u) \sum_{j=0}^{n} \beta_j (\cos u)^j\right] d u
 \end{aligned}
\end{multline*}
where we have used the analytic expression of $U_n$ given by
$$U_n(\cos u) = \sum_{j=0}^{E((n-1)/2)} (-1)^j C_{n}^{2j+1} (\cos u)^{n-2j-1} (1 - \cos^2 u)^j.$$
 Hence, we obtain
\begin{align*}
 \int_{0}^{2 \pi} e^{a \cos(u)} e^{\i nu} du &=  \int_{0}^{2 \pi}  \sum_{k \geq 0} \frac{a^k (\cos u)^k}{k!}  T_n(\cos u)  d u  \\ &\quad +  \i \sum_{k \geq 0}\sum_{j=0}^n \beta_j \frac{a^k}{k!}  \int_{0}^{2 \pi}   \sin u (\cos u)^{k+j} d u \\
 & =  \int_{0}^{2 \pi}  \sum_{k \geq 0} \frac{a^k (\cos u)^k}{k!}  T_n(\cos u)  d u \\
 &  =  \int_{0}^{2 \pi} e^{a \cos(u)}  \cos (n u) d u \in \R \quad \text{if}\quad a\in\R.
 \end{align*}
We denote $A_n$ the following (holomorphic) function of the variable $a$ as
$$
  A_n(a):= \int_{0}^{2 \pi} e^{a \cos(u)} \cos(n u) du,$$
  and equation (\ref{eq:norm2psi}) yields
  \begin{equation}\label{eq:psi2normexplicit}
  \|\psi_a\|^2 = \sum_{n \in \Z} |c_n(h)|^2 A_n(a)^2.
  \end{equation}
Moreover, for each $n$,  $A_n$ is not the null function, otherwise it would be the case for each of its derivative but remark that $(\cos u)^n$ may be decomposed in the basis $(T_k)$ and using successive derivations
\begin{align*}
 A_n^{(n)}(0) &=\frac{d^{(n)}}{da^{(n)}} \left[ \sum_{k=0}^\infty \frac{a^k}{k!} \int_0^{2\pi} (\cos u)^k \cos(n u) \, du \right] (0) \\
 &=\int_{0}^{2 \pi} (\cos u)^n T_n(\cos u) d u \\
 & =  \int_{0}^{2 \pi} \left[ \sum_{k=0}^{n-1} \alpha_k T_k(\cos u) + 2^{1-n} T_n(\cos u) \right] T_n(\cos u) d u \\
 &= 2^{1-n} \pi >0.
\end{align*}
Note that in the meantime, we also obtain that $A_n^{(j)}(0)=0, \forall j < n$, so that
\begin{equation}\label{eq:Ansize}
A_n(a) \sim_{a \mapsto 0} \frac{2^{1-n} \pi}{n!} a^n.
\end{equation}
We can conclude the proof of the identifiability of $g$ using \eqref{eq:psi2normexplicit} in \eqref{eq:dvt_g_psi} to obtain
$$
  d_{TV}(\PP^1_{f,g},\PP^1_{\tilde{f},\tilde{g}}) \geq \frac{1}{8\pi^2} 
 \sum_{n \in \Z} |c_n(h)|^2  \underbrace{\left(\int_{0}^{+ \infty} \rho e^{- [\theta_1+\rho]^2} A_n(2 \rho \theta_1)^2 d \rho\right)}_{:=I_n(\theta_1)}.
  $$
  From  \eqref{eq:Ansize}, we can deduce that each integral $I_n(\theta_1) \neq 0, \forall n \in \Z$ and we then conclude that: 
  $$ d_{TV}(\PP^1_{f,g},\PP^1_{\tilde{f},\tilde{g}})  \Longleftrightarrow g=\tilde{g} \qquad \text{et} \qquad \theta_1=\tilde{\theta}_1.$$
  
  \paragraph{Point 3: Identifiability on $f$} 
  We end the argument and prove that $\PP_{f,g}  = \PP_{\tilde{f},\tilde{g}}$ implies $f=\tilde{f}$. We already know that $g=\tilde{g}$ and it remains to establish the equality for all the Fourier coefficients whose frequency is different from $0$ and $1$. By a similar argument as the one used for the identifiability of $\theta_1$ (Point 1), we can easily show that 
$$   
d_{TV}(\PP^k_{f,g},\PP^k_{\tilde{f},\tilde{g}})=0 \Longrightarrow |\theta_k| = |\tilde{\theta}_k|.
$$
But we cannot directly conclude here since it is not reasonable to restrict the phase of each others coefficients $\theta_k(f)$ to a special value (as it is the case for $\theta_1(f)$ which is positive). 
We assume that  $\tilde{\theta}_k= \theta_k e^{\i \varphi}$.  Since $g=\tilde{g}$, we have
$$   
d_{TV}(\PP^k_{f,g},\PP^k_{\tilde{f},g})=\frac{1}{2\pi} \int_{\C} \underbraceabs{\int_{0}^{2\pi} e^{- | z - \theta_k e^{- \i k \alpha} |^2} - e^{-| z - \theta_k e^{\i (\varphi- k\alpha)}|^2} g(\alpha) d\alpha }{:=F(z)} d z.$$
Now, if one considers $z=x+\i y$, $F$ is differentiable with respect to $x$ and $y$ and $F(0)=0$. A simple computation of $\nabla F(0)$ shows that $\nabla F(0)$ is the  vector (written in the complex form)
$$
\nabla F(0) = \theta_k e^{- |\theta_k|^2} c_k(g) [1-e^{\i \varphi}].
$$
Since $g \in \M^{\star}$,  this last term is non vanishing except if $\theta_k=0$ (which trivially implies that $\tilde{\theta}_k=0=\theta_k$) or if $\varphi\equiv 0 (2\pi)$. In both cases, $F'(0) = 0 \Longleftrightarrow  \tilde{\theta}_k = \theta_k$. Thus, as soon as
$\theta_k \neq \tilde{\theta}_k$, we have $\nabla F(0) \neq 0$ and we may find a neighbourhood of $0$ denoted $B(0,r)$ such that $|F|(z) >0$ when $z \in B(0,r)\setminus \{0\}$ . 
  This is sufficient to end the proof of identifiability.
\end{proof}


In a sense, the main difficulty of the proof above is the implication of $d_{TV}(\PP^1_{f,g},\PP^1_{\tilde{f},\tilde{g}})  \Longrightarrow g=\tilde{g}$. Then, the identifiability follows using a chaining argument $\theta_1(f) \rightarrow g \rightarrow \theta_k(f), \forall k \notin \{0,1\}$. We will see that this part of the proof can also be used to obtain a contraction rate for $f$ and $g$ around $f^0$ and $g^0$. We recall here the main inequality used above: $\forall \theta_1>0$ and $ \forall (g,\tilde{g}) \in \Mnu$, the identifiability on $g$ is traduced by

\begin{equation}\label{eq:main_g}
d_{TV} \left(\PP^1_{\theta_1,g},\PP^1_{\theta_1,\tilde{g}} \right) 
\geq \frac{1}{8 \pi^2} 
\sum_{n\in\Z} |c_n(g-\tilde{g})|^2 
\left( \int_{0}^{\infty} \rho e^{-(\rho+\theta_1)^2} A_n(2 \rho \theta_1)^2 d \rho \right)
\end{equation}
The aim of the next paragraph is to exploit this inequality to produce a contraction rate of $g$ aroung $g^0$.

\subsection[Posterior contraction rate around f0 and g0]{Contraction rate of the posterior distribution around $f^0$ and $g^0$}

\subsubsection{Link with deconvolution with unknown variance operator}
We  provide in this section an upper bound on the contraction rate of the posterior law around $f^0$ and $g^0$. This question is somewhat natural owing to the identifiability result obtained in the previous section. We thus assume for the rest of the paper that $f \in \cF_s$ and $g \in \Mnu$ for some parameters $s \geq  1$ and $\nu > 1$.

Remark first that our problem written in the Fourier domain seems strongly related to the standard deconvolution with unknown variance setting. For instance, the first observable Fourier coefficients are
$$
\theta_1(Y_j) = \theta_1 e^{-\i 2 \pi \tau_j} + \epsilon_{1,j}, \forall j \in \{1 \ldots n\}
$$
and up to a division by $\theta_1$, it can also be parametrised as
\begin{equation}\label{eq:deconvolution_unknown}
\tilde{\theta}_1(Y_j) = e^{-\i 2 \pi \tau_j} + \frac{\epsilon_{1,j}}{\theta_1}, \forall j \in \{1 \ldots n\},
\end{equation}
which is very similar to the problem $Y=X+\epsilon$ studied for instance by \cite{M02} where $\epsilon$ follows a Gaussian law whose variance (here $1/\theta_1^2$) is unknown. As pointed in \cite{M02} (see also the more recent work \cite{BM05} where similar situations are extensively detailed), such a particular setting is rather unfavourable for statistical estimation since  convergence rates are generally of logarithmic order. Such a phenomenon also occurs in our setting, except for the first Fourier coefficient of $f$ as pointed in the next proposition.

The roadmap of this paragraph is similar to the proof of Theorem \ref{theo:ident}. We first provide a simple lower bound of $d_{TV}$ which enables to conclude for the first Fourier coefficient. Then, we still use the first marginal to compute a contraction rate for the posterior distribution on $g$ around $g^0$. At last, we chain all these results to provide a contraction rate for the posterior distribution on $f$ around $f^0$.

\subsubsection{Contraction rate on the first Fourier coefficient}

\sloppy
\begin{proposition}\label{prop:theta1}
Assume that $(f,g) \in \cF_s \times \Mnu$, then the posterior distribution satisfies
$$
\Pi_n \left( \left.\theta_1 \in B\left(\theta_1^0,M \epsilon_n^{1/3}\right)^c \right\vert Y_1, \ldots, Y_n \right) \mapsto 0
$$ in $\PP_{f^0,g^0}$ probability as $n \rightarrow + \infty$ for a sufficiently large $M$. The contraction rate around the true Fourier coefficient is thus at least $n^{-1/3 \times [\nu/(2\nu+1) \wedge s /(2s+2) \wedge 3/8]} (\log n)^{1/3}$.
\end{proposition}
\fussy

\begin{proof}
The demonstration is quite simple. Remark that using the beginning of the proof of Theorem \ref{theo:ident}, one can show that for any $\theta_1$ such that $0<\eta<|\theta_1-\theta_1^0| <\theta_1^0/2$, one can bound, for any $g \in \Mnu$, the Total Variation distance between $\PP_{f,g}$ and $\PP_{f^0,g^0}$. Remark that
$$
d_{TV}\left(\PP_{f,g},\PP_{f^0,g^0} \right) \geq d_{TV}\left(\PP^1_{f,g},\PP^1_{f^0,g^0} \right),
$$
owing to the restriction of $\PP_{f,g}$ to the first Fourier marginal and the variational definition of the Total Variation distance. Then 
\begin{multline*}
 d_{TV}\left(\PP^1_{f,g},\PP^1_{f^0,g^0} \right) \\
 \begin{aligned}
  &\geq \frac{1}{2\pi} \int_{B\left(0,\frac{|\theta_1-\theta_1^0|}{4}\right)} \left|\int_{0}^1 g(\alpha) e^{-|z-\theta_1e^{\i 2 \pi \varphi}|^2} -  g^0(\alpha) e^{-|z-\theta^0_1e^{\i 2 \pi \varphi}|^2} d \varphi \right| d z \\
  & \geq \frac{\eta^2}{32}\left| e^{-(3 \theta_1^0+\theta_1)^2/16} - e^{-(3 \theta_1+\theta^0_1)^2/16}\right| \geq C(\theta_1^0) \eta^3,
 \end{aligned}
\end{multline*}
for a suitable small enough constant $C(\theta_1^0)$. Now, one can use simple inclusions and Pinsker inequality 
\begin{multline*}
 \left\{ \theta_1 \in B(0,\eta)^c \right\} \subset \left\{ \theta_1 \vert d_{TV}(\PP_{f,g},\PP_{f^0,g^0}) \geq C(\theta_1^0) \eta^3 \right\} \\ \subset 
\left\{ \theta_1 \vert d_{H}(\PP_{f,g},\PP_{f^0,g^0}) \geq C(\theta_1^0) \eta^3 \right\}.
\end{multline*}
The proof is now achieved according to Theorem \ref{theo:posterior_shift}.
\end{proof}

\subsubsection[Posterior contraction rate around g0]{Posterior contraction rate around $g^0$}

We now study the contraction rate of the posterior distribution around the true mixture law $g^0$. This result is stated below.

\begin{theo}\label{theo:contraction_g}
Assume $(f^0,g^0)\in \cF_s \times \Mnu$, then
$$
\Pi_n\left( \left. g : \|g-g^0\|^2 > M \log^{- 2 \nu}(n) \right\vert Y_1, \ldots, Y_n \right) \longrightarrow 0
$$ in $\PP_{f^0,g^0}$ probability as $n \rightarrow + \infty$ for a sufficiently large $M$.
\end{theo}
\begin{proof}
We first restrict ourselves to the first marginal on Fourier coefficient as before. Using Theorem 2.2, we know that
$$
\Pi_n \left\{ \PP_{f,g} : d_H(\PP_{f,g} ,\PP_{f^0,g^0}) \geq M \epsilon_n \vert Y_1, \ldots Y_n \right\} \rightarrow 0 \quad \text{as}  \quad  n \rightarrow + \infty. 
$$
Since $d_{TV}(\PP^1_{\theta_1,g} ,\PP^1_{\theta_1^0,g^0}) 
\leq d_H(\PP_{f,g} ,\PP_{f^0,g^0})$, we then get 
\begin{equation}\label{eq:tronc0}
\Pi_n \left\{ \PP_{f,g} : d_{TV}(\PP^1_{\theta_1,g} ,\PP^1_{\theta_1^0,g^0}) \geq M \epsilon_n \vert Y_1, \ldots Y_n \right\} \rightarrow 0 \quad \text{as} \quad n \rightarrow + \infty. \end{equation}

For any $g \in \Mnu$, the triangular inequality yields
\begin{equation}\label{eq:tronc1}
 d_{TV}\left(\PP^1_{\theta^0_1,g},\PP^1_{\theta_1,g}\right) + 
d_{TV} \left(\PP^1_{\theta_1,g},\PP^1_{\theta_1^0,g^0}\right) \geq d_{TV} \left(\PP^1_{\theta^0_1,g},\PP^1_{\theta_1^0,g^0}\right) .
\end{equation}

Now, let $\tilde{f}$ be defined by $\theta_1(\tilde{f}) = \theta_1(f)$, and for any $k\in\Z\backslash\{1\}$, $\theta_k(\tilde{f}) = \theta_k(f^0)$. Then Lemma \ref{lemma:dVT_sur_f} yields
\[ d_{TV}\left(\PP^1_{\theta^0_1,g},\PP^1_{\theta_1,g}\right) = d_{TV}\left(\PP^1_{\tilde{f},g},\PP^1_{f^0,g}\right) \leq \frac{\|\tilde{f}-f^0\|}{\sqrt{2}} = \frac{|\theta_1-\theta_1^0|}{\sqrt{2}}. \]
Therefore
\begin{multline}\label{eq:tronc2}
\Pi_n \left( \left. \PP_{f,g} \, s.t. \, d_{TV}\left(\PP^1_{\theta^0_1,g},\PP^1_{\theta_1,g} \right) \leq \frac{M}{\sqrt{2}} \epsilon_n^{1/3} \right\vert Y_1, \ldots, Y_n\right) \\ 
\geq \Pi_n \left(\left. \PP_{f,g} \, s.t. \, |\theta_1-\theta_1^0| \leq M \epsilon_n^{1/3}\right\vert Y_1, \ldots, Y_n \right) \longrightarrow 1 
\end{multline}
as $n \rightarrow + \infty$. 
In conclusion, we deduce from \eqref{eq:tronc0},\eqref{eq:tronc1} and \eqref{eq:tronc2} that for $M$ large enough:
$$\Pi_n \left(\left. \PP_{f,g} \, s.t. \, d_{TV}\left(\PP^1_{\theta^0_1,g},\PP^1_{\theta^0_1,g^0}\right) \leq M \epsilon_n^{1/3}\right\vert Y_1, \ldots, Y_n \right) \longrightarrow 1 \quad \text{as} \quad n \rightarrow + \infty.
$$

We then use equation \eqref{eq:main_g} applied with $\theta_1=\theta_1^0$ and the last equation to obtain our rate of consistency. Remark that
  \begin{equation}\label{eq:dvtlower}
  d_{TV}(\PP^1_{\theta^0_1,g},\PP^1_{\theta^0_1,g^0}) \geq \frac{1}{8\pi^2 } 
 \sum_{n \in \Z} |c_n(g-g^0)|^2  \int_{0}^{+ \infty} \rho e^{- (\rho+\theta_1^0)^2} A_n(2 \rho \theta_1^0)^2 d \rho,
 \end{equation}
where we have used the definition
$$
A_n(a) = \int_{0}^{2\pi}e^{a \cos(u)} \cos (nu) d u.
$$
Now, we use equivalents given by Lemma \ref{lemma:equi_bessel} detailed in Appendix \ref{sec:bessel_appendix}. We only keep the integral of $A_n$ for $a \in [0,c\sqrt{n}]$ since it can be shown that the tail of such integral will yield neglictible term 
We just use the equivalent given by \eqref{eq:equi2}.
One can find a sufficiently small constant $\kappa$ such that
\begin{multline*}
 \int_{0}^{+ \infty} \rho e^{- (\rho+\theta_1^0)^2} A_n(2 \rho \theta_1^0)^2 d \rho \\
 \begin{aligned}
  & \geq \int_{0}^{\frac{\sqrt{n}}{2 \theta_1^0}}  \frac{4 \pi^2 \rho^{2n+1} \{\theta_1^0\}^{2n}}{n!^2}e^{-(\rho+\theta_1^0)^2} \left(1-\kappa\frac{ [2 \rho \theta_1^0]}{n}\right)^2 d \rho \\
  & \geq \left(1-\frac{ \kappa}{\sqrt{n}}\right)^2 \frac{4 \pi^2 \{\theta_1^0\}^{2n}}{n!^2} e^{-\big(\theta_1^0+\frac{\sqrt{n}}{2\theta_1^0}\big)^2}  \int_{0}^{\frac{\sqrt{n}}{2\theta_1^0}} \rho^{2n+1} d \rho
 \end{aligned}
\end{multline*}
Now, we can apply the Stirling formula to obtain:
\begin{multline*}
 \frac{4 \pi^2 \{\theta_1^0\}^{2n}}{n!^2} e^{-\big(\theta_1^0+\frac{\sqrt{n}}{2\theta_1^0}\big)^2}  \int_{0}^{\frac{\sqrt{n}}{2\theta_1^0}} \rho^{2n+1} d \rho \\
 \begin{aligned}
  &\sim \frac{4 \pi^2 \{\theta_1^0\}^{2n}}{(n/e)^{2n} 2 \pi n}e^{-\big(\theta_1^0+\frac{\sqrt{n}}{2\theta_1^0}\big)^2} \frac{\left( \sqrt{n}/(2 \theta_1^0)\right)^{2n+2}}{2n+2} \\
  & \sim  \frac{2 \pi}{n(2n+2)} e^{ - 2n \log \left[ \frac{n}{e \theta_1^0}\right]-\big(\theta_1^0+\frac{\sqrt{n}}{2\theta_1^0}\big)^2+(n+1)\log \left[\frac{n}{4 \{\theta_1^0\}^2}\right]}.
 \end{aligned}
\end{multline*}
Hence, this last term is lower bounded by $C(\theta_1^0) e^{- n \log(n)}$.
  As a consequence, we can plug such lower bound in \eqref{eq:dvtlower} to get
$$
d_{TV}(\PP^1_{\theta^0_1,g},\PP^1_{\theta^0_1,g^0}) \geq c 
 \sum_{k \in \Z} |c_k(g-\tilde{g})|^2  e^{- c k \log k}.
$$
for $c$ sufficiently small.
We now  end the proof of the Theorem: choose a frequency cut-off $k_n$ that depends on $n$ and remark that
\begin{align*}
\forall g \in \Mnu \quad 
\|g-g^0\|^2 &= \sum_{|\ell| \leq k_n} |c_{\ell}(g-g^0)|^2 + \sum_{|\ell|>k_n} |c_{\ell}(g-g^0)|^2  \\
& \lesssim e^{c k_n \log k_n} \sum_{|\ell| \leq k_n} |c_{\ell}(g-g^0)|^2 e^{-c \ell \log \ell} +  k_n^{-2\nu} \\ 
& \lesssim e^{c k_n \log k_n}d_{TV}(\PP^1_{\theta^0_1,g},\PP^1_{\theta^0_1,g^0}) + k_n^{-2\nu}.
\end{align*}
We know from Equation \eqref{eq:tronc2} that the last bound is smaller than 
$e^{c k_n \log k_n} \epsilon_n^{1/3} + k_n^{-2\nu}$ up to a multiplicative constant, with probability close to $1$ as $n$ goes to $+ \infty$. The optimal choice for $k_n$ yields
$$
[k_n +2\nu] \log k_n = \frac{1}{3} \log \frac{1}{\epsilon_n}.
$$
This thus ensures that
$$
\Pi_n \left\{g 
\, s.t.\,  \|g-g^0\|^2 \leq  M \log (n) ^{-2 \nu} \vert Y_1, \ldots Y_n \right\} \longrightarrow 1 \, \text{as}  \,  n \rightarrow + \infty. 
$$
\end{proof}

\subsubsection[Posterior contraction rate around f0]{Posterior contraction rate around $f^0$}

We then aim to obtain a similar result for the posterior weight on neighbourhoods of $f^0$. Even if our results are quite good for the first coefficient $\theta_1$, we will see that indeed,  this is far from being the case for the rest of its Fourier expansion. 

\begin{theo}\label{theo:contractionf}
 Assume $(f^0,g^0) \in \cF_s \times \Mnu$ and 
 \[\exists (c) > 0 \quad \exists \beta > \nu+\tfrac{1}{2} \quad \forall k \in \mathbb{Z} \qquad |\theta_k(g^0)| \geq c k^{-\beta}, \]
 then
$$
\Pi_n \left( f  : \|f-f^0\|^2 > M \left( \log n \right)^{-2s \times \frac{2\nu}{2s+2\nu+1}}  \vert Y_1, \ldots, Y_n \right) \longrightarrow 0
$$
in $\PP_{f^0,g^0}$ probability as $n \rightarrow + \infty$, for a sufficiently large $M$.
\end{theo}

\begin{proof}
The idea of the proof is very similar to the former used arguments, we aim to study the posterior weight on neighbourhoods of the true Fourier coefficients of $f^0$, whose frequency is larger than $1$.

\paragraph{Point 1: Triangular inequality}
For any $f \in \cF_s$, we have for any $k \in \Z$:
$$d_{TV}(\PP^k_{f,g^0},\PP^k_{f^0,g^0})  \leq d_{TV}(\PP^k_{f,g^0},\PP^k_{f,g})  + d_{TV}(\PP^k_{f^0,g^0},\PP^k_{f,g}).
$$
The second term does not exceed $\epsilon_n \ll\log(n) ^{-\nu} $ with a probability tending to $1$, more precisely
\begin{equation}\label{eq:bound_dvt_triangle1}
\Pi_n \left( \left. 
\forall k \in \mathbb{Z} \quad  d_{TV}(\PP^k_{f,g},\PP^k_{f^0,g^0}) < M \epsilon_n \right\vert Y_1, \ldots, Y_n\right)  \longrightarrow 1
\end{equation}
as $n \longrightarrow +\infty$.

\paragraph{Point 2: $\displaystyle
\Pi_n \left( \left. 
\sup_{ k \in \mathbb{Z}} d_{TV}(\PP^k_{f,g^0},\PP^k_{f,g}) < M \log(n)^{- \nu} \right\vert Y_1, \ldots, Y_n \right) 
\rightarrow
1$}

To obtain such a limit, we can  use first the Cauchy-Schwarz inequality as follows:
\begin{eqnarray*}
d_{TV}(\PP^k_{f,g^0},\PP^k_{f,g}) & = & \frac{1}{2\pi} \int_{\mathbb{C}} \left|\int_{0}^{2 \pi}
e^{-|z-\theta_k e^{\i k \varphi}|^2} [g(\varphi) - g^0(\varphi)] d \varphi \right| dz\\
& \leq & \frac{\|g-g^0\|}{2\pi} \int_{\mathbb{C}} \left[\int_{0}^{2 \pi}
e^{-2 |z-\theta_k e^{\i k \varphi}|^2} d \varphi \right] ^{1/2}dz
\end{eqnarray*}
Now, the Young inequality implies that for any $M>0$,
$$
|z-\theta_k e^{\i k \varphi}|^2 = |z|^2 + |\theta_k|^2 - 2 \Re \left( \bar{z} \theta_k e^{\i k \varphi}\right) \geq |z|^2 \left(1-\frac{1}{M}\right)- |\theta_k|^2(M-1)
$$
and the choice $M=2$ yields
\begin{equation}\label{eq:bound_dvt_triangle3}
d_{TV}(\PP^k_{f,g^0},\PP^k_{f,g})  \leq \frac{\|g-g^0\|}{2\pi} \int_{\mathbb{C}}
 \left( e^{-|z|^2 + 2|\theta_k|^2} \right)^{1/2} dz \leq 
 \|g-g^0\| \frac{e^{|\theta_k|^2}}{2}.
\end{equation}

To obtain that the former term is bounded, we first establish that indeed the posterior distribution asymptotically only weights functions $f$ with bounded Fourier coefficients.
We hence denote $$\mathcal{A}_n = \{(f,g) : \exists k \in \mathbb{Z} \quad  d_{TV}(\PP^k_{f,g^0},\PP^k_{f,g})  \geq M \log (n)^{-\nu}  \}$$ and the two sets
$$\mathcal{B} = \{f  : \forall k \in \mathbb{Z} \quad |\theta_k| \leq |\theta_k^0|+M \log(n)^{-\nu}\}$$ and 
$$\mathcal{C} = \{f  : \forall k \in \mathbb{Z} \quad |\theta_k^0| \leq |\theta_k|+M \log(n)^{-\nu}\}.$$
We first consider an integer $k$ and $\theta_k$ such that $|\theta_k|>|\theta_k^0|+M\log(n)^{-\nu}$, then
$$
d_{TV}(\PP^k_{f^0,g^0},\PP^k_{f,g})  = \frac{1}{2\pi} \int_{\mathbb{C}} \left|\int_{0}^{2\pi}\left[
e^{-|z-\theta_k e^{\i k \varphi}|^2} g(\varphi) - e^{-|z-\theta^0_k e^{\i k \varphi}|^2} g^0(\varphi)\right] d \varphi \right| dz.
$$
For any $z$ in the centered complex ball $B_{n}=B\left(0,\frac{M \log(n)^{-\nu}}{3}\right)$, one has for any $\varphi \in [0,2\pi]$
\begin{eqnarray*}
 |z-\theta^0_k e^{\i k \varphi}| \leq   \frac{M \log(n)^{-\nu}}{3}+|\theta_k^0|& \leq& 2 \frac{M \log(n)^{-\nu}}{3}+|\theta_k^0|\\
 &  \leq& |\theta_k| -\frac{M \log(n)^{-\nu}}{3}\leq |z-\theta_k e^{\i k \varphi}|.
\end{eqnarray*}
Hence if $|\theta_k| \geq |\theta_k^0|+M\log(n)^{-\nu}$, one has
\begin{multline*}
 d_{TV}(\PP^k_{f^0,g^0},\PP^k_{f,g}) \\
 \begin{aligned}
  &\geq \frac{1}{2\pi} \int_{B_{n}} \left|\int_{0}^{2\pi} \left[e^{-|z-\theta_k e^{\i k \varphi}|^2} g(\varphi) - e^{-|z-\theta^0_k e^{\i k \varphi}|^2} g^0(\varphi) \right]d \varphi \right| dz\\
  & \geq \frac{1}{2\pi} \int_{B_{n}} e^{-[|\theta_k^0|+M\log(n)^{-\nu}/3]^2} - e^{-[|\theta_k^0|+2M\log(n)^{-\nu}/3]^2} dz \\
  & \geq c |\theta_k^0|^2 e^{-|\theta_k^0|^2} \log(n)^{-3\nu},
 \end{aligned}
\end{multline*}
for a sufficiently small absolute constant $c>0$. Since the sequence $(\theta_k^0)_{k \in \mathbb{Z}}$ is bounded, for $n$ large enough, we know that 
$\|\theta^0\|^2 e^{-\inf_{k}|\theta_k^0|^2} \log(n)^{-3\nu} \gg \epsilon_n$. We can deduce from \eqref{eq:bound_dvt_triangle1} that

\begin{equation}\label{eq:bound_dvt_triangle2}
\Pi_n \left(\mathcal{B}^c  \vert Y_1, \ldots, Y_n\right)
   \longrightarrow 0 \quad  \text{as} \quad n \longrightarrow +\infty.
\end{equation}
A similar argument yields
$$\Pi_n \left(\mathcal{C}^c  \vert Y_1, \ldots, Y_n\right)
   \longrightarrow 0 \quad  \text{as} \quad n \longrightarrow +\infty.
$$
Gathering now \eqref{eq:bound_dvt_triangle2} and \eqref{eq:bound_dvt_triangle3}, 
we get for a sufficiently large $M$ 
\begin{eqnarray*}
\Pi_n \left( \mathcal{A}_n    \vert Y_1, \ldots, Y_n\right) &= &
\Pi_n \left( \mathcal{A}_n  \cap \mathcal{B} \cap \mathcal{C} \vert Y_1, \ldots, Y_n\right)  \\
& & 
+ \Pi_n \left( \mathcal{A}_n  \cap (\mathcal{B} \cap \mathcal{C})^c \vert Y_1, \ldots, Y_n\right)  \\
& \leq & \Pi_n \left( \|g-g^0\| \geq M e^{-\left(1+\sup_{k} |\theta_k^0|^2\right)}\log (n)^{-\nu} \right) \\
& &
+  \Pi_n \left(  \mathcal{B}^c \vert Y_1, \ldots, Y_n\right)  +  \Pi_n \left(  \mathcal{C}^c \vert Y_1, \ldots, Y_n\right)  
\end{eqnarray*}
We can now apply Theorem \ref{theo:contraction_g} to obtain the desired result:
\begin{equation}\label{eq:dvt_impt}
\Pi_n \left(\left. \sup_{k\in\Z} d_{TV}(\PP^k_{f,g^0},\PP^k_{f,g}) < M \log(n)^{- \nu} \right\vert Y_1, \ldots, Y_n \right)  \longrightarrow 1 \quad  \text{as} \quad n \longrightarrow +\infty.
\end{equation}

\paragraph{Point 3: Contraction of $\theta_k$ near $\theta_k^0$} 
From the arguments of Point 2, we see that
$$ \Pi_n \left( f: \forall k \in \mathbb{Z} \quad \left||\theta_k| - |\theta_k^0| \right| <  M \log(n)^{- \nu} \right)  \longrightarrow 1 \quad  \text{as} \quad n \longrightarrow +\infty.$$
We now study the situation when $\left||\theta_k| - |\theta_k^0| \right| < M \log(n)^{-\nu}$, and we can write $\theta_k=\theta_k^0 e^{\i \varphi} + \xi_n$ where $\xi_n$ is a complex number such that $|\xi_n| \leq M \log(n)^{-\nu}$.
$$   
d_{TV}(\PP^k_{f,g^0},\PP^k_{f^0,g^0})=\frac{1}{2\pi} \int_{\C} \underbraceabs{\int_{0}^{2\pi} \left[ e^{-| z - \theta_k e^{i k \alpha} |^2} - e^{- | z - \theta_k^0 e^{i k \alpha}|^2}\right]  g^0(\alpha) d\alpha }{:=F(z)} dz$$
Indeed, $F(0) \simeq0$ since a Taylor expansion near $0$ yields at first order in $z$ and $\xi_n$ that
\begin{eqnarray*}
F(z)& = 	&2 e^{-|\theta_k^0|^2} \int_{0}^{2 \pi}\left[1+ \re \left(z \bar{\theta}_k e^{- i k \alpha} \right) -(1+ \re \left(z \bar{\theta}^0_k e^{- i k \alpha} \right) \right] g^0(\alpha)d \alpha\\ & &  + o(|z|)+\cO(|\xi_n|).
\end{eqnarray*}
If one uses now $\theta_k = \theta_k^0 e^{\i \varphi} + \mathcal{O}(\log(n)^{-\nu})$, the computation of the integral above yields for $c<2$ and $\eta$ small enough such that $|z|\leq \eta$:
$$
|F(z)| \geq  c e^{-|\theta_k^0|^2}  \left|\sin (\varphi/2) \re \left(z \i e^{\i \varphi/2} \bar{\theta}_k^0 c_{-k}(g^0) \right) \right| + \mathcal{O}(\log(n)^{-\nu})
$$
Now, denote $\bar{u} = \frac{\i e^{\i \varphi/2} \bar{\theta}_k^0  c_{-k}(g^0)}{|\theta_k^0| \times |c_{-k}(g^0)|}$ which is a complex number of norm $1$, and let $v= \bar{u} e^{\i \pi/2}$. The vector $v$ is orthogonal to $\bar{u}$ and $z$ may be decomposed as
$$
z = a \bar{u} + b v.
$$
We then choose  $|b|<|a|/2$ and denotes $\mathcal{R}_a$ the area where $z$ is living. For $a < \eta$ small enough, we obtain that there exists an absolute constant $c$ independent of $k$ such that
\begin{eqnarray*}
d_{TV}(\PP^k_{f,g^0},\PP^k_{f^0,g^0}) &\geq  &\int_{\mathcal{R}_a}|F(z)| \geq  c \eta^3 e^{- |\theta_k^0|^2} |\sin(\varphi/2)| |\bar{\theta}_k^0| | c_{-k}(g^0)|\\ &&+ \mathcal{O}\left(\log(n)^{-\nu}\right).
\end{eqnarray*}
 Since 
$|\theta_k-\theta_k^0| =  2  |\sin(\varphi/2)| |\theta_k^0|+\mathcal{O}\left(\log(n)^{-\nu}\right)$, we  get that :
\begin{equation}\label{eq:mino_dvt2}
d_{TV}(\PP^k_{f,g^0},\PP^k_{f^0,g^0}) \geq c \eta^3 e^{- |\theta_k^0|^2} | c_{-k}(g^0)| |\theta_k-\theta_k^0|+\mathcal{O}\left(\log(n)^{-\nu}\right).
\end{equation}
Thus, we can conclude using \eqref{eq:dvt_impt} and \eqref{eq:mino_dvt2}  that there exists a sufficiently large $M$ such that, as $n \longrightarrow +\infty$,
\begin{equation}\label{eq:contraction_thetak}
\Pi_n \left(\left. f : \sup_{k \in \mathbb{Z}} \left|(\theta_k-\theta_k^0) c_{-k}(g^0)\right|  < M \log(n)^{- \nu} \right\vert Y_1, \ldots, Y_n \right)  \longrightarrow 1.
\end{equation}

\paragraph{Point 4: Contraction on $f^0$} 

We can now produce a  very similar proof to the one used at the end of Theorem \ref{theo:contraction_g}:
\begin{eqnarray*}
\|f-f^0\|^2&  =  &\sum_{|\ell| > k_n} |\theta_{\ell}-\theta_{\ell}^0|^2 + \sum_{|\ell| \leq k_n} |\theta_{\ell}-\theta_{\ell}^0|^2\\
& \lesssim & k_n^{-2s} + \sum_{|\ell|\leq k_n} 
\frac{|\theta_{\ell}-\theta_{\ell}^0|^2|c_{-\ell}(g^0)|^2 }{ |c_{-\ell}(g^0)|^2}\\
& \lesssim & k_n^{-2s} +k_n^{2 \beta}
 \sum_{|\ell|\leq k_n} 
|\theta_{\ell}-\theta_{\ell}^0|^2|c_{-\ell}(g^0)|^2 \\
& \lesssim &  k_n^{-2s} +k_n^{2 \beta+1} \sup_{|\ell| \leq  k_n } |\theta_{\ell}-\theta_{\ell}^0|^2|c_{-\ell}(g^0)|^2 
\end{eqnarray*}
Hence,  \eqref{eq:contraction_thetak} implies
$$
\Pi_n \left(\left. f : \|f-f^0\|^2 \leq k_n^{-2s} +k_n^{2 \beta+1} \log(n)^{-2\nu} \right\vert Y_1, \ldots, Y_n\right)  \longrightarrow 1 \,  \text{as} \, n \longrightarrow +\infty.
$$
The optimal choice of the frequency cut-off  is $k_n = \left( \log n\right)^{\frac{2\nu}{2 \beta + 2s+1}}$, which yields
$$
\Pi_n \left(\left. f : \|f-f^0\|^2 \leq M \left(\log n \right)^{-4s\nu/(2s+2\beta+1)} \right\vert Y_1, \ldots, Y_n \right) \longrightarrow 1 \, \text{as} \, n \rightarrow + \infty. 
$$
This last result is the desired inequality.
\end{proof}

\begin{remark}
The lower bound obtained on $d_{TV}(\PP^k_{f,g^0},\PP^k_{f^0,g^0})$ will be important to understand how one should build an appropriate net of functions $(f_j,g_j) \in \cF_s  \times \Mnu$ hard to distinguish according to the $L^2$ distance. 
When $|\theta_k| \neq |\theta_k^0|$, it is quite easy to distinguish the two hypotheses but it is far from being the case when their modulus is equal. In such a case, the behaviour of the Fourier coefficients of $g^0$ becomes important. 
This is a clue to exhibit an efficient lower bound through the Fano lemma (for instance). This is detailed in the next paragraph. 
\end{remark}

\subsection{Lower bound from a frequentist point of view\label{sec:lowerbound}}

We complete now our study of the Shape Invariant Model by a small investigation on how one could obtain some lower bounds in the frequentist paradigm. We could consider several methods. Among them, the first one could be the use of results in the literature, such as the works of \cite{M02} or \cite{BM05}. Indeed, in the convolution model with unknown variance
\begin{equation}\label{eq:matias}
Y_i = X_i+\epsilon_i, \forall i \in \{1 \ldots n\} \qquad (X_i)_{i=1 \ldots n} \sim g,\end{equation}
 we already know that one cannot beat some $\log n$ power for the convergence rate of any estimator of both $g$ and of the variance of the noise $\sigma^2$. Such a nice result is obtained using the so-called  van Trees  inequality  which is a Bayesian Cramer-Rao bound (see for instance \cite{GL95} for further details). 
However their result cannot be used here:  Proposition \ref{prop:theta1} is much more optimistic since we obtain there a polynomial rate for the posterior contraction around $\theta_1^0$.
  
 First, note the results given by \cite{M02} and Proposition \ref{prop:theta1} are not opposite. Indeed, \cite{M02} considers lower bounds in a larger class than the estimation problem of $\theta_1$ written as \eqref{eq:deconvolution_unknown}: from a minimax point of view, the supremum over all hypotheses is taken in a somewhat larger set than ours. 
Moreover, if one considers \eqref{eq:deconvolution_unknown}, the density of $e^{- \i 2 \pi \tau_j}$ is supported by $\mathbb{S}^1$ instead of the whole complex plane which would be a natural extension of \eqref{eq:matias}. 
Hence, $g$ is a singular measure with respect to the noise measure: the ability of going beyond the logarithmic convergence rates is certainly due to this degeneracy nature of our problem according to the Gaussian complex noise. It is an important structural information which is not available when one considers general problems such as \eqref{eq:matias}.

Following such considerations, we are thus driven to build some nets of hypotheses hard to distinguish between and then apply some classical tools for lower bound results. 
We have chosen to use the Fano Lemma (see \cite{ibragimov_book} for instance) instead of Le Cam's method, since we will only be able to find some \textit{discrete} (instead of convex) set of pairs $(f_j,g_j)$ in $\cF_s \times \Mnu$ closed according to the Total Variation distance. We first recall the version of the Fano Lemma we used.
\begin{lemma}[Fano's Lemma]\label{lemma:fano} Let $r \geq 2$ be an integer and $\cM_r  \subset \cP $ which contains $r$ probability distributions indexed by $j=1 \ldots r$ such that
$$
\forall j \neq j' \qquad d(\theta(P_j),\theta(P_{j'}) \geq \alpha_r,
$$
and 
$$
d_{KL}(P_j,P_{j'}) \leq \beta_r.
$$
Then, for any estimator $\hat{\theta}$, the following lower bound holds
$$
\max_{j} \EE_{j} \left[d(\hat{\theta},\theta(P_j)) \right] \geq \frac{\alpha_r}{2} \left(1 - \frac{\beta_r + \log 2}{\log r} \right).
$$
\end{lemma}

We derive now our lower bounds result.

\begin{theo} \label{theo:lowerbound}
There exists a sufficiently small $c$ such that the minimax rates of estimation over $\cF_s \times \Mnu$ satisfy
$$
\liminf_{n \rightarrow + \infty} \left( \log n\right)^{2s+2}
\inf_{\hat{f} \in \cF_s}  \sup_{(f,g) \in \cF_s \times \Mnu} \|\hat{f} - f\|^2 \geq c,
$$
and
$$
\liminf_{n \rightarrow + \infty} \left( \log n\right)^{2\nu + 1}
\inf_{\hat{g} \in \cF_s}  \sup_{(f,g) \in \cF_s \times \Mnu} \|\hat{g} - g\|^2 \geq c .
$$
\end{theo}

\begin{proof} 
We will adapt the Fano Lemma to our setting and we are looking for a set $(f_j,g_j)_{j=1 \ldots p_n}$ such that each $\PP_{f_j,g_j}$ are closed together with rather different functional parameters $f_j$ or $g_j$. Reading carefully the Bayesian contraction rate is informative to build $p_n$ hypotheses which are difficult to distinguish. First, we know that since each $f_j$ should belong to $\cF_s$, we must impose for any $f_j$  that $\theta_1(f_j)>0$. 
From Proposition \ref{prop:theta1}, we know that one can easily distinguish two laws $\PP_{f_j,g_j}$ and $\PP_{f_{j'},g_{j'}}$ as soon as $\theta_1(f_j) \neq \theta_1(f_{j'})$. Then we build our net using a common choice for the first Fourier coefficient of each $f_j$ in our net. For instance, we impose that
$$
\forall j \in \{1 \ldots p_n\} \qquad \theta_1(f_j) = 1.
$$

\paragraph{Point 1: Net of functions $(f_j)_{j = 1\ldots p_n}$}
We choose the following construction
\begin{equation}\label{eq:def_net}
\forall j \in \{1 \ldots p_n\}  \qquad f_j (x) = e^{\i 2\pi x} + p^{-s}_n e^{\i 2 \frac{(j-1)}{p_n} \pi} e^{\i 2 \pi p_n x}.
\end{equation}
The number of elements in the net $p_n$ will be adjusted in the sequel and will grow to $+ \infty$. Note that our construction naturally satisfies that each net $(f_j)_{j =1 \ldots p_n}$ belongs to $\cF_s$ since the modulus of the $p_n$-th Fourier coefficient is of size $p_n^{-s}$. At last, we have the following rather trivial inequality: $
\forall (j,j') \in \{1 \ldots p_n\}^2$ 
$$  \|f_j - f_{j'}\|^2 \geq p_n^{-2s} \times \left| e^{\i 2 \pi/p_n} - 1\right|^{2} \geq 4 p_n^{-2s} \sin^2 ( \pi/p_n) \sim_{n \mapsto +\infty} 4\pi^2 p_n^{-2s-2}.
$$

\paragraph{Point 2: Net of measures $(g_j)_{j = 1\ldots p_n}$}
The core of the lower bound is how to adjust the measures of the random shifts to make the distributions $\PP_{f_j,g_j}$, $j = 1 \ldots p_n$, as close as possible. First, remark that the Fano Lemma \ref{lemma:fano} is formulated with entropy between laws although it is quite difficult to handle when dealing with mixtures. In the sequel, we will choose to still use the Total Variation distance, and then use the chain of inequalities: $\forall j \neq j'$
\begin{eqnarray*}
d_{TV} \left( \PP_{f_j,g_j} , \PP_{f_{j'},g_{j'}}\right)  \leq \eta & \Rightarrow& d_{H} \left( \PP_{f_j,g_j} , \PP_{f_{j'},g_{j'}}\right) \leq \sqrt{2\eta}\\ &\Rightarrow & d_{KL} \left( \PP_{f_j,g_j} , \PP_{f_{j'},g_{j'}}\right) \lesssim \sqrt{\eta} \log \frac{1}{\eta}.
\end{eqnarray*}
Hence, from the tensorisation of the entropy, we must find a net such that $d_{TV} \left( \PP_{f_j,g_j} , \PP_{f_{j'},g_{j'}}\right)  \leq \eta_n$ with $-\sqrt{\eta_n} \log \eta_n =\cO( 1/n)$ to obtain a tractable application of the Fano Lemma (in which $P_j = \PP_{f_j,g_j}^{\otimes n}$). It imposes to find some mixture laws such that 
$d_{TV} \left( \PP_{f_j,g_j} , \PP_{f_{j'},g_{j'}}\right) \lesssim \frac{1}{(n \log n)^2}$. From the triangular inequality, it is sufficient to build $(g_j)_{j = 1 \ldots p_n}$ satisfying
\begin{equation}\label{eq:condTV}
\forall j \in \{1 \ldots p_n\} \qquad 
d_{TV} \left( \PP_{f_j,g_j} , \PP_{f_1,g_1}\right) \lesssim \frac{1}{(n \log n)^2}.
\end{equation}
For sake of convenience, we will omit the dependence of $p_n$ on $n$ and simplify the notation to $p$. In a similar way, $\theta_p^j $ will denote the $p$-th Fourier coefficient of $f_j$ given by
$\theta_p^j = e^{\i 2 \pi \alpha_j } \theta_p^1$ where $\alpha_j=\frac{j-1}{p_n}$.
From our choice of $(f_j)_{j = 1 \ldots p_n}$ given by \eqref{eq:def_net}, we have
\begin{align*}
d_{TV} \left( \PP_{f_j,g_j} , \PP_{f_1,g_1}\right) &= \frac{1}{2\pi^2} \int_{\C \times \C} \left| 
\int_{0}^{1}
e^{-|z_1 - e^{\i 2 \pi \varphi}|^2 - |z_2 - e^{\i 2 \pi p \varphi}\theta_p^1|^2}  g_1(\varphi)   d \varphi
\right. \\ 
& \left.
 - \int_{0}^{1} e^{-|z_1 - e^{\i 2 \pi \varphi}|^2 - |z_2 - e^{\i 2 \pi p \varphi}\theta_p^j|^2}  g_j(\varphi) d \varphi \right| dz_1 dz_2
\end{align*}

We will use the smoothness of Gaussian densities to obtain a suitable upper bound. 
Call $F$ the function defined on $\R^4$ by
$$
 F(x_1,y_1,x_2,y_2): = \int_{0}^{1} \left( e^{-\|z-\theta^1 \bullet \varphi\|^2} g_1(\varphi) - e^{-\|z-\theta^j \bullet \varphi\|^2}  g_j(\varphi) \right) d \varphi,
$$
where $z=(x_1+\i y_1,x_2+\i y_2)$ and $\theta^j \bullet \varphi=(e^{\i 2\pi \varphi}, \theta_j^p e^{\i 2\pi p \varphi})$.

To control $F$, we adapt the proof of Proposition 3.3. We use a truncature for $(x_1,x_2,y_1,y_2) \in\mathcal{R}_{R_n} :=  B_{\R^2}(0,R_n)^2$. Outside $\mathcal{R}_{R_n}$, we use the key inequality (that comes from a Taylor's expansion):
\begin{equation*}
\forall k \in \N \quad 
\forall y \in \R_+ \qquad\underbraceabs{e^{-y} - \sum_{j=0}^{k-1} \frac{(-y)^j}{j!}}{:=R_k(y)}
\leq \frac{|y|^k}{k!} \leq \frac{(e |y|)^k}{k^k}.
\end{equation*}
Inside $\mathcal{R}_{R_n}$ we need to satisfy some constraints on the Fourier coefficients. 
Since here the only non null Fourier coefficients are of order $1$ and $p$, we have to ensure that
\begin{equation} \label{eq:liens_coefs_gj}
 \forall m, l \leq d \quad  \forall (s,\tilde{s}) \in \{-1; +1\}^2 \qquad c_{s m+\tilde{s} \ell p}(g_j) e^{\tilde{s} \ell  \alpha_j} = 
c_{s m+\tilde{s} \ell p}(g_1).
\end{equation}
Hence, the maximum size of $d$ is $d=p/4$. We have
\begin{align*}
 d_{TV} \left( \PP_{f_j,g_j} , \PP_{f_1,g_1}\right) &= \frac{1}{2\pi^2} \int_{\mathcal{R}_{R_n}} |F(x_1,y_1,x_2,y_2) | dx_1  dy_1  dx_2  dy_2   \\ & \quad 
 + \frac{1}{2\pi^2} \int_{\mathcal{R}^c_{R_n}} |F(x_1,y_1,x_2,y_2) | dx_1  dy_1  dx_2  dy_2 \\
  & \lesssim e^{- R_n^2/2} + \left(\frac{(e R_n)^{p/4}}{(p/4)^{p/4}}\right)^{4} \lesssim e^{- R_n^2/2}  + \frac{(e R_n)^p}{(p/4)^p}.
\end{align*}
We choose now $R_n$ such as $R_n := 3 \sqrt{ \log n}$ to obtain that $e^{-R_n^2/2} \ll (n \log n)^{-2}$ as required in condition \eqref{eq:condTV}. Now, we control the last term of the last inequality: the Stirling formula yields
$$
\frac{(e R_n)^p}{(p/4)^p} \lesssim e^{p \log (3 \sqrt{\log n}) - p \log p/4}. $$
 If one chooses $p_n = \kappa \log n$ with $\kappa>12$, we then obtain that
$$ d_{TV} \left( \PP_{f_j,g_j} , \PP_{f_1,g_1}\right) \lesssim  e^{-C p_n \log p_n} \lesssim (n \log n)^{-2}.$$
Such a choice of $R_n$ and $p_n$ ensures that
\eqref{eq:condTV} is fulfilled.

We have to make sure that our conditions \eqref{eq:liens_coefs_gj} for the Fourier coefficients of the $g_j$'s lead to valid densities. Take for instance, for some $\beta > \nu + 1/2$,
\[ a = \frac{A}{\left(2\sum_{k\geq 1} k^{-2\beta+2\nu}\right)^{1/2}} \wedge \frac{1}{\left(2\sum_{k\geq 1} k^{-2\beta}\right)^{1/2}}. \]
Then take $c_0(g_j) :=1$, and $\forall k \in \Z^{\star}$, $c_k(g_1) := a |k|^{-\beta}$. This ensures that
\[ \sum_{k \in \Z^{\star}} |c_k(g_j)| \leq 1, \] and therefore all $g_j$ remains nonnegative. \\
Note that the densities $g_j$ fulfill the condition appearing in Theorem \ref{theo:semiparametric_rates}; the lower bounds below are also valid in this slightly smaller model.

We then conclude our proof: we aim to apply the Fano Lemma (see Lemma \ref{lemma:fano}) with $\alpha_n=p_n^{-2s-2}$ and $\beta_n = 
\cO(1)$ for the parametrization of $(f_j)_{j=1 \ldots p_n}$. We then deduce the first lower bound
$$
\liminf_{n \rightarrow + \infty} \left( \log n\right)^{2s+2}
\inf_{\hat{f} \in \cF_s}  \sup_{(f,g) \in \cF_s \times \Mnu} \|\hat{f} - f\|^2 \geq c.
$$

Our construction implies also that each $g_j$ are rather different each others since one has for instance,
$c_p(g_j) e^{\i \alpha_j} = c_p(g_1) = c_p(g_{j'}) e^{\i \alpha_{j'}}$. Thus
$$\forall j \neq j'  \qquad \|g_j-g_{j'}\|^2_2 \geq |c_p(g_j) - c_p(g_{j'})|^2 = p^{-2\nu} \left| e^{\alpha_j } - e^{\alpha_{j'} }\right|^2  \geq c p^{-2\nu-2}.
$$
Applying the Fano Lemma to $(g_j)_{j=1\ldots p_n}$ we get
$$
\liminf_{n \rightarrow + \infty} \left( \log n\right)^{2\nu+2}
\inf_{\hat{g} \in \cF_s}  \sup_{(f,g) \in \cF_s \times \Mnu} \|\hat{g} - g\|^2 \geq c.
$$
This ends the proof of the lower bound. 
\end{proof}

\section{Concluding remarks} \label{sec:conclusion}

In this paper, we exhibit a suitable prior which enable to obtain a contraction rate of the posterior distribution near the true underlying distribution $\mathbb{P}_{f^0,g^0}$. Moreover, this rate is polynomial with the number $n$ of observations, even if our SIM is an inverse problem with unknown operator of translation which depends on $g$.
From a technical point of view,  the keystones of such results are the tight link between the white noise model and the Fourier expansion as well as the smoothness of Gaussian law which permits to obtain an efficient covering strategy.

Up to non restrictive condition, we can also obtain a large identifiability class but in this class, the contraction  of the posterior is dramatically damaged since we then obtain a logarithm rate instead of a polynomial one. This last point cannot be so much improve using the standard $L^2$ distance to measure the neighbourhoods of $f^0$ as pointed by our last lower bound. Remark that we do not obtain exactly the same rates for our lower and upper bounds of reconstruction. This may be due to the rough inequality  
$|    \psi_{a}(\varphi) |\geq \frac{|\psi_{a}(\varphi) |^2}{\|\psi_{a}\|_{\infty}}$ used to obtain
\eqref{eq:dvt_g_psi}, it may be the reason why we do not obtain optimal rates.

Indeed, the degradation of the contraction rate occurs when one tries to invert the identifiability map $\mathcal{I}: (f,g) \mapsto \mathbb{P}_{f,g}$. Such difficulty should be understood as a novel consequence of the impossibility to exactly recover the random shifts parameters when only $n$ grows to $+\infty$. Such phenomenon is highlighted for instance in several papers such as \cite{BG10} or \cite{BGKM12}.

However, it may be possible to obtain a polynomial rate using a more appropriate distance adapted to our problem of randomly shifted curves as pointed in our main results:
$$
d_{Frechet}(f_1,f_2) := \inf_{\tau \in [0,1]} \|f_1(.-\tau)-f_2(.)\|.
$$ 
We plan to tackle this problem in a future work. The important requirement in this view is to find some relations between the neighbourhoods of $\mathbb{P}_{f^0,g^0}$ and the neighbourhoods of $f^0$ according to the distance $d_{Frechet}$.

An interesting extension would consider the SIM with a noise level $\sigma$ depending on $n$ in the Bayesian framework. This asymptotic setting is linked to the work of \cite{BG12} in which their $J$ curves are sampled at the $n$ points of a discrete design in $[0, 1]$.

At last, an open and challenging question concerns the research of stochastic algorithm to approach the posterior distribution in our non parametric Shape Invariant Model. 
One may think of an adaptation of the SAEM strategy proposed in \cite{AKT} even if this approach is at the moment valid only in a parametric setting.

\appendix

\section{Topology on probability spaces}\label{section:appendix_proba}

\paragraph{Probability distances}

We study consistency using standard distance over probability measures. If $P$ and $Q$ are two probability measures over a set $X$, absolutely continuous with respect to a reference measure $\lambda$, $d_H$ refers to the Hellinger distance defined as
$$
d_H(P,Q) := \sqrt{\int_X  \left[ \sqrt{\frac{dP}{d\lambda}} - \sqrt{\frac{dQ}{d\lambda}}  \right]^2 d\lambda}.
$$
Note that $d_H$ does not depend on the choice of the dominating measure $\lambda$, and that the definition can be extended to any finite measures $P$ and $Q$ in a straightforward way.

When needed, we use the Total Variation distance between two probability measures $P$ and $Q$. If $\cB$ is the 
$\sigma$-algebra of measurable sets with the reference measure $\lambda$, this distance is given by
$$
d_{TV}(P,Q) : = \sup_{A \in \cB} | P(A) - Q(A)| = \frac{1}{2} \int_X \left| \frac{dP}{d\lambda}
- \frac{dQ}{d\lambda} \right| d\lambda.$$
At last, we recall the definition of the Kullback-Leibler divergence (entropy) between $P$ and $Q$ since it is sometimes be used in the work:
$$d_{KL}(P,Q): = \int_{X} - \log \frac{dQ}{dP} dP.$$ In the sequel, we shall also use $V(P,Q)$ defined as a second order moment associated to the Kullback-Leibler divergence
$$
V(P,Q) := \int_{X} \left(\log \frac{dQ}{dP}\right)^2 dP.
$$
It may be reminded  the classical Pinsker's inequality 
\begin{equation}
\sqrt{\frac{1}{2}d_{KL}(P,Q) } \geq d_{TV}(P,Q),
\end{equation}
as well as
\begin{equation}\label{eq:dvt_dh}
\tfrac{1}{2}\, d_H(P,Q)^2 \leq d_{TV}(P,Q) \leq d_H(P,Q).
\end{equation}

\paragraph{Model Complexity}
To obtain the posterior consistency and convergence rate, we shall use results given by Theorem 2.1 of \cite{GGvdW00} which is stated below. This theorem exploits the notion of complexity of the studied model, and this complexity is traduced according to packing or covering numbers. 
For any set of probability measures $\cP$ endowed with a metric $d$, $D(\epsilon,\cP,d)$  refers to the $\epsilon$-packing number (the maximum number of points in $\cP$ such that the minimal distance between each pair is larger than $\epsilon$). 
The $\epsilon$-covering number  $N(\epsilon,\cP,d)$ is the minimum number of balls of radius $\epsilon$ needed to cover $\cP$. These two numbers are linked through the following inequality
$$
N(\epsilon,\cP,d) \leq D(\epsilon,\cP,d) \leq N(\epsilon/2,\cP,d) .
$$
At last, for $d$ a metric on finite measures, an $\epsilon$-bracket is a set of the form
$$
[L, U] := \left\{P \text{ s.t. } \frac{dL}{d\lambda} \leq \frac{dP}{d\lambda} \leq \frac{dU}{d\lambda} \right\},
$$
for $L$ and $U$ two finite measures such that $d(L, U)\leq \epsilon$ and $\lambda$ any dominating measure. The $\epsilon$-bracketing number $N_{[]}(\epsilon,\cP,d)$ is the minimal number of $\epsilon$-brackets 
needed to cover $\cP$. Note that $N_{[]}(\epsilon,\cP,d_H)$ is an upper bound of the $(\epsilon/2)$-covering number $\N(\epsilon/2,\cP,d_H)$. The bracketing entropy is then defined by $H_{[]}(\epsilon,\cP,d) := \log N_{[]}(\epsilon,\cP,d)$.

\section{Tools for the proof of Theorem 2.2 and Theorem 2.3}

\subsection{Entropy estimates}

\begin{proof}[Proof of Proposition \ref{prop:atheta}]
The proof is similar to Lemma 1 of \cite{GW}, we set $p=2\ell+1$ and for any $\epsilon >0$, we are going to build an explicit bracketing of $\cA_{\theta}$ and then bound $N_{[]}(\epsilon,\cA_{\theta},d_H)$.
For 
an integer $K$ which will be chosen in the sequel, we define $[\varphi^i_-,\varphi^i_+]$ of size $\Delta_\varphi = 1/K$, with $\varphi^i_- = (i-1) \Delta_\varphi$ and $\varphi^i_+ = i \Delta_\varphi$. 
For any $\delta>0$, we consider the lower and upper brackets
$$
l_i := (1+\delta)^{-1} \gamma_{\theta\bullet\varphi^i_-,(1+\delta)^{ - \alpha} Id} \qquad \text{and} \qquad u_i := (1+\delta) \gamma_{\theta\bullet\varphi^i_-,(1+\delta)^{ \alpha} Id}.
$$
We are looking for some admissible values of $\alpha$, $\delta$, and $K$ such that the set $([l_i,u_i])_{i = 1 \ldots K}$ is an $\epsilon$-bracket of $\cA_{\theta}$ for $d_H$.
Of course, for all $\varphi \in [\varphi^i_-,\varphi^i_+]$, 
$l_i \leq \gamma_{\theta\bullet\varphi,Id}(.) \leq u_i$ should hold, but we can check that $\forall x \in \C$, 
\[  \frac{l_i(x)}{\gamma_{\theta\bullet\varphi,Id}(x)} \leq \frac{1}{1+\delta} \frac{1}{(1+\delta)^{-p\alpha}} e^{\frac{\|\theta\bullet\varphi -\theta\bullet\varphi_i^-\|^2 }{1-(1+\delta)^{-\alpha}}} \leq (1+\delta)^{p\alpha-1} e^{\frac{4 \pi^2 \Delta_\varphi^2 \|\theta\|_{\cH_1}^2}{1-(1+\delta)^{-\alpha}}}. \]
Hence, we must have $\alpha \leq 1/p$, 
and we must also satisfy
$$
|\Delta_\varphi|^2 \leq \frac{1-p\alpha}{4 \pi^2 \|\theta\|_{\cH_1}^2} \left(1-(1+\delta)^{-\alpha}\right) \log(1+\delta) = \frac{\alpha (1-p\alpha) \delta^2}{4 \pi^2 \|\theta\|_{\cH_1}^2} \left(1 +o(1)\right),
$$
 where $o(1)$ does not depend on $p$ and goes to zero as $\delta\rightarrow 0$ uniformly in $\alpha$ in any positive neighbourhood of zero. 
 In a same way considering $ \gamma_{\theta\bullet\varphi, Id} u_i^{-1}$, we obtain
$$
\forall x \in \C \qquad \frac{\gamma_{\theta\bullet\varphi,Id}(x)}{u_i(x)} \leq (1+\delta)^{\alpha p -1 } e^{ \frac{4 \pi^2 \Delta_\varphi^2 \|\theta\|_{\cH_1}^2}{(1+\delta)^{\alpha}-1}},
$$ and the same conditions arise. In order to minimize the cardinal of the bracketing, $\Delta_\varphi$ must be as large as possible, we then maximize $\alpha (1-p\alpha)$ and choose $\alpha = (2p)^{-1}$.

We must now check that $d_H(l_i,u_i) \leq \epsilon$. Rapid computations show that
$$
d_H(l_i,u_i)^2 = \delta^2 + d_H(\gamma_{\theta\bullet\varphi_-^i,(1+\delta)^{-\alpha}Id}(.),\gamma_{\theta\bullet\varphi_-^i,(1+\delta)^{\alpha}Id}(.))^2.
$$
Using standard formula on Hellinger distance for multivariate gaussian laws,  we obtain
\begin{align*}
d_H(l_i,u_i)^2 &= \delta^2 + 2\left[ 1-\frac{2^p}{\left((1+\delta)^{\alpha}+(1+\delta)^{-\alpha} \right)^p}\right]  \\
&= \delta^2 + 2 \left[ 1-\frac{2^p\sqrt{1+\delta}}{\left( 1+(1+\delta)^{1/p}\right)^p}\right].
\end{align*}
One can easily check that, whatever $p\geq 1$, $\left( 1+(1+\delta)^{1/p}\right)^p \leq 2^p e^{\delta/2}$, which yields
$$
d_H(l_i,u_i)^2 \leq \tfrac{3}{2} \delta^2 + o(\delta^2) \leq 2 \delta^2
$$
for $\delta$ small enough. An admissible  choice of  $\delta$  should be $\delta = \epsilon/\sqrt{2}$, which insures $d_H(l_i,u_i) \leq \epsilon$. We then obtain
$$
\Delta_\varphi^2 \leq \frac{\delta^2 + o(\delta^2)}{16 \pi^2 p \|\theta\|_{\cH_1}^2} = \frac{ \epsilon^2+o(\epsilon^2)}{32 \pi^2 p \|\theta\|_{\cH_1}^2},
$$
where $o(\epsilon^2)$ does not depend on $p$. The number of brackets is now $K=\Delta_\varphi^{-1}$, this ends the proof of the proposition .
\end{proof}

\begin{proof}[Proof of Proposition \ref{prop:Pf}]
 We first fix the notation $p=2 \ell+1$ which refers to the dimension of the multivariate mixture. For any $R>0$ which will be chosen later, $\cE_{R}$ is the ball of in $\C^p$ of radius $R$. For sake of simplicity, we will sometimes omit the dependence on $\epsilon$ with the notation $p$.
 According to the hypotheses in Proposition \ref{prop:Pf}, there exists an absolute constant $a$ such that $\|\theta\| \leq w \leq a \sqrt{p}$. 
We first write
$$
d_{TV}(\PP_{\theta,g},\PP_{\theta,\tilde{g}}) \leq \frac{1}{2} \underbrace{\int_{\cE_R^c } \left| d\PP_{\theta,g} - d\PP_{\theta,\tilde{g}}  \right| (z)}_{:=(A)} +\frac{1}{2} \underbrace{\int_{\cE_R } \left| d\PP_{\theta,g} - d\PP_{\theta,\tilde{g}}  \right| (z)}_{:=(B)}.
$$
Let $\nu$ be a measure on $[0,1]$ that dominates both $g$ and $\tilde{g}$.

\paragraph{Term $(A)$}
We will pick $R$ such that $(A)$ is smaller than $\epsilon/2$, first set $R^2>(1+a)^2 p\geq a^{-2} (1+a)^2 \|\theta\|^2$ and with this choice,
$$
\forall \varphi \in [0,1]  \quad \forall z \in \cE_R^c \qquad \|z-\theta \bullet \varphi\| > \|z\|/(1+a).
$$
This simply implies that, 
\begin{eqnarray*}
(A)
&\leq& \pi^{-p} \int_{\cE_R ^c } \int_{0}^1 e^{-\frac{\|z\|^2}{(1+a)^2}}  \left| \frac{d g}{d \nu}(\varphi) - \frac{d \tilde{g}}{d \nu}(\varphi)\right| d \nu(\varphi) dz \\ & \leq& 2 (1+a)^{2p}\, \PR\left(\chi_{2p}^2 \geq \frac{2 R^2}{(1+a)^2}\right). 
\end{eqnarray*}
To deal with we last term we use a concentration of chi-square statistics inequality (see Lemma 1 of \cite{IL06}): for any $k\geq 1$ and $c>0$,
\begin{equation} \label{eq:chi_square}
 \PR\left( \chi_k^2 \geq (1+c) k \right) \leq \frac{1}{c\sqrt{2\pi}} e^{ -\frac{k}{2} [c -\log(1+c)] - \frac{1}{2} \log k}.
\end{equation}
Therefore, writing $R^2 = (1+a)^2 (1+c) p$ for $c>0$, one gets
\[ (A) \leq  \frac{1}{c\sqrt{\pi}} e^{-p [c - \log (1+c) - 2\log (1+a)] - \frac{1}{2}\log p }  \]
and this term is smaller than $\epsilon/2$ if we pick $c$ large enough, since $\log \frac{1}{\epsilon} \lesssim p$. 

\paragraph{Term $(B)$} \label{term_B}
We then consider $(B)$, following the strategy of \cite{GvdW01} which exploits the smoothness of Gaussian densities. We will exhibit a discrete mixture law which will be close to $\PP_{\theta,g}$, for any given $g$. Taylor's  expansion theorem yields:
\begin{equation}\label{eq:expo}
\forall k \in \N \quad 
\forall y \in \R_+ \qquad\underbraceabs{e^{-y} - \sum_{j=0}^{k-1} \frac{(-y)^j}{j!}}{:=R_k(y)}
\leq \frac{|y|^k}{k!} \leq \frac{(e |y|)^k}{k^k}.
\end{equation}
Thus, for all $z \in \cE_R$, we have
\begin{align*}
\PP_{\theta,g}(z) - \PP_{\theta,\tilde{g}}(z) 
& = \pi^{-p}\int_{0}^1  e^{-\|z-\theta \bullet \varphi\|^2} \left[ \frac{d g}{d \nu}(\varphi) - \frac{d \tilde{g}}{d \nu}(\varphi)\right] d \nu(\varphi) \\
& =  \pi^{-p}\sum_{j=0}^{k-1} \frac{(-1)^j}{j!}\int_{0}^1  \|z-\theta \bullet \varphi\|^{2j} \left[ \frac{d g}{d \nu}(\varphi) - \frac{d \tilde{g}}{d \nu}(\varphi)\right] d \nu(\varphi)  \\ & 
+ \pi^{-p}\int_{0}^1 R_k \left(\|z-\theta \bullet \varphi\|^2 \right)   \left[ \frac{d g}{d \nu}(\varphi) - \frac{d \tilde{g}}{d \nu}(\varphi)\right] d \nu(\varphi). 
\end{align*}
We now decompose  $\theta = (\theta_{-\ell}, \ldots, \theta_{\ell})$ and $z=(z_{-\ell}, \ldots, z_{\ell})$ using polar coordinates: $\theta_m = \rho^{(1)}_m e^{\i \alpha_m}$ and $z_m = \rho^{(2)}_m e^{\i \beta_m}$ for $|m| \leq \ell$. This leads to 
\[ \|z-\theta \bullet \varphi\|^{2} 
 = \|z\|^2 + \|\theta\|^2- 2 \sum_{m=-\ell}^{\ell} \rho^{(1)}_m \rho^{(2)}_m \cos (\beta_m - \alpha_m - m \varphi).
\]
For any integer $j \leq k$, we deduce that
$$
\|z-\theta \bullet \varphi\|^{2j} = C_j(z,\theta) + \sum_{r=1}^j \sum_{m=-\ell}^{\ell} a_{r,m}(z,\theta) \left[\cos (\beta_m - \alpha_m - m \varphi)\right]^r,
$$
where $(a(r,m))_{r = 1 \ldots k, m=-\ell \ldots \ell}$ is a complex matrix which only depends on $z$ and $\theta$. Using Euler's identity,
\begin{align*}
\|z-\theta \bullet \varphi\|^{2j} 
& = C_j(z,\theta) + \sum_{r=-j \ell }^{j \ell }b_{r}(z,\theta) e^{\i r \varphi},
\end{align*}
where $b$ stands for a complex vector obtained by the Binomial formula and coefficients $a_{r,m}(z,\theta)$. Consequently, for all $z \in \cE_R$
\begin{align*}
  \left(\PP_{\theta,g} - \PP_{\theta,\tilde{g}}\right)(z)
  &= \pi^{-p}\sum_{j=0}^{k-1}\frac{(-1)^j}{ j!} \int_{0}^{1} \bigg[ C_j(z,\theta) \\ &\quad + \sum_{r=-j \ell }^{j \ell }b_{r}(z,\theta) e^{\i r \varphi}\bigg] \left[ \frac{d g}{d \nu}(\varphi) - \frac{d \tilde{g}}{d \nu}(\varphi)\right] d \nu(\varphi) \\
  &\quad +\pi^{-p} \int_{0}^1 R_k \left(\|z-\theta \bullet \varphi\|^2\right)   \left[ \frac{d g}{d \nu}(\varphi) - \frac{d \tilde{g}}{d \nu}(\varphi)\right] d \nu(\varphi) \\
  & = \pi^{-p}\sum_{j=0}^{k-1}\frac{(-1)^j}{ j!} \bigg[ C_j(z,\theta) c_0(g-\tilde{g}) +\! \sum^{j \ell }_{r=-j \ell} b_{r}(z,\theta)c_{r}(g-\tilde{g})\bigg]\\ 
  &\quad + \pi^{-p}\int_{0}^1 R_k \left(\|z-\theta \bullet \varphi\|^2 \right) \left[ \frac{d g}{d \nu}(\varphi) - \frac{d \tilde{g}}{d \nu}(\varphi)\right] d \nu(\varphi). 
\end{align*}
Caratheodory's theorem shows that one can find  $\tilde{g}$ with a finite support of size $2 (k-1) \ell+1 \sim 2 k \ell$ such that
$$
\forall r \in [-(k-1) \ell, (k-1) \ell] \qquad c_r(g) = c_r(\tilde{g}).
$$
For such finite mixture law $\tilde{g}$, we obtain $\forall z \in \C^p$, 
\[ \PP_{\theta,g}(z) - \PP_{\theta,\tilde{g}}(z)  = \pi^{-p}\int_{0}^1 R_k \left(\|z-\theta \bullet \varphi\|^2\right) \left[ \frac{d g}{d \nu}(\varphi) - \frac{d \tilde{g}}{d \nu}(\varphi)\right] d \nu(\varphi), \]
and of course
\begin{align*}
(B) &\leq \pi^{-p} \int_{\cE_R} \left| \int_{0}^1 R_k \left(\|z-\theta \bullet \varphi\|^2\right) \left[ \frac{d g}{d \nu}(\varphi) - \frac{d \tilde{g}}{d \nu}(\varphi)\right] d \nu(\varphi) \right| dz \\
& \leq 2 \pi^{-p} \sup_{z \in \cE_R,\varphi \in (0,1)} R_k\left(\|z-\theta\bullet\varphi\|^2\right) Vol(\cE_R).
 \end{align*}
According to the choice $R = (1+a) \sqrt{(1+c) p}$ which implies that $\|z-\theta \bullet \varphi\| \leq (1+2a) \sqrt{(1+c) p}$, and using the volume of $\cE_R$ and Stirling's formula, we obtain
\begin{align*}
(B) &\lesssim \pi^{-p} \frac{\left(e (1+2a)^2 (1+c) p\right)^{k}}{k^k} \frac{\pi^p[ (1+a)^2 (1+c) p]^{p} }{p!} \\
& \lesssim  C_1 ^ p C_2 ^k e^{ - k \log(k) + k \log(p)},\\
\end{align*}
where we used in the last equation $p^p/p! \leq C^p$.
If we define the threshold $k$ in \eqref{eq:expo} such that $k \sim b \ell$ for a sufficiently large $b$, we then obtain for a universal $C$:
$$
(B) = 
\int_{\cE_R } \left| d\PP_{\theta,g} - d\PP_{\theta,\tilde{g}}  \right| (z) \lesssim e^{\ell (C - b   \log(b))}.
$$
In order to bound $(B)$ by $\epsilon/2$, we thus choose $k_{\epsilon} \sim b \ell_{\epsilon}$ for a sufficiently large absolute constant $b$. For such a choice, since $\log \frac{1}{\epsilon} \lesssim \ell_{\epsilon}$ we have found $\tilde{g}$ with a discrete support of cardinal $s_{\epsilon} \sim 2 b \ell_{\epsilon}^2$ points, with $s_\epsilon$ not depending on $g$, such that
$$
d_{TV}(\PP_{f,g},\PP_{f,\tilde{g}}) \leq \epsilon/2.
$$
Now, the first inequality in Proposition \ref{prop:Pf} comes from Proposition \ref{prop:Mktheta}.

The second inequality in Proposition \ref{prop:Pf} is proved from the first one, using the relation $\|\theta\|_{\cH_1} \leq \ell \|\theta\|$ valid for any $f \in \cH^\ell$.
\end{proof}

\begin{proof}[Proof of Lemma \ref{lemma:dVT_sur_f}]
 We follow a straightforward argument: $\PP_{f, g}$ is a mixture model so 
 \[ \PP_{f, g} = \int_{0}^1\PP_{f, \delta_\alpha} d g(\alpha). \]
Thus
 \begin{align*}
  d_{TV}\left( \PP_{f, g}, \PP_{\tilde{f}, g} \right) &= \left\| \int_{0}^1 \left( \PP_{f, \delta_\alpha} - \PP_{\tilde{f}, \delta_\alpha} \right) d g(\alpha)\right\|_{TV} \\
  &\leq  \int_{0}^1  \left\| \PP_{f, \delta_\alpha} - \PP_{\tilde{f}, \delta_\alpha} \right\|_{TV} d g(\alpha) \\
  &= \left\| \PP_{f, \delta_0} - \PP_{\tilde{f}, \delta_0} \right\|_{TV} \leq d_H\left(\PP_{f, \delta_0}, \PP_{\tilde{f}, \delta_0}\right).
 \end{align*}
 Assume now $Y\sim \PP_{f, \delta_0}$, hence from \eqref{eq:model} $d Y = f(x) d x + d W$, with $W$ is a complex standard Brownian motion. If we denote $U$ a random variable $\cN_\C(0,1)$, standard argument using  Girsanov's formula yields
 \begin{align*}
  d_H^2\left(\PP_{f, \delta_0}, \PP_{\tilde{f}, \delta_0}\right) &= 2 \left( 1- \EE_{f, \delta_0} \sqrt{ \frac{ d\PP_{\tilde{f}, \delta_0} }{ d\PP_{f, \delta_0} }(Y)} \right) \\
  &= 2 \left( 1- \EE_{f, \delta_0} \sqrt{ \exp\left( 2\re\langle \tilde{f}-f, d W\rangle - \|\tilde{f}-f\|^2 \right)} \right) \\
  &= 2 \left( 1- \exp\left(\frac{-\|\tilde{f}-f\|^2}{2}\right) \EE_{U}\left[\exp\left(\|\tilde{f}-f\| \re(U)\right)\right] \right)  \\
  &= 2 \left( 1- \exp\left(\frac{-\|\tilde{f}-f\|^2}{4}\right) \right) \leq \frac{\|\tilde{f}-f\|^2}{2}.
 \end{align*}
\end{proof}

\subsection{Link between Kullback-Leibler and Hellinger neighbourhoods}

\begin{proof}[Proof of Proposition \ref{prop:appli_wong_shen}]
This proposition uses a corollary of Rice's formula (see \cite{AW09} for various applications of such formula), stated in Lemma \ref{lemma:rice} and postponed after this proof.

We begin with Girsanov's formula \eqref{eq:Girsanov}. 
Write now $Y=f^{0, -\tau}+W$ where $W$ stands for a complex standard Brownian motion independent of the random shift $\tau$ (whose law is $g^0$). The $L^2$ norm is invariant with any shift thus
\begin{align*}
\frac{d\PP_{f^0,g^0}}{d\PP_{f,g}}(Y) &= \exp \left(\|f\|^2 - \|f^0\|^2\right)
\frac{\int_{0}^1 e^{2 \re\langle f^{0, -\alpha_1}, f^{0, -\tau}+ d W \rangle} d g^0(\alpha_1) }{\int_{0}^1 e^{2\re\langle f^{-\alpha_2}, f^{0, -\tau}+ d W \rangle} d g(\alpha_2)} \\ 
& \leq \exp \left(\|f\|^2 - \|f^0\|^2\right) \exp \left(2 \sup_{\alpha_1,\alpha_2}\re \langle f^{0, -\alpha_1} - f^{-\alpha_2},f^{0, - \tau} \rangle \right) \\
& \exp \left(2 \sup_{\alpha_1,\alpha_2} \re\langle f^{0, -\alpha_1} - f^{-\alpha_2} ,d  W \rangle \right) \\
& \leq e^{(\|f\| + \|f^0\|)^2} e^{Z_1+Z_2},
\end{align*}
where the last inequality is obtained using Cauchy-Schwarz's inequality and the notations
\begin{align*}
 Z_1&:=2\sup_{\alpha_1} \re\langle f^{0, -\alpha_1} , d  W \rangle  = 2 \sup_{\alpha_1} \re\int_{0}^1 \overline{f^0}(s-\alpha_1) d W_s, \\
 Z_2&:=2\sup_{\alpha_2} \re\langle -f^{-\alpha_2} , d  W \rangle = 2 \sup_{\alpha_2} \re\int_{0}^1 -\overline{f}(s-\alpha_2) d W_s.
\end{align*}
We now set $\delta \in (0,1]$  (it will be precisely fixed in the sequel)  and we define the trajectories $\cE_{\delta}$ as
$$
\cE_\delta := \left\{ Y = f^{0, - \tau}+ W \quad | \quad\frac{d\PP_{f^0,g^0}}{d\PP_{f,g}}(Y) \geq e^{1/\delta}\right\}.
$$
Hence, following the definition of $M_{\delta}^2$ of \eqref{eq:condition_moment}, we have
$$
M_{\delta}^2= \EE_{Y \sim \PP_{f^0,g^0}} \left[ \left( \frac{d\PP_{f^0,g^0}}{d\PP_{f,g}}(Y)\right)^{\delta} \1_{Y \in \cE_\delta} \right].
$$
For $\delta$ small enough, ($\delta \leq \frac{1}{2\left(\|f\| + \|f^0\| \right)^2})$:
\begin{eqnarray*}
 M_{\delta}^2 &\leq& e^{\delta (\|f\| + \|f^0\| )^2}  \EE e^{\delta (Z_1+Z_2)} \1_{Z_1+Z_2 \geq \frac{1}{\delta}- (\|f\| + \|f^0\| )^2} \\ &
 \leq & e^{\delta (\|f\| + \|f^0\|)^2} \EE e^{\delta (Z_1+Z_2)} \1_{Z_1+Z_2 \geq \frac{1}{2\delta}} \\ &
 \leq &e^{\delta (\|f\| + \|f^0\|)^2} \EE e^{\delta (Z_1+Z_2)} \1_{e^{\delta (Z_1+Z_2)}  \geq \sqrt{e}}.
\end{eqnarray*}
Integrating by parts the last expectation, the use of Lemma \ref{lemma:rice} yields
\begin{align}\nonumber
M_{\delta}^2 &\leq e^{\delta (\|f\| + \|f^0\|)^2} \int_{\sqrt{e}}^{+ \infty} \PR \left(e^{\delta (Z_1+Z_2)} > u \right) d u \\\nonumber
& =  e^{\delta (\|f\| + \|f^0\|)^2} \int_{\sqrt{e}}^{+ \infty} \left[\PR \left(\frac{Z_1}{2} \geq \frac{\log u}{4\delta} \right) + \PR \left(\frac{Z_2}{2} \geq \frac{\log u}{4\delta} \right)\right] d u \\\label{eq:rice_application}
M_{\delta}^2 & \leq  C(f^0,f)  e^{\delta (\|f\| + \|f^0\|)^2} \int_{\sqrt{e}}^{+ \infty} \left[ e^{- \frac{\log^2(u)}{16\delta^2\|f^0\|^2}}  + e^{- \frac{\log^2(u)}{16\delta^2\|f\|^2}} \right] d u.
\end{align}
Now, we can choose $\delta$ non negative and small enough such that  $M_\delta^2 < \infty$ since for $u \geq \sqrt{e}$, we have

$$
e^{-\frac{\log^2(u)}{16\delta^2 \|f^0\|^2}} \leq e^{- \frac{\log(u)}{32 \delta^2 \|f^0\|^2}} = u^{-1/{32 \delta^2 \|f^0\|^2}},
$$
which is an integrable function as soon as
$\delta^2 < \frac{1}{32 \|f^0\|^2}$, and the same holds with $f$ instead of $f_0$. Note that $M_{\delta}^2$ is uniformly bounded if $f$ is picked into a ball centered at $0$ with radius $2\|f^0\|$.
\end{proof}

We now show that the technical inequality used in \eqref{eq:rice_application} is satisfied.

\begin{lemma}\label{lemma:rice}

 Let $W$ a complex standard Brownian motion and $u$ a complex $1$-periodic map of $\cH_s$. We assume that $u$ is of class $\cC^2$. Then when $t/\|u\| \longrightarrow +\infty$, we have
 $$ \PR\left( \sup_\alpha \re\langle u^{-\alpha},  d W\rangle >t\right) \lesssim \frac{\|u'\|}{2\pi \|u\|} \exp\left(\frac{-t^2}{\|u\|^2}\right).$$
  In particular, if $u \in \cH^\ell$, we have
  $$ \PR\left( \sup_\alpha \re\langle u^{-\alpha}, d  W\rangle > t\right) \lesssim \frac{\ell}{2\pi} \exp\left(\frac{-t^2}{\|u\|^2}\right).$$
 \end{lemma}

\begin{proof}
We define the following process
$$
\forall \alpha \in [0,1] \qquad 
X(\alpha) := \frac{\sqrt{2}\re\left(\int_{0}^1 \overline{u}(s-\alpha) d W_s\right)}{\|u\|}.
$$
$X$ is a Gaussian centered process. Its covariance function is given by
$$
\Gamma(t) = \EE \left[X(0) X(t)\right].
$$
Obviously, one has $\Gamma(0)=1$ and Cauchy-Schwarz's inequality implies that  $\Gamma(s) \leq \Gamma(0)$. Moreover, since $\Gamma$ is $\mathcal{C}^1([0,1])$, we deduce that
$\Gamma'(0) = 0$ and simple computation yields
$$
\Gamma"(0)= \frac{\re\left(\int_{0}^1 u'(s) u"(s) ds\right)}{\|u\|^2} = - \frac{\|u'\|^2}{\|u\|^2}.
$$
Rice's formula (see for instance exercice 4.2, chapter 4 of \cite{AW09}) then yields that when 
$t \longrightarrow +\infty$, we have
$$
\PR \left( \sup_{\alpha} X(\alpha) > t \right) \sim \frac{\|u'\|}{2 \pi \|u\|} e^{- t^2/2}.
$$
This ends the proof of the first inequality. Assume furthermore that $u \in \cH^\ell$,  Parseval's equality implies that $\|u'\| \leq \ell \|u\|$ and we obtain the second inequality. 
\end{proof}

\subsection{Hellinger neighbourhoods}

\begin{proof}[Proof of Proposition \ref{prop:E1}]
Recall first that if $Y$ follows $\PP_{f^0,g^0}$, one shift $\beta$ is randomly sampled according to $g^0$. Conditionally to this shift $\beta$, $Y$ is described trough a white noise model $d Y(x) = f^0(x-\beta) d x + d W(x)$. For any function $F$ of the trajectory $Y$, we will denote $\EE_\beta F(Y)$ the expectation of  $F(Y)$ up to the condition that the shift is equal to $\beta$, and of course one has 
 $$\EE_0[F(Y)] = \int_0^1 \EE_\beta[F(Y)] d g^0(\beta).$$
For each possible value of $\beta \in [0,1]$, we define 
\begin{align*}
D_\beta(\alpha) &:= \exp\left( 2\re\langle f_{ \ell_n}^{0,-\alpha}, f^{0,-\beta}\rangle + 2\re\langle f_{\ell_n}^{0,-\alpha}, d W\rangle - \|f_{\ell_n}^0\|^2\right), \\
 X_\beta(\alpha) &:= \exp\Big(2\re \langle (f^0 - f^0_{ \ell_n})^{-\alpha}, f^{0,-\beta}\rangle \\ &\quad\, + 2\re \langle (f^0-f^0_{\ell_n})^{-\alpha}, d W\rangle - \|f^0-f^{0}_ {\ell_n}\|^2\Big).
\end{align*}
 We can now split the randomness of the Brownian motion into two parts: the first one is spanned by the Fourier frequencies from $-\ell_n$ to $\ell_n$ and the second part is its orthogonal (in $L^2$): $W=W_1+W_2$. Of course,  $W_1$ and $W_2$ are independent.
 
 Moreover, $\langle f_{ \ell_n}^{0,-\alpha}, d  W\rangle = \langle f_{\ell_n}^{0,-\alpha}, d  W_1\rangle$ and $\langle (f^0-f^0_{ \ell_n})^{-\alpha}, d W\rangle = \langle (f^0-f^0_{ \ell_n})^{-\alpha}, d W_2\rangle$. For any fixed $\beta$, $D_\beta(\alpha)$ is measurable with respect to the filtration associated to  $W_1$, and $X_\beta(\alpha)$ is independent of $W_1$.
We thus obtain using Jensen's inequality and this filtration property that
 \begin{align*}
(\tilde{E_1})^2 &= \displaystyle \EE\left[ \log \frac{\int_{0}^1 D_\beta(\alpha) X_\beta(\alpha) d g^0(\alpha)}{\int_{0}^1 D_\beta(\alpha) d g^0(\alpha)} \right] \\
  &\leq \displaystyle \log \int_0^1 \EE^{W_2}_\beta\left[ \EE^{W_1}_\beta\left[ \left. \frac{\int_{0}^1 D_\beta(\alpha) X_\beta(\alpha) d g^0(\alpha)}{\int_0^1 D_\beta(\alpha) d g^0(\alpha)} \right| W_2 \right]\right] d g^0(\beta) \\
  & \leq \displaystyle \log \int_0^1 \EE^{W_2}_\beta\left[ X_\beta(\alpha)  \EE^{W_1}_\beta\left[ \left. \frac{\int_0^1 D_\beta(\alpha) d g^0(\alpha)}{\int_0^1 D_\beta(\alpha) d g^0(\alpha)} \right| W_2 \right]\right] d g^0(\beta) \\
  &\leq \displaystyle\log \int_0^1 \left( \sup_\alpha \EE^{W_2}_\beta \left[ X_\beta(\alpha) \right]\right) d g^0(\beta) .
 \end{align*}
The notation $\EE_{\beta}^{W_1} F(Y)$ (resp. $\EE_{\beta}^{W_2} F(Y)$) used above refers to the expectation of $F(Y)$ with respect to $W_1$ (resp. with respect to $W_2$) with a fixed $\beta$.

Now, one should remark that $X_\beta(\alpha)$ has the same law as $$\exp\left(2\re\langle (f^0 - f^{0}_{\ell_n})^{-\alpha}, f^{0,-\beta}\rangle + U\right),$$ 
where $U\sim\cN_\R\left(-\|f^0-f^{0}_{\ell_n}\|^2, 2\|f^0-f^{0}_{\ell_n}\|^2\right)$, and $\EE\left[e^U\right]=1$. Hence
\begin{eqnarray*}
(\tilde{E_1})^2 & \leq &\log \int_{0}^1 \sup_{\alpha} \exp\left(2\re\langle (f^0 - f^{0}_{\ell_n})^{-\alpha}, f^{0,-\beta}\rangle \right) d g^{0}(\beta)\\ & \leq& \log \sup_{\alpha,\beta} \exp\left(2\re\langle (f^0 - f^{0}_{\ell})^{-\alpha}, f^{0,-\beta}\rangle \right) 
\end{eqnarray*}
 We can now switch $\log$ and $\sup$ since $\log$ is increasing, and we obtain
$$
 (\tilde{E_1})  
 \leq  \sqrt{2 \sup_{\alpha,\beta} \re\langle (f^0 - f^{0}_{\ell_n})^{-\alpha}, f^{0,-\beta}\rangle}. 
 $$
Again, we can use the orthogonal decomposition $f^{0,-\beta} = f^{0,-\beta}_{\ell_n} + f^{0,-\beta}- f^{0,-\beta}_{\ell_n}$ and Cauchy-Schwarz's inequality yields
$ (\tilde{E_1})  \leq \sqrt{2} \|f^{0}- f^{0}_{\ell_n}\|.$

Note that untill now we did not use the hypothesis $f^0\in \cH_s$. It is only needed to get the last inequality in Proposition \ref{prop:E1}.
\end{proof}


To establish Lemma \ref{lemma:mixture_lemma}, we first remind the following useful result.
\begin{lemma}\label{lemma:l1_gauss}
For any any dimension $p$ and any couple of points $(z_1,z_2) \in \C^p$, if $\|z_1-z_2\|$ is the Euclidean distance in $\C^p$, then one has
$$
d_{TV}(\gamma_{z_1}, \gamma_{z_2}) = \frac{1}{2} \|\gamma_{z_1} - \gamma_{z_2}\|_{L_1} = \left[ 2\Phi\left(\frac{\|z_1-z_2\|}{2} \right)-1\right]  \leq \frac{ \|z_1-z_2\|}{ \sqrt{2\pi}},
$$
where $\Phi$ stands for the cumulative distribution function of a real standard Gaussian variable.
\end{lemma}

\begin{proof}[Proof of Lemma \ref{lemma:mixture_lemma}]
Adapting the proof of Lemma 5.1 of \cite{GvdW01}, 
\begin{align*}
\|\PP_{f^0_{\ell_n},\check{g}} -\PP_{f^0_{\ell_n},\tilde{g}}\|_{L_1} &\leq \sum_{j=1}^J \int_{\varphi_j-\eta/2}^{\varphi_j+\eta/2}  \|\gamma_0(.-\theta\bullet\varphi) - \gamma_0(.-\theta\bullet\varphi_j)\|_{L_1} d \check{g}(\varphi) \\
& + 2 \sum_{j=1}^J \left|\check{g}([\varphi_j-\eta/2,\varphi_j+\eta/2]) - p_j\right|.
\end{align*}
Using Lemma \ref{lemma:l1_gauss} ends the proof.
\end{proof}

\begin{proof}[Proof of Proposition \ref{prop:finite_mixture}] \sloppy
The construction used in the proof of Proposition \ref{prop:Pf} provide a mixture $\tilde{\tilde{g}}$ such that $\tilde{\tilde{g}}$ is supported by $\tilde{J}_n := C \ell_n^2$ points (denoted $(\varphi_j)_{j=1\ldots \tilde{J}_n}$) so that $d_H(\PP_{f^0_{\ell_n},g^0} ,\PP_{f^0_{\ell_n},\tilde{\tilde{g}}}) \leq \epsilon_n$. 
Therefore $\tilde{\tilde{g}}= \sum_{j=1}^{\tilde{J}_n} w_j \delta_{\varphi_j}.$ 
As pointed by \cite{GvdW01}, one can slightly modify $\tilde{\tilde{g}}$ so that the support points are separated enough as follows. First, denote
 $(\psi_j)_{j=1 \ldots J_n}$ the subset of $(\varphi_j)_{j=1 \ldots \tilde{J}_n}$ 
which is $\eta_n$-separated with a maximal number of elements. Hence, $J_n \leq \tilde{J}_n$ and up to a permutation, 
one can divide $(\varphi_j)_{j=1 \ldots \tilde{J}_n}$ in two parts:
$ (\varphi_j)_{j=1 \ldots \tilde{J}_n} = (\psi_j)_{j=1 \ldots J_n} \cup  (\varphi_j)_{j=J_n+1 \ldots \tilde{J}_n}$. 
For any $i \in \{J_n+1, \ldots, \tilde{J}_n\}$, we define $\psi_{j(i)}$ as the closest point of $(\psi_j)_{j=1 \ldots J_n}$, the new discrete mixture law is then given by
$$
\tilde{g} = \sum_{j=1}^{J_n}\underbrace{\left(w_j +  \sum_{i > J_n \vert j(i) = j} w_{i} \right)}_{:=\tilde{w}_j} \delta_{\psi_j}.
$$
\fussy
Of course, $\tilde{g}$ as a support which is $\eta_n$-separated. Moreover, we have
\begin{multline*}
 2d_{TV}\left( \PP_{f^0_{\ell_n},\tilde{g}},\PP_{f^0_{\ell_n},\tilde{\tilde{g}}} \right) \\
 \begin{aligned}
  &= \int_{\C^{\ell_n}} \left| \sum_{i=1}^{J_n} \tilde{w}_i \gamma(z-\theta \bullet \psi_i)  - \sum_{i=1}^{\tilde{J}_n} w_i \gamma(z-\theta \bullet \varphi_i) \right| d z \\
  & = \int_{\C^{\ell_n}} \left| \sum_{j=1}^{J_n} (\tilde{w}_j -w_j) \gamma(z-\theta \bullet \psi_j) - \sum_{i > J_n} w_i \gamma(z-\theta \bullet \varphi_i) \right| d z\\
  & = \int_{\C^{\ell_n}} \left| \sum_{j=1}^{J_n} \sum_{i > J_n \vert j(i)=j}  w_i [\gamma(z-\theta \bullet \psi_j) - \gamma(z-\theta \bullet \varphi_i) ] \right| d z.
 \end{aligned}
\end{multline*}
Then, Fubini's theorem yields
$$
d_{TV}\left( \PP_{f^0_{\ell_n},\tilde{g}},\PP_{f^0_{\ell_n},\tilde{\tilde{g}}} \right)  \leq 
\sum_{j=1}^{J_n}
\sum_{i > J_n \vert j(i)=j}  w_i d_{TV} ( \gamma_{\theta \bullet \varphi_i},\gamma_{\theta \bullet \psi_j}),
$$
and we deduce from Lemma \ref{lemma:l1_gauss}  that
$$
d_{TV}\left( \PP_{f^0_{\ell_n},\tilde{g}},\PP_{f^0_{\ell_n},\tilde{\tilde{g}}} \right)  \leq \sqrt{2 \pi}
\sum_{j=1}^{J_n} \sum_{i > J_n \vert j(i)=j}  w_i \|\theta\|_{\cH_1}\eta_n \leq \sqrt{2 \pi} \|\theta\|_{\cH_1}\eta_n.
$$
Now the relations between Hellinger and Total Variation distances \eqref{eq:dvt_dh} yield
\[ d_H(\PP_{f^{0}_{\ell_n},g^0},\PP_{f^{0}_{\ell_n},\tilde{g}})\leq \epsilon_n+d_H(\PP_{f^{0}_{\ell_n},\tilde{g}},\PP_{f^{0}_{\ell_n},\tilde{\tilde{g}}}) \leq \left(1+(8\pi)^{1/4}\|\theta\|_{\cH_1}^{1/2}\right) \epsilon_n. \]
Lemma \ref{lemma:mixture_lemma} permits to conclude.
\end{proof}

\subsection{Checking the conditions of Theorem \ref{theo:posterior}}

\begin{proof}[Proof of Propostion \ref{prop:prior_mino}]

We have seen in the proof of Proposition \ref{prop:appli_wong_shen} that $M_\delta^2$ is uniformly bounded with respect to $\|f\|$ and $\|f^0\|$ for a suitable choice of $\delta$.  
We restrict our study to the elements $f$ such that $\|f\| \leq 2 \|f^0\|$.
We know from Proposition \ref{prop:appli_wong_shen} and Theorem \ref{theo:wong_shen} that as soon as $\tilde{\epsilon}_n \log \frac{1}{\tilde{\epsilon}_n} \leq c \epsilon_n$ with $c$ small enough:
\begin{eqnarray*}
\cV_{\tilde{\epsilon}_n}(\PP_{f^0,g^0},d_H)& := &\left\{ \PP_{f,g} \in \cP \vert d_H(\PP_{f^0,g^0} ,\PP_{f,g} ) \leq \tilde{\epsilon}_n  \,\text{and} \,  \|f\| \leq 2 \|f^0\|\right\}\\ & \subset &\cV_{\epsilon_n}(\PP_{f^0,g^0},d_{KL}).
\end{eqnarray*}

This last condition on $\tilde{\epsilon}_n$ is true as soon as
\begin{equation}\label{eq:def_tilde_epsilon}
\tilde{\epsilon}_n := \tilde{c} \epsilon_n \left( \log \frac{1}{\epsilon_n}\right)^{-1}
\end{equation}
with $\tilde{c}$ small enough. Now, Proposition \ref{prop:Hellinger_neighbourhood} permits to describe a subset of $\cV_{\tilde{\epsilon}_n}(\PP_{f^0,g^0},d_H)$, by the definition of subsets $\cF_{\tilde{\epsilon}_n}$ and $\cG_{\tilde{\epsilon}_n}$ for $f$ and $g$. Choose $\ell_n:=\tilde{\epsilon}_n^{-1/s}$. 

We first bound the prior mass on $\cG_{\tilde{\epsilon}_n}$. This follows from the lower bound given by Lemma \ref{lemma:mino_simplex}. The prior for $g$ is a Dirichlet process with a finite base measure $\alpha$ admitting a continuous positive density on $[0, 1]$. Since $\eta_n$ goes to zero, for $n$ large enough $\alpha(\psi_j - \eta_n/2, \psi_j+\eta_n/2)$ for any $j=1 \ldots J_n$. 
Note that $J_n \lesssim \ell_n^2 = \tilde{\epsilon}_n^{-2/s} \leq  \tilde{\epsilon}_n^{-2}$. Thus, there exists an absolute constant $a\in (0, 1]$ such that the condition $J_n \leq 2 (a\tilde{\epsilon}_n)^{-2}$ is fulfilled, and 
one can find universal constants $C$ and $c$ such that for $n$ large enough
\begin{equation}\label{eq:mino_G}
\Pi_n\left( \cG_{\tilde{\epsilon}_n} \right) \geq \Pi_n\left( \cG_{a\tilde{\epsilon}_n} \right) \geq C e^{-c J_{n} \log \frac{1}{\tilde{\epsilon}_n^2} } \geq C e^{-c \ell_n^2 \log \frac{1}{\tilde{\epsilon}_n}}.
\end{equation}

We next consider the prior mass on $\cF_{\tilde{\epsilon}_n}$. Remark that when $n$ is large enough, any element of  $\cF_{\tilde{\epsilon}_n}$ satisfies $\|f\| \leq 2 \|f^0\|$ and the additional condition on $\|f\|$ in the definition of $\cV_{\tilde{\epsilon}_n}(\PP_{f^0,g^0},d_H)$ is instantaneously fulfilled. Remark that from the construction of our prior on $f$, one has
$$
\Pi_n \left(\cF_{\tilde{\epsilon}_n} \right) \geq \lambda({\ell_n}) \times \pi_{\ell_n} \left(
B\left(\theta^0_{\ell_n}, \tilde{\epsilon}_n^2 \right) \right).
$$
From our assumption on the prior $\lambda$, we have
$\lambda(\ell_n) \geq e^{- c \ell_n^2 \log^\rho \ell_n}$, and the value of the volume of the $(4 \ell_n+2)$-dimensional Euclidean ball of radius $\tilde{\epsilon}_n^2$ implies
$$
\Pi_n \left(\cF_{\tilde{\epsilon}_n} \right) \geq  e^{- c \ell_n^2\log^\rho  \ell_n}
\inf_{u\in B
\left(0, \tilde{\epsilon}_n^2 \right)} \left(\frac{
e^{-\|\theta^0+u\|^2/\xikn^2} 
}{\pi^{2\ell_n+1} \xikn^{2(2\ell_n+1)} 
} \right)
\left( \tilde{\epsilon}_n^2 \right)^{4 \ell_n+2} 
\frac{\pi^{2 \ell_n+1}}{\Gamma(2 \ell_n+2)}.
$$
For $n$ large enough we get
\begin{align}\notag
\Pi_n \left(\cF_{\tilde{\epsilon}_n} \right) &\geq \exp -\left[ c \ell_n^2\log^\rho \ell_n +  \xikn^{-2} \right. \\ \notag &\quad \left.
 + (2\ell_n+1) \left(\log \ell_n + 4 \log(1/\tilde{\epsilon}_n) - \log \xikn^{-2}
 + \cO(1) \right)\right] \\ \label{eq:mino_F}
 &\geq \exp \left[ - (c+o(1))\, \left[ \ell_n^{2} \log^\rho \ell_n  \vee \xikn^{-2} \right]\right]
\end{align}
Gathering \eqref{eq:mino_G} and \eqref{eq:mino_F}, the relations $\ell_n=\tilde{\epsilon}_n^{-1/s}$ and \eqref{eq:def_tilde_epsilon} lead to
\begin{align*}
 \Pi_n \left( \cV_{\epsilon_n}(\PP_{f^0,g^0},d_{KL}) \right) &\geq \Pi_n \left(\cF_{\tilde{\epsilon}_n} \right) \Pi_n\left( \cG_{\tilde{\epsilon}_n} \right) \\
 &\geq \exp \left[ - (c +o(1))\, \left[ \ell_n^{2} \log^\rho \ell_n \vee \xikn^{-2} \right]\right] \\
 &\geq \exp \left[ - (c +o(1))\, \left[\tilde{\epsilon}_n^{-2/s} \log^\rho\left(1/\tilde{\epsilon}_n\right)  \vee \xikn^{-2} \right]\right] \\
 &\geq \exp \left[ - (c +o(1))\, \left[\epsilon_n^{-2/s} \left(\log (1/\epsilon_n)\right)^{\rho+2/s}  \vee \xikn^{-2} \right]\right]
\end{align*}
for constants $c>0$.
%
\end{proof}

\begin{proof}[Proof of Proposition \ref{prop:prior_majo}]
The upper bound on the packing number comes directly from Theorem \ref{theo:recouvrement} since we set $w_n = \sqrt{2k_n+1}$.

Now, to control the prior mass outside the sieve, remark first that owing to the construction of our prior, we have
\begin{equation}\label{eq:complement}
\Pi_n \left(\cP \setminus \cP_{k_n,w_n}\right)  \leq \sum_{|k| \geq k_n} \lambda(k) + \Pr \left( \sum_{|k| \leq k_n} |\theta_k|^2 \geq w_n^2 \right),
\end{equation}
where each $\theta_k$ for $-k_n \leq k \leq k_n$ follows a centered Gaussian law of variance $\xikn^2$.
Now, there exists some constants $c$ and $C$ such that for sufficiently large $n$:
$$
\sum_{|k| \geq k_n} \lambda(k) \leq C \lambda(k_n) \leq e^{- c k_n^2 \log^\rho(k_n)}.$$
Regarding now the second term of the upper bound in \eqref{eq:complement}, we use \eqref{eq:chi_square} to get
\begin{align*}
\Pr \left( \sum_{|k| \leq k_n} |\theta_k|^2 \geq w_n^2 \right) &=
\Pr \left( \sum_{|k| \leq k_n} \left|\frac{\theta_k}{\xikn}\right|^2  \xikn^2\geq w_n^2 \right) \\
& \leq  \PR\left( \chi_{2k_n+1}^2 \geq 2  (2k_n+1) \xikn^{-2} \right) \\
& \leq  \frac{1}{( \xikn^{-2}-1)\sqrt{\pi}} e^{-(2k_n+1)[ \xikn^{-2}-1-\log  \xikn^{-2}]-\log (2k_n+1)/2}.
\end{align*}
Now, using the value of $\xikn$, we obtain
$$
\Pi_n \left(\cP \setminus \cP_{k_n,w_n}\right)  \leq  e^{ 
- c [ k_n^2 \log^\rho(k_n) \wedge k_n \xikn^{-2}]}.
$$
This concludes the proof of the Proposition.
\end{proof}

\subsection[Bound with the Wasserstein metric]{Bound of $d_{TV}(\PP_{f,g},\PP_{f,\tilde{g}})$ with the Wasserstein metric}
\begin{proof}[Proof of Proposition \ref{prop:transport}]
 We use a change of variable, the convexity of $d_{TV}$, and Lemma \ref{lemma:dVT_sur_f} to get
 \begin{align*}
  d_{TV}(\PP_{f,g},\PP_{f,\tilde{g}}) &= \frac{1}{2} \left\| \int_0^1 \PP_{f,\delta_\alpha} d g(\alpha) - \int_0^1 \PP_{f,\delta_\alpha} d \tilde{g}(\alpha) \right\|_{L^1} \\
  &= \frac{1}{2} \left\| \int_0^1 \left(\PP_{f,\delta_{G^{-1}(u)}} - \PP_{f,\delta_{\tilde{G}^{-1}(u)}} \right) d u \right\|_{L^1} \\
  &\leq \int_0^1 d_{TV}\left(\PP_{f^{-G^{-1}(u)},\delta_0}, \PP_{f^{-\tilde{G}^{-1}(u)},\delta_0} \right) d u \\
  &\leq \frac{1}{\sqrt{2}} \int_0^1 \left\|f^{-G^{-1}(u)} - f^{-\tilde{G}^{-1}(u)} \right\| d u.
 \end{align*}
 Then
 \begin{align*}
  \left\|f^{-G^{-1}(u)} - f^{-\tilde{G}^{-1}(u)} \right\| &= \sqrt{\sum_{k\in\Z} |c_k(f)|^2 \left| e^{-\i 2 \pi k G^{-1}(u)} - e^{-\i 2 \pi k \tilde{G}^{-1}(u)} \right|^2 } \\
  &\leq 2 \pi \left| G^{-1}(u) - \tilde{G}^{-1}(u) \right| \sqrt{\sum_{k\in\Z} k^2 |c_k(f)|^2 }.
 \end{align*}
 Therefore we get the first inequality:
 \[ d_{TV}(\PP_{f,g},\PP_{f,\tilde{g}}) \leq \sqrt{2} \pi \|f\|_{\cH_1} \int_0^1 \left| G^{-1}(u) - \tilde{G}^{-1}(u) \right| d u. \]
 
 Now, the second inequality is a classical result: see for instance \cite[Theorem 4]{GS02}. The last inequality is well known too.
\end{proof}

\section{Small ball probability for integrated Brownian bridge}\label{sec:lower_bound}

In the sequel, we still use the notation $p_v$ defined by \eqref{eq:log_gaussian} to refer to the probability distribution which is proportionnal to $e^{v}$.
We detail here how one can obtain a lower bound of the prior weight around any element $g^0$. Since we deal with a log density model, it will be enough to find a lower bound of the weight around $w^0$ if one writes $g^0 \propto e^{w^0}$ according to Lemma \ref{prop:log_density} (which is the Lemma 3.1 of \cite{vdWvZ}).

\begin{lemma}[\cite{vdWvZ}]\label{prop:log_density}
For any real and measurable functions $v$ and $w$ of $[0,1]$, the Hellinger distance between $p_v$ and $p_w$ is bounded by
$$
d_H(p_v,p_w) \leq \|v-w\|_{\infty} e^{\|v-w\|_{\infty}/2}.
$$
\end{lemma}
We now obtain a lower bound of the prior weight on the set $\mathcal{G}_{\epsilon}$ previously defined as:
$$
\mathcal{G}_{\epsilon} := \left\{ g \in \mathfrak{M}_{\nu}([0,1])(2A) : d_{TV}(g,g^0) \leq \epsilon \right\}.
$$
This bound is given by the following Theorem.
\begin{theo}\label{theo:lower_bound_proba}
The prior $q_{\nu,A}$ defined by \eqref{eq:prior_process} and \eqref{eq:log_gaussian} satisfies for $\epsilon$ small enough:
$$
q_{\nu,A} \left( \mathcal{G}_{\epsilon} \right) \geq c e^{- \epsilon^{-\frac{1}{k_\nu+1/2}}},
$$
where $c$ is a constant which does not depend on $\epsilon$.
\end{theo}
\begin{proof}
The proof is divided in 4 steps.
\paragraph{Structure  of the prior}
We denote $w_0 := \log g^0$, which is a $k_\nu$-differentiable function of $[0,1]$, that can be extended to a $1$-periodic element of $\mathcal{C}^{k_\nu}(\mathbb{R})$.  We define $\tilde{q}$ the prior defined by \eqref{eq:prior_process} on such a class of periodic functions (and omit the dependence on $\nu$ and $A$ for sake of simplicity).
The prior $q_{\nu,A}$ is then derived from $\tilde{q}$ through \eqref{eq:log_gaussian}.
We can remark that our situation looks similar to the one described in paragraph 4.1 of \cite{vdWvZ} for integrated brownian motion. Indeed, the log-density $w_0$ should be approximated by some "Brownian bridge started at random" using
$$
w=J_{k_\nu}(B) + \sum_{i=1}^{k_{\nu}} Z_i \psi_i,
$$
where $B$ is a real Brownian bridge between $0$ and $1$. We suppose $B$ built as $B_t = W_t - t W_1$ on the basis of a Brownian motion $W$ on $[0, 1]$. 
 Of course, in the above equation, one can immediately check that $J_{k_\nu}(B)(0)=J_{k_\nu}(B)(1)=0$. Moreover, the relation 
 $J_{k}(f)'
 = J_{k-1}(f)
 - \int_{0}^1 J_{k-1}(f)
 $ and an induction argument yields
$$
\forall j \in \left\{1,\ldots, k_{\nu}\right\} \qquad J_{k_{\nu}}(B)^{(j)}(0)=J_{k_{\nu}}(B)^{(j)}(1).$$
Hence, $J_{k_\nu}(B)$ and its first $k_\nu$ derivatives are $1$-periodic.
Of course, the functions $\psi_i$ are also $1$-periodic and $\mathcal{C}^{\infty}(\mathbb{R})$ and thus our prior $\tilde{q}$ generates admissible functions of $[0,1]$ to approximate $w_0$. We will denote this set of admissible trajectories $\mathcal{C}^{k_{\nu}}_1$ to refer to $1$-periodic functions which are $k_{\nu}$ times differentiable.
 
 \paragraph{Transformation of the Brownian bridge}
 We denote $\mathbb{B}_1$ the separable Banach space of Brownian bridge trajectories between $0$ and $1$ and $\mathbb{B}_2 = \mathbb{R}^{k_{\nu}+1}$.
It is possible to check that the map
 $$
 T: (B,Z_0,\ldots,Z_{k_{\nu}}) \longmapsto J_{k_{\nu}}(B)+\sum_{i=0}^{k_{\nu}} Z_i  \psi_i $$
 is \textit{injective} from the Banach space $\mathbb{B}=\mathbb{B}_1 \times \mathbb{B}_2$ to the set $\underline{\mathbb{B}}:=T(\mathbb{B})$.
  More precisely, an recursive argument shows that each map $J_k(B)$ may be decomposed as
 \begin{equation} \label{eq:JkB}
  \forall t \in [0,1] \qquad J_{k}(B)(t)=I_k(W)(t) + \sum_{i=1}^{k+1} c_{i,k}(W) t^i,
 \end{equation}
 where $ c_{i,k}(W)$ are explicit linear functionals that depend on $W_1$ and on the collection 
 $\big(\int_0^1 (1-t)^{k-j} W_t d t\big)_{1\leq j\leq k}$ 
 (and not on $t$), and $I_k$ is the operator used in \cite{vdWvZ} defined as $I_1(f)=\int_{0}^t f$ and $I_k = I_{1} \circ I_{k-1}$ for $k \geq 2$. Hence,
 \begin{multline}\label{eq:jk_expand}
 \forall t \in [0,1], \quad 
 T(B,Z_1,\ldots,Z_{k_{\nu}})^{(k)}(t) \\
 = W_t +  c_{k,k}(W)k! +  c_{k+1,k}(W) (k+1)! t + \sum_{i=0}^{k_{\nu}} Z_i \psi_i^{(k)}(t)
 \end{multline}
According to the Brownian bridge representation \textit{via} its Karhunen-Loeve expansion (as sinus series), and since each $\psi_{i}^{(k)}$ possesses a non vanishing cosinus term: $t \mapsto \cos(2 \pi i t)$, we then deduce that
$$
 T(B^1,Z^1_1,\ldots,Z^1_{k_{\nu}}) =  T(B^2,Z^2_1,\ldots,Z^2_{k_{\nu}})
$$
necessarily implies that $Z^1_i=Z^2_i$ for  $i \in \{0, \ldots, k_{\nu}\}$, and next that $W^1=W^2$ and $B^1=B^2$.

 Thus, it is possible to apply Lemma 7.1 of \cite{vdWvZ2} to deduce that the Reproducing Kernel Hilbert Space (shortened as RKHS in the sequel) associated to the Gaussian process \eqref{eq:prior_process} in $\underline{\mathbb{B}}$ is $\underline{\mathbb{H}} := T \mathbb{H}$ where $\mathbb{H}$ is the RKHS derived in the simplest space $\mathbb{B}=\mathbb{B}_1 \times \mathbb{B}_2$. Moreover, the map $T$ is an isometry from $\mathbb{H}$ to $\underline{\mathbb{H}}$ for the RKHS-norms. At last, since the sets $\mathbb{B}_1$ and $\mathbb{B}_2$ are independent, the RKHS $\mathbb{H}$ may be described as
$$
\mathbb{H} := \left\{ (f,z) \in AC([0,1]) \times \mathbb{R}^{k_{\nu}+1} : f(0)=f(1)=0, \int_{0}^1 f'^2< \infty \right\},
$$ 
where $AC([0,1])$ is the set of absolutely continuous functions on $[0,1]$, $\mathbb{H}$ is endowed with the following inner product:
 $$
 \langle (f_1,z^1),(f_2,z^2)\rangle_{\mathbb{H}} := \int_0^1 f'_1 f'_2 + \langle z^1,z^2 \rangle_{\mathbb{R}^{k_{\nu}+1}}.
 $$

\paragraph{Extremal derivatives} 
We  study the influence of the process $$b:=\sum_{i=0}^{k_{\nu}} Z_i \psi_i$$ and are looking for realizations of $(Z_i)_{i}$ that suitably matches arbitrarily values $w_0^{(j)}(0) = w_0^{(j)}(1)$. In this view, simple computations yield that for any integer $p$:
$$
\psi^{(2p)}_k(t) = (-1)^{p} (2\pi k)^{2p} \psi_k(t),
$$
and
$$
\psi^{(2p+1)}_k(t) = (-1)^{p} (2\pi k)^{2p+1}[- \sin (2 \pi k t) + \cos (2 \pi k t)].
$$
Hence, the matching of $w_0^{(j)}(0)$ by $b^{(j)}(0)$ is quantified by
$$w_0^{(j)}(0) - b^{(j)}(0)= w_0^{(j)}(0) - \sum_{k=0}^{k_{\nu}} (-1)^{\lfloor j/2  \rfloor}  (2 \pi k)^{j} Z_k.
$$
If one denotes  $\alpha_k := 2 \pi k$, the vector of derivatives as $d_0:=(w_0^{(j)}(0))_{j=0 \ldots k_{\nu}}$, $Z=(Z_0,\ldots, Z_{k_{\nu}+1})$ and the squared matrix of size $(k_{\nu}+1) \times (k_{\nu}+1)$:
$$
A_0 := \left( 
\begin{matrix}
 1  & 1 & \hdots &  1 \\ \alpha_1 & \alpha_2 & \hdots & \alpha_{k_\nu} \\
 -\alpha_1^2  & - \alpha_2^2 & \ldots & - \alpha_{k_\nu}^2 \\
 -\alpha_1^3  & - \alpha_2^3 & \ldots & - \alpha_{k_\nu}^3 \\ 
\alpha_1^{4} &  \alpha_2^4 & \ldots &  \alpha_{k_\nu}^4 \\ 
 \vdots & & & \\
\end{matrix}
\right),
$$
then we are looking for values of $Z$ such that $d_0 = A_0 Z$.
The matrix $A_0$ is invertible since it may be linked with the Vandermonde matrix. 

We can now establish that the support of the prior (adherence of $\underline{\mathbb{B}}$) is exactly $\mathcal{C}_1^{k_{\nu}}$. Indeed, the support of the transformed Brownian bridge $J_k(B)$ is included in the set of $1$-periodic functions $\mathcal{C}_1^{k_{\nu}}$ which possesses at the most $k+1$ constraints on the values of their $k_{\nu}+1$ first derivatives at the point $0$. These constraints are given by the  coefficients $(c_{i,k_\nu})_{i=0 \ldots k_{\nu}}$ in \eqref{eq:jk_expand}. From the invertibility of the matrix $A_0$, it is possible to match \textit{any} term $w_0^{(j)}(0), 0 \leq j \leq k_{\nu}$ with the additional process $b$ \citep[see][section 10]{vdWvZ2}.

\paragraph{Small ball probability estimates}
We now turn into the core of the proof of the Theorem.
Since the Total Variation distance is bounded from above by the Hellinger distance, an immediate application of Lemma \ref{prop:log_density} shows that it is sufficient to find a lower bound of the $\tilde{q}(\tilde{\mathcal{G}}_{\epsilon})$ where
$$
\tilde{\mathcal{G}}_{\epsilon}:= \left\{w \in \mathcal{C}_{1}^{k_\nu}([0,1]) : \|w-w_0\|_{\infty} \leq \epsilon \right\}.
$$ 
Following the argument of \cite{KWL} on \textit{shifted} Gaussian ball, we have
$$
\log \left( 
\tilde{q}\left(\tilde{\mathcal{G}}_{\epsilon}\right) \right) \geq - \inf_{h \in \underline{\mathbb{H}} : \|h-w_0\|_{\infty} \leq \epsilon} \|h\|_{\underline{\mathbb{H}}}^2 - \log \tilde{q}\left( \|J_k(B)+b\|_{\infty} \leq \epsilon \right).
$$
From the isometry $T$ from $\mathbb{H}$ to $\underline{\mathbb{H}}$, we can write that the approximation term 
$\inf_{h \in \underline{\mathbb{H}} : \|h-w_0\|_{\infty} \leq \epsilon} \|h\|_{\underline{\mathbb{H}}}^2$ is of the same order as the approximation term that we can derive in $\mathbb{H}$, and the arguments of Theorem 4.1 in \cite{vdWvZ} can be applied here to get
$$
\inf_{h \in \underline{\mathbb{H}} : \|h-w_0\|_{\infty} \leq \epsilon} \|h\|_{\underline{\mathbb{H}}}^2 \lesssim \epsilon^{-\frac{1}{k_{\nu}+1/2}}.
$$

It reminds to obtain a lower bound of the small ball probability of the \textit{centered} Gaussian ball. Note that $b$ and $J_{k_{\nu}}$ are independent Gaussian processes. We have somewhat trivially that
$
\log \left( \frac{1}{\epsilon} \right) \lesssim \log \mathbb{P} \left( \|b\|_{\infty} \leq \epsilon \right).
$
Thus, the main difficulty relies on the lower bound of
$$
\phi_0(\epsilon):=\log \mathbb{P} \left( \|J_k(B)\|_{\infty} \leq \epsilon \right).
$$
Going back to \eqref{eq:JkB}, we see that $J_k(B)$ can be decomposed into two nonindependent Gaussian processes: $I_k(W)$ and a polynomial $\sum_{i=1}^{k+1} c_{i,k}(W) t^i$ which is a linear functional of $W_1$ and of the collection 
$\big(\int_0^1 (1-t)^{k-j} W_t d t\big)_{1\leq j\leq k}$. 
Therefore 
\[\log \left( \frac{1}{\epsilon} \right) \lesssim \log \mathbb{P} \left( \left\|
J_k(B) - I_k(W)
\right\|_{\infty} \leq \epsilon \right). \]
Now, applying Theorems 3.4 and 3.7 of \cite{Li_Shao} yields
\[ \log \mathbb{P} \left(\|J_{k_\nu}(B)\|_{\infty} \leq \epsilon  \right) \sim \log \mathbb{P} \left(\|I_{k_\nu}(W)\|_{\infty} \leq \epsilon  \right) \geq -\epsilon^{-\frac{1}{k_{\nu}+1/2}}, \]
which is of the same order as the approximation term. Gathering now our lower bound on shifted Gaussian ball and the term above ends the proof of the Theorem.
\end{proof}

\section{Equivalents on Modified Bessel functions}\label{sec:bessel_appendix}


\begin{lemma}\label{lemma:equi_bessel}  For any $n \in \Z$ and $a>0$, define
$$
A_n(a) := \int_{0}^{2\pi}e^{a \cos(u)} \cos (nu) d u.
$$ 
Then, the following equivalent holds:
$$
\forall a \in [0,\sqrt{n}] \qquad A_n(a) \sim \frac{2 \pi}{n!} \left(\frac{a}{2}\right)^n 
\left( 1+\mathcal{O}\left(\frac{a}{n} \right)\right).
$$

\end{lemma}
\begin{proof}
This equivalent is related to the modified Bessel functions (see \cite{AS64} for classical equivalents on Bessel functions and \cite{LL10} for standard results on continuous time random walks). More precisely, $I_m(a)$ is defined as
$$
\forall m \in \N, \forall a >0 \qquad 
I_m(a):= \sum_{k\geq0}\frac{1}{k!(k+m)!} \left(\frac{a}{2}\right)^{2k+m},
$$
and we have (see for instance \cite{AS64}) $$
I_0(a) +2 \sum_{m=1}^{+\infty}I_m(a) \cos (m u) = e^{a \cos u}.
$$
Hence, we easily deduce that $A_n(a) = 2 \pi I_n(a)$.
For small $a$, it is possible to use standard results on modified Bessel functions. Equation (9.7.7) of \cite{AS64}, p. 378. yields
\begin{equation}\label{eq:equi2}
\forall a \in [0,\sqrt{n}] \qquad 
I_n(a) \sim \frac{1}{n!} \left(\frac{a}{2}\right)^n \left( 1+\mathcal{O}\left( \frac{a}{n}\right) \right) .
\end{equation}
\end{proof}

This integral is strongly related to the density of continuous time random walk if one remark that if $B_n(a)=e^{-a} A_n(a)/(2\pi)$, one has
$
B_n(0)=0, \forall n \neq 0$ and $B_0(0)=1$ and at last
$$
B'_n(a) = \frac{B'_n(a-1)+B'_n(a+1)}{2} - B_n(a).
$$
Hence, $B_n(a)$ is the probability of a continuous time random walk to be in place $n \in \Z$  at time $a$. In this way, we  get some asymptotic equivalents of $B_n(a)$ (and so of $A_n(a)$): from the Brownian approximation of the CTRW , we should suspect that for $a$ large enough
\begin{equation}\label{eq:equi_ctrw}
B_n(a) \sim \frac{1}{\sqrt{2 \pi a}} e^{-n^2/(2a)}, \forall a \gg n^2.
\end{equation}
 Moreover, from \cite{AS64}, we know that 
\begin{equation}\label{eq:equi1}
I_n(a) \sim \frac{e^{a}}{\sqrt{2 \pi a}}, \quad \text{as soon as} \quad a \geq 2n,
\end{equation}
and this equivalent is sharp when $a$ is large enough: from equation (9.7.1) p. 377 of \cite{AS64}, we know that
$$
\forall a \geq 4 n^2 \qquad 
I_n(a) \geq \frac{1}{2} \times \frac{e^{a}}{\sqrt{2 \pi a}}.
$$
We remark that \eqref{eq:equi1} yields the heuristic equivalent suspected in \eqref{eq:equi_ctrw}: $B_n(a) = e^{-a} I_n(a) \sim \frac{1}{\sqrt{2 \pi a}}$,  although \eqref{eq:equi2} provides a quite different information for smaller $a$. 
We do not have purchase more investigation on this asymptotic since we will see that indeed, \eqref{eq:equi2} is much more larger than \eqref{eq:equi1}. 

For $a \in [\sqrt{n},2n]$, we do not have found any satisfactory equivalent of modified Bessel functions. Formula of \cite{AS64} is still tractable but yields some different equivalent which is not  "uniform enough" since we need to integrate this equivalent for our bayesian analysis.
 This is not  so important since we can see for our range of application that the most important weight belongs to the smaller values of $a$. 
 \section*{Acknowledgements} S. G. is indebted to Jean-Marc Aza\"is, Patrick Cattiaux and Laurent Miclo for stimulating discussions related to some technical parts of this work. Authors also thank J\'er\'emie Bigot, Isma\"el Castillo, Xavier Gendre, Judith Rousseau and Alain Trouv\'e for enlightening exchanges.

\bibliographystyle{alpha}
\bibliography{paper}

\end{document}